\documentclass[10pt,reqno]{amsart}
\usepackage{ifthen, amsmath,amssymb, graphicx, mathrsfs, amsfonts}
\usepackage{esint}

\newcommand{\R}{\mathbb{R}}
\newcommand{\C}{\mathbb{C}}
\newcommand{\Q}{\mathbf{Q}}

\newcommand{\Real}{\textrm{\rm Re}\,}
\newcommand{\Imag}{\textrm{\rm Im}\,}
\newcommand{\diag}{\textrm{\rm diag}\,}
\newcommand{\dx}{\textrm{\rm dx}}
\newcommand{\dt}{\textrm{\rm dt}}
\newcommand{\A}{\mathbf{A}}
\newcommand{\At}{\mathbf{A}^\tau}
\newcommand{\ess}{\sigma_{\textrm{ess}}}
\newcommand{\ptsp}{\sigma_{\textrm{pt}}}
\newcommand{\cT}{\mathcal{T}}
\newcommand{\cL}{\mathcal{L}}

\newcommand{\sgn}{\mathrm{sgn}\,}
\newcommand{\ind}{\text{ind}\,}

\provideboolean{shownotes}
\setboolean{shownotes}{true} 
\newcommand{\margnote}[1]{
	\ifthenelse{\boolean{shownotes}}
	{\marginpar{\raggedright\tiny\texttt{#1}}}
	{}
	}

\newtheorem{theorem}{Theorem}[section]
\newtheorem{lemma}[theorem]{Lemma}
\newtheorem{corollary}[theorem]{Corollary}
\newtheorem{proposition}[theorem]{Proposition}
\newtheorem{definition}[theorem]{Definition}
\theoremstyle{remark}
\newtheorem{remark}[theorem]{Remark}

\numberwithin{equation}{section}

\begin{document}

\title[Traveling waves for the Allen-Cahn relaxation model]{Analytical and numerical investigation of traveling waves for the Allen--Cahn model with relaxation}



\author[C. Lattanzio]{Corrado Lattanzio}
\address[Corrado Lattanzio]{Dipartimento di Ingegneria e Scienze dell'Informazione e Matematica, Universit\`a degli Studi dell'Aquila (Italy)}
	\email{corrado@univaq.it} 
 
\author[C. Mascia]{Corrado Mascia}
\address[Corrado Mascia]{Dipartimento di Matematica,
	Sapienza Universit\`a di Roma (Italy)}
	\email{mascia@mat.uniroma1.it}
 
\author[R.G. Plaza]{Ram\'on G. Plaza}
\address[Ram\'on G. Plaza]{Departamento de Matem\'aticas y Mec\'anica,
	IIMAS, Universidad Nacional Aut\'onoma de M\'exico (Mexico)}
	\email{plaza@mym.iimas.unam.mx}
\thanks{RGP is grateful to DISIM, University of L'Aquila, and to the MathMods Program (Erasmus Mundus)
for their hospitality and financial support in academic visits during the Falls of 2012 and 2013,
when this research was carried out.}

\author[C. Simeoni]{Chiara Simeoni}
\address[Chiara Simeoni]{Laboratoire de Math\'ematiques J.A. Dieudonn\'e, 
	Universit\'e Nice Sophia-Antipolis (France)}
	\email{simeoni@unice.fr}

\thanks{This work was partially supported by CONACyT (Mexico) and MIUR (Italy),
through the MAE Program for Bilateral Research, grant no.\ 146529
and by the Italian Project FIRB 2012 ``Dispersive dynamics: Fourier Analysis and Variational Methods''.}


\maketitle

\begin{abstract} 
A modification of the parabolic Allen--Cahn equation, determined by the substitution
of Fick's diffusion law with a relaxation relation of Ca\-tta\-neo-Maxwell type, is considered.
The analysis concentrates on traveling fronts connecting the two stable states of the model,
investigating both the aspects of existence and stability.
The main contribution is the proof of the nonlinear stability of the wave, as a consequence
of detailed spectral and linearized analyses.
In addition, numerical studies are performed in order to determine the propagation
speed, to compare it to the speed for the parabolic case, and to explore
the dynamics of large perturbations of the front.
\end{abstract}

\tableofcontents

\section{Introduction}\label{sect:introduction}

The main topic of this paper is to investigate the dynamical behavior of a scalar variable $u$ subject
to a transport mechanism of hyperbolic type coupled with a reaction process, driving the unknown
$u$ toward one among two different competing stable states. Restricting the attention to a one-dimensional environment and denoting by $v$ the flux of $u$,
standard balance of mass provides the relation
\begin{equation}\label{balance}
	 u_t- v_x = f(u),
\end{equation}
dictating that the quantity $u$ diffuses with flux $v$ and grows/decays according to the
choice of the function $f$.
The simplest structure for the reaction term $f$ giving raise to two competing stable states is
referred to as a {\it bistable form}, meaning that $f$ is smooth and such that for some $\alpha\in(0,1)$,
\begin{equation}\label{bistablef}
	\begin{aligned}
	&f(0)=f(\alpha)=f(1)=0,
		&\qquad &f'(0), f'(1)<0,\quad f'(\alpha)>0,\\
	&f(u)>0\textrm{ in }(-\infty,0)\cup(\alpha,1),
		&\qquad &f(u)<0\textrm{ in }(0,\alpha)\cup(1,+\infty).
	\end{aligned}
\end{equation}
Equivalently, the function $f$ can be considered as the derivative of a double-well
potential with wells centered at $u=0$ and $u=1$. 
A typical reaction function satisfying \eqref{bistablef} which is often found in
the literature is the cubic polynomial
\begin{equation}\label{cubicf}
	f(u)=\kappa\,u(1-u)(u-\alpha),\qquad\kappa>0,\quad \alpha\in(0,1).
\end{equation}
Reaction functions of bistable type arise in many models of natural phenomena,
such as kinetics of biomolecular reactions \cite{Murr02,Mikh94},
nerve conduction \cite{Lieb67,McKe70}, and electrothermal instability \cite{IDRWZB95},
among others.

To complete the model, an additional relation has to be coupled with \eqref{balance}.
The standard approach is based on the use of Fick's diffusion law, which consists in the equality
$v=u_x$, so that one ends up with the semilinear parabolic equation
\begin{equation}\label{parAC}
	u_t = u_{xx} + f(u).
\end{equation}
Equation \eqref{parAC} has appeared in many different contexts and the nomenclature is not uniform.
It is known as the bistable reaction-diffusion equation \cite{FifeMcLe77}, the Nagumo equation in
neurophysiological modeling \cite{McKe70,NAY62}, the real Ginzburg--Landau for the variational
description of phase transitions \cite{MelSch04}, and the Chafee-Infante equation \cite{ChIn74},
among others.
In tribute to S. M. Allen and J. W. Cahn, who proposed it in connection with the motion of boundaries
between phases in alloys \cite{AlleCahn79}, we call it the {\it (parabolic) Allen--Cahn equation}.

Equation \eqref{parAC} undergoes the same criticisms received by the standard linear
diffusion equation, mainly concerning the unphysical infinite speed of propagation
of disturbances (see discussion in \cite{Holm93}).
Thus, following the modification proposed by Cattaneo \cite{Catt49} (see also
Maxwell \cite{Maxw1867}) for the heat equation, it is meaningful to couple \eqref{balance}
with an equation stating that the flux $v$ relaxes toward $u_x$ in
a time-scale $\tau>0$, namely
\begin{equation*}
	\tau v_t + v = u_x,
\end{equation*}
usually named {\it Maxwell--Cattaneo transfer law} (for a complete discussion on its role
and significance in the modeling of heat conduction, see \cite{JosePrez89,JosePrez90}).
With this choice, the couple density/flux $(u,v)$ solves the hyperbolic system 
\begin{equation}\label{relAC}
	u_t = v_x + f(u),\qquad
	\tau v_t = u_x - v.
\end{equation}
We are interested in studying the dynamics of solutions to system
\eqref{relAC}, which we refer to as the {\it Allen--Cahn model with relaxation}.
The corresponding Cauchy problem is determined by initial conditions
\begin{equation}\label{relACinitial}
	u(x,0)=u_0(x),\qquad v(x,0)=v_0(x), \qquad \, x\in\R.
\end{equation}

It is to be observed that we can eliminate the variable $v$ by a procedure known in some references
as \textit{Kac's trick} (cf. \cite{Hill97,HaMu01,Kac74}):
differentiate the first equation in \eqref{relAC} with respect to $t$, and the second with
respect to $x$, to obtain the following scalar second-order equation for the variable $u$,
\begin{equation}\label{relAConef}
	\tau u_{tt} + (1-\tau f'(u)) u_t = u_{xx} + f(u),
\end{equation}
which we call the {\it one-field equation} determined by \eqref{relAC}.
The initial conditions corresponding to \eqref{relACinitial} read
\begin{equation*}
	u(x,0)=u_0(x),\qquad u_t(x,0)=f(u_0(x))+v_0'(x),\qquad  \, x\in\R.
\end{equation*}
Notice that equation \eqref{relAConef} formally reduces to \eqref{parAC} in the singular limit $\tau\to 0^+$.

We also observe that if we include a diffusion coefficient $\nu>0$,
\begin{equation*}
	u_t = v_x + f(u),\qquad
	\tau  v_t = \nu  u_x -v,
\end{equation*}
then last system can be reduced to the form of \eqref{relAC} under the rescaling $x\mapsto x/\sqrt{\nu}$ and $v\mapsto v/\sqrt{\nu}$. Thus, we can consider the case $\nu = 1$ without loss of generality.

The hyperbolic system \eqref{relAC} can be interpreted as a model for a 
reaction-diffusion process. 
An intriguing issue is to compare the properties of the usual parabolic reaction-diffusion
equation \eqref{parAC} with the ones of its hyperbolic counterpart \eqref{relAC}.
This paper pertains to one of the main hallmarks of the Allen--Cahn equation: the presence
of stable heterogeneous structures, describing the interaction between the two stable states.
Specifically, we examine traveling wave solutions to \eqref{relAC}, i.e. special solutions of the form
\begin{equation}\label{twave}
	(u,v)(x,t)=(U,V)(\xi),\qquad\qquad \xi=x-ct,\quad c\in\R,
\end{equation}
with asymptotic conditions
\begin{equation}\label{asymptoticstates}
	(U,V)(\pm\infty)=(U_\pm,0),\qquad
	\textrm{where}\quad U_-:=0,\quad U_+:=1,
\end{equation}
with the aim of investigating their existence and stability from both an analytical
and numerical point of view.

Existence of traveling waves for systems of the form \eqref{relAC} with monostable/logistic
reaction terms has been widely investigated \cite{Hade88,Hade94,GildKers13}.
The situation for the bistable case is less explored, even if it is more or less known that,
under reasonable assumptions, there exist traveling fronts with a uniquely determined propagation speed.
For the sake of completeness, in \S \ref{sect:existence} we give a self-contained
proof of the following existence result, which provides also some properties crucial to the stability analysis.

\begin{theorem}\label{thm:existence}
Let $f$ be such that \eqref{bistablef} holds and let $\tau$ satisfy
\begin{equation}\label{smalltau2}
	0 < \tau < \tau_m:=1/\sup\limits_{u\in [0,1]} |f'(u)|.
\end{equation}
Then, there exists a unique value $c_\ast\in \R$
for which system \eqref{relAC} possesses a traveling wave propagating with speed $c_\ast$ and connecting the 
state $(0,0)$ to $(1,0)$.
Moreover,\par
(a) the function $U$ is monotone increasing;\par
(b) both components $U$ and $V$ are positive and 
	converge to their asymptotic states exponentially fast; and,\par
(c) the speed $c_\ast$ belongs to the interval $(-1/\sqrt{\tau},1/\sqrt{\tau})$ and depends continuously with respect to $\tau\in(0,\tau_m)$. Moreover, it converges to the speed of the (parabolic) Allen--Cahn equation as $\tau\to 0^+$.
\end{theorem}

The smallness assumption \eqref{smalltau2} on the relaxation parameter $\tau$ is not sharp
and arises as a consequence of the specific choices of the natural variables $u$ and $v$, which
generates some obstruction in the course of the proof. 
A different choice of unknowns could provide a more general result allowing a weaker requirement on $\tau$.

We observe, however, that condition \eqref{smalltau2} is tantamount to the positivity of the damping
coefficient in equation \eqref{relAConef}, a condition which is usually imposed in order to ensure that
the solution is positive
(hyperbolic equations may have negative solutions even with positive initial conditions; cf. \cite{Hade99}).
The latter is an important feature for a density solution, for example.
Furthermore, although condition \eqref{smalltau2} may seem a pure mathematical assumption,
it relates the relaxation time $\tau$ with the typical time scale
$\tau_{\mathrm{reac}} = \inf |u/f(u)| \sim 1/\sup |f'(u)|$ associated to the reaction. 

Passing to the stability issue, as for evolution problems on the whole real line defined by autonomous
partial differential equations, invariance with respect to translations determines that any traveling
wave belongs to a manifold of solutions of the same type with dimension at least equal to one.
Thus, small perturbations of a given front are not expected to decay to the front itself, but to the manifold
generated by the traveling wave and, at best, to a specific element of such set. 
Such property, called {\it orbital stability}, holds for the present case.
More precisely, we establish the following

\begin{theorem}\label{thm:nonlinstab}
Let $f\in C^3$ be such that \eqref{bistablef} holds and $\tau \in [0,\tau_m)$.
Let $(U,V)$ be a traveling wave of \eqref{relAC} satisfying \eqref{asymptoticstates}
with speed $c_\ast$.
Then, there exists $\varepsilon>0$ such that, for any initial data satisfying $(u_0,v_0) - (U,V) \in H^1(\R;\R^2)$ with
$|(u_0,v_0)-(U,V)|_{{}_{H^1}}<\varepsilon$, the solution $(u,v)$ to the Cauchy problem
\eqref{relAC}--\eqref{relACinitial} satisfies for any $t>0$
\begin{equation*}
	|(u,v)(\cdot,t)-(U,V)(\cdot-c_\ast t+\delta)|_{{}_{H^1}}
		\leq C|(u_0,v_0)-(U,V)|_{{}_{H^1}}\,e^{-\theta\,t}
\end{equation*}
for some shift $\delta\in\R$ and constants $C,\theta>0$.
\end{theorem}

This statement is the final outcome of some intermediate fundamental steps
which are not detailed at this stage for the sake of simplicity of presentation.
Actually, Theorem \ref{thm:nonlinstab} is proved by following a well-estabilished approach
based on linearization, spectral analysis, linear and nonlinear stability.
All of these steps are developed in a complete way, altogether providing a sound rigorous basis 
to the stability of propagation fronts for the Allen--Cahn model with relaxation.
At this point, we find it suitable to mention recent work by Rottmann-Matthes \cite{Rot11,Rot12},
who proves that spectral stability implies orbital stability of traveling waves for a large class of
hyperbolic systems (which includes the Allen-Cahn model with relaxation \eqref{relAC}), and provides
numerical evidence of spectral stability with spectral gap for the particular case of system \eqref{relAC}
by approximating the spectrum of the linearized operator around the wave with periodic boundary conditions, 
which reproduce the point and essential spectrum on the unbounded domain accurately (cf. \cite{SS00}).
These numerical observations, however, do not constitute a proof of stability.
Our contribution is a self-contained study of the dynamics of traveling fronts for the particular
model \eqref{relAC} (which warrants note because of its importance in the theory of hyperbolic diffusion),
as well as a complete analysis of their stability.

Both the existence and the stability analyses are complemented with a numerical study confirming the
theoretical results and providing additional relevant information on what should be expected
beyond the boundaries of the proved statements. In the first part, we numerically determine the values
of the propagation speeds and discuss their relation with the corresponding value in the case
of the standard Allen--Cahn equation.
Interestingly enough, the model with relaxation exhibits in some regimes fronts
that are faster with respect to their corresponding parabolic ones.
At the end of the paper, we consider some numerical simulations relative to perturbations
of the traveling front, restricting the attention to the case of standard and perturbed Riemann problems.
The outcome is a strong numerical evidence that the domain of attractivity of the wave is wider than
what described by the stability result of Theorem \ref{thm:nonlinstab}
(see Figure \ref{fig:Random2tot}, detailed description in Section \ref{sect:numerics}).
The algorithm used in this part is based on a reformulation of \eqref{relAC}, discretized by using a
standard finite-difference method with upwinding of the space derivatives.\\

\begin{figure}\centering
\resizebox{0.8\hsize}{!}{\includegraphics*{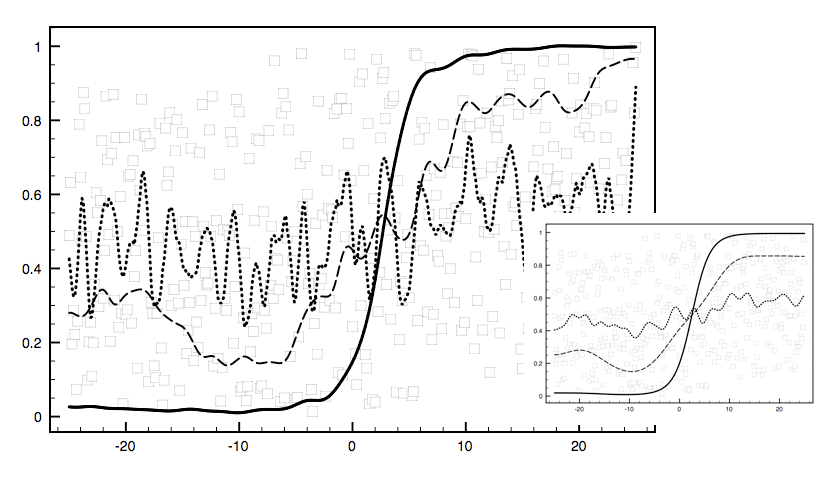}}
\caption{Random initial datum in $(-\ell,\ell)$, $\ell=25$ (squares).
Solution profiles for the Allen--Cahn equation with relaxation at time $t=0.5$ (dot),
$t=7.5$ (dash), $t=15$ (continuous).
For comparison, in the small window, the solution to the parabolic Allen--Cahn equation.
For details, see Section \ref{sect:numerics}.
}\label{fig:Random2tot}
\end{figure}

\noindent \textbf{Plan of the paper.}
This work is divided into four more sections. Section \ref{sect:existence} deals with the existence of the propagation fronts
together with the estabilishment of their main properties, which are essential for completing the stability arguments. 
It contains the detailed proof of Theorem \ref{thm:existence}, and it is based on a phase-plane analysis
that takes advantage of a specific monotonicity property of the system under consideration.
The content of Section \ref{sect:linspecstab} is the spectral analysis of the linearized operator around the front.
The main result is Theorem \ref{thm:spectrum}, estabilishing the spectral stability of the wave.
The proof is based on a perturbation argument at $\tau=0$, combined with a continuation 
procedure to show that the same spectral structure holds all along the interval $(0,\tau_m)$.
The first part of Section \ref{sect:linnonlinstab} deals with the linear stability property.
This is consequence of spectral stability because of the hyperbolic
nature of system \eqref{relAC} together with an additional resolvent estimate controling
the behavior at large frequencies. 
With such tools at hand, and following the classical ideas of Sattinger developed for the parabolic setting 
\cite{Satt76}, it is possible to prove the nonlinear stability theorem,
taking advantage of the presence of a spectral gap separating the zero eigenvalue
from the rest of the spectrum. That is the content of the final part of the Section.
Finally, Section \ref{sect:numerics} is entirely devoted to the numerical approximation of \eqref{relAC}.
The principal part of the system is diagonalized and
a finite-difference upwind approximation is considered.
First, we analyze the Riemann problem connecting the two asymptotic states of the front.
Such choice is used to compare the asymptotic speed of propagation with the values
determined in Section \ref{sect:existence} and to show the numerical evidence of
convergence to the front. Then, as large perturbations of the front, we consider initial data which
are randomly chosen with some bias on the values at the left and at the right, mimicking an initial configuration which
vaguely resembles the transition from $u=0$ to $u=1$.
The large-time convergence to the propagation front is evident from the numerical output.\\

\noindent \textbf{Notations.}
We use lowercase boldface roman font to indicate column vectors (e.g., $\mathbf{w}$), and with the exception
of the identity matrix $I$ we use upper case boldface roman font to indicate square matrices (e.g., $\mathbf{A}$).
Elements of a matrix $\mathbf{A}$ (or vector $\mathbf{w}$) are denoted $A_{ij}$ (or $w_j$, respectively).
Linear operators acting on infinite-dimensional spaces are indicated with calligraphic letters (e.g., $\cL$ and $\cT$).
For a complex number $\lambda$, we denote complex conjugation by $\overline{\lambda}$ and denote its real
and imaginary parts by $\Real \lambda$ and $\Imag \lambda$, respectively.
Complex transposition of vectors or matrices are indicated by the symbol ${}^*$ (e.g., $\mathbf{w}^*$ and $\mathbf{A}^*$),
whereas simple transposition is denoted by the symbol $\mathbf{A}^\top$.
For a linear operator $\mathcal{L}$, its formal adjoint is denoted by $\mathcal{L}^*$.
Given $m\in\mathbb{N}$, the space $H^m(\R;\C^n)$ is composed of vector functions from $\R$ to $\C^n$,
where each component belongs to the Sobolev space $H^m(\R;\C)$.
It is endowed with the standard scalar product. Finally, we denote derivatives with respect to the indicated argument by `$'$' (e.g., $f'(u)$, $a'(x)$). Partial or total derivatives with respect to spatial and time variables (e.g. $x$ and $t$) are indicated by lower subscript. For the sake of simplicity, we sometimes use the symbol $\partial$ to indicate the latter when appropriate.

\section{Propagating fronts}\label{sect:existence}

Parabolic Allen--Cahn equation \eqref{parAC} supports traveling waves connecting the stable
states $u=0$ and $u=1$.
Also, the propagating speed and, up to translations, the wave profile, are unique.
For special forms of the reaction function $f$, existence of the wave can be proved by
determining explicit formulae for speed and profile.
For a general bistable $f$, the proof is based on a phase-plane analysis for the corresponding
nonlinear ordinary differential system.
Uniqueness arises as consequence of the fact that the heteroclinic orbit 
linking the asymptotic states is a saddle/saddle connection.

For system \eqref{relAC}, it is not anymore possible to find explicit traveling wave solutions
even for $f$ of polynomial type.
Phase-plane analysis, instead, is a more flexible approach and it can be applied also 
in the case of the Allen--Cahn equation with relaxation, as it is shown in the sequel.
In Subsection \ref{secnumspeed}, we also tackle the problem of the numerical evaluation of the
propagation speed in the case of a third-order polynomial function $f$, analyzing the relation
between the velocities of the hyperbolic and the parabolic Allen--Cahn equations for different
choices of the parameters $\alpha$ and $\tau$.

\subsection{Existence of the traveling wave}
Traveling waves of \eqref{relAC} are special solutions of the form $(u,v)(x,t)=(U,V)(\xi)$ where
$\xi=x-ct$ and $c$ is a real parameter.
The couple $(U,V)$ is referred to as the {\it profile of the wave} and the value $c$ as
its {\it propagation speed}.
Here, we are concerned with traveling waves satisfying the asymptotic
conditions \eqref{asymptoticstates}, so that the corresponding solution describes how
the relaxation system resolves the transition from one stable state to the other. 
In this respect, the value of the speed $c$ is very significant, since it describes how fast and
in which direction the switch from value $u=0$ to value $u=1$ is performed.

Existence of a traveling wave for \eqref{relAC} satisfying \eqref{asymptoticstates} can be deduced
by the fact that the ordinary differential equation for the profile, obtained by inserting the \textit{ansatz}
$u(x,t)=U(x-ct)$ in the one-field equation \eqref{relAConef}, is convertible into the corresponding
equation arising in the case of reaction-diffusion models with density-dependent diffusion.
Then, one applies a general result proven by Engler \cite{Engl85}, that relates the existence of
traveling wave solutions of reaction-diffusion equations with constant diffusion coefficients to the
ones of the density-dependent diffusion coefficient case. 
Considering such path too tangled, we prefer to give an explicit proof of the existence by dealing
directly with the system in the original form \eqref{relAC}.

Substituting the form of the traveling wave solutions,
we obtain the system of ordinary differential equations
\begin{equation}\label{twode0}
	cU_\xi+V_\xi +f(U)=0,\qquad
	U_\xi+c\tau V_\xi-V=0,
\end{equation}
to be complemented with the asymptotic boundary conditions \eqref{asymptoticstates}.
The value $c$ in \eqref{twode0} is an unknown and its determination is part of the problem.

\begin{proposition}\label{prop:properties}
Assume hypothesis \eqref{bistablef} and let $\tau \in [0,\tau_m)$.
If $(U,V)$ is a solution to \eqref{twode0} satisfying the asymptotic conditions
\eqref{asymptoticstates}, then  \par
\textrm{\rm (i)} the velocity $c$ has the same sign of $-\int_{0}^{1} f(u)\,du$;\par
\textrm{\rm (ii)}  there holds
\begin{equation}\label{subchar}
	c^2 \tau <1.
\end{equation}
\end{proposition}

\begin{proof}
\textrm{\rm (i)} Multiplying the first equation in \eqref{twode0} by $U_\xi$
and using the second, we obtain
\begin{equation*}
	\begin{aligned}
	0&=c\left|{U_\xi}\right|^2+{V_\xi}{U_\xi}+f(U){U_\xi}\\
	  &=c\left|{U_\xi}\right|^2
	  	+ ({U_\xi} 
	 	+c\tau{V_\xi})_\xi {U_\xi}+f(U){U_\xi}\\
	&=c\left|{U_\xi}\right|^2+{U_\xi}{U_{\xi\xi}} 
	 	-c\tau ( c{U_\xi}+f(U) )_\xi{U_\xi}
		+f(U){U_\xi}\\
	&=c(1-\tau f'(U))\left|{U_\xi}\right|^2
		+(1-c^2\tau){U_\xi}{U_{\xi\xi}} 
		+f(U){U_\xi}.
	\end{aligned}
\end{equation*}
Thus, denoting by $F$ a primitive of $f$, there holds
\begin{equation}\label{waverelation}
	 \big( \tfrac{1}{2}(1-c^2\tau)\left|{U_\xi}\right|^2+F(U) \big)_\xi 
	+c(1-\tau f'(U))\left|{U_\xi}\right|^2=0.
\end{equation}
Integrating in $\R$, we infer the relation
\begin{equation*}
	c\int_{\R} (1-\tau f'(U))\left|{U_\xi}\right|^2\,d\xi
	 =F(0)-F(1)=-\int_{0}^{1} f(u)\,du.
\end{equation*}
Since $\tau<\tau_m$, then $\tau f'(u)<1$ for any $u$ and part (i) follows.

\textrm{\rm (ii)} The case $c=0$ is obvious.
Let us assume $c<0$ (the opposite case being similar).
Integrating the equality \eqref{waverelation} in $(-\infty,\xi)$, we get
\begin{equation*}
	\tfrac{1}{2}(1-c^2\tau)\left|{U_\xi}\right|^2
		=F(0)-F(U(\xi))-c\int_{-\infty}^{\xi}(1-\tau f'(U))\left|{U_\xi}\right|^2\,d\xi.
\end{equation*}
Choosing $\xi$ such that $U(\xi)\in(0,\alpha)$, since $F$ is strictly decreasing
in $(0,\alpha)$, the right-hand side is strictly positive and thus \eqref{subchar} holds.
\end{proof}

Condition \eqref{subchar} should be regarded as a \textit{subcharacteristic condition}.
Indeed, it has a similar interpretation as the corresponding relation for hyperbolic systems
with relaxation: the equilibrium wave velocity cannot exceed the characteristic speed of
the perturbed wave equation \eqref{relAConef}. 

Thanks to \eqref{subchar}, we are allowed to introduce the independent variable
$\eta=(1-c^2\tau)^{-1}\xi$,  so that system \eqref{twode0} becomes
\begin{equation}\label{twode1}
	U_\eta=\phi(U,V):=c\tau f(U)+V,\qquad
	V_\eta=\psi(U,V):=-f(U)-cV.
\end{equation}
Departing from a detailed description of the unstable and stable manifold of the singular points
$(0,0)$ and $(1,0)$, respectively, it is possible to show the existence of a saddle/saddle connection
between the asymptotic states required by \eqref{asymptoticstates}.
The existence result is based on the analysis of the limiting regimes $c\to \pm 1/\sqrt{\tau}$
of the system \eqref{twode1} and on their (monotone) variations for the values in between, using the
notion of a {\it rotated vector field} (cf. \cite{Perk93}).

\begin{proof}[Theorem \ref{thm:existence}]
1. The linearization of \eqref{twode1} at $(\bar U,0)$ with $f(\bar U)=0$ is described
by the jacobian matrix calculated at $(\bar U,0)$
\begin{equation*}
	\frac{\partial(\phi,\psi)}{\partial(u,v)}
	=\begin{pmatrix} c\tau f'(\bar U)	& \;\;1\\ - f'(\bar U) & \;\; -c \end{pmatrix}
\end{equation*}
whose determinant is $(1-c^2\tau)f'(\bar U)$.
In particular, if $f'(\bar U)<0$, the singular point $(\bar U,0)$ is a saddle.
The eigenvalues are the roots of the polynomial
\begin{equation*}
	p(\mu)=\mu^2+c(1-\tau f'(\bar U))\mu+(1-c^2\tau)f'(\bar U)
\end{equation*}
and they are given by
\begin{equation*}
	\mu_\pm(\bar U)=- \tfrac{1}{2}c(1-\tau f'(\bar U))\pm
		\tfrac{1}{2}\sqrt{c^2(1-\tau f'(\bar U))^2-4 f'(\bar U)}
\end{equation*}
with corresponding eigenvectors $\mathbf{r}_\pm(\bar U)=(1,\mu_\pm(\bar U)-c\tau f'(\bar U))^\top$.

For later use, let us note that, as a consequence of \eqref{smalltau2}, 
\begin{equation}\label{lateruse}
	\mu_+(0)-c\tau f'(0)>-\sqrt{\tau}\,f'(0)
	\quad\textrm{and}\quad
	\mu_-(1)-c\tau f'(1)<\sqrt{\tau}\,f'(1),
\end{equation}
for $c\in(-1/\sqrt{\tau},1/\sqrt{\tau})$.
Indeed, for $f'(\bar U)<0$, there holds
\begin{equation*}
	p\bigl((c\tau\pm\sqrt{\tau})f'(\bar U)\bigr)
		=\sqrt{\tau}\left(\frac{1}{\sqrt{\tau}}\pm c\right)(1-\tau f'(\bar U))f'(\bar U)<0,
\end{equation*}
so that the values $(c\tau\pm\sqrt{\tau})f'(\bar U)$ 
belong to the interval $(\mu_-(\bar U),\mu_+(\bar U))$.
\vskip.25cm

2. Given $c\in (-1/\sqrt{\tau},1/\sqrt{\tau})$, let us denote by ${\mathbb{U}}_{0}(c)$ the unstable
manifold of the singular point $(0,0)$ and by ${\mathbb{S}}_{1}(c)$ the stable manifold of $(1,0)$.
Also, let ${\mathbb{U}}_{0}^+(c)$ be the intersection  of ${\mathbb{U}}_{0}(c)$ with the strip
$[0,\alpha]\times\R$ and let ${\mathbb{S}}_{1}^-(c)$ the intersection of ${\mathbb{S}}_{1}(c)$ with
the strip $[\alpha,1]\times\R$.
Such sets are graphs of appropriate solutions to the first order equation
\begin{equation}\label{traj}
	\frac{dV}{dU}=-\frac{f(U)+cV}{c\tau f(U)+V}.
\end{equation}
Thanks to \eqref{lateruse}, ${\mathbb{U}}_{0}^+(c)$ lies above the graph of 
the function $v=-\sqrt{\tau} f(u)$ in a neighborhood of $(0,0)$.
In addition, for $u\in(0,\alpha)$, there holds
\begin{equation*}
	(\phi,\psi)\cdot(\sqrt{\tau}\,f'(u),1)\bigr|_{v=-\sqrt{\tau} f(u)}
	=-\left(1-c\sqrt{\tau}\right)(1-\tau f'(u))f(u)>0
\end{equation*}
showing that no trajectories may trespass the graph of the function $v=-\sqrt{\tau} f(u)$
for $u\in(0,\alpha)$.
In particular, the set ${\mathbb{U}}_{0}^+(c)$ lies above the graph $v=-\sqrt{\tau} f(u)$
and hits the line $u=\alpha$ for a given value $v_0(c)\in(0,+\infty)$.
Similar considerations show that ${\mathbb{S}}_{1}^-(c)$ stays above the  graph
$v=\sqrt{\tau} f(u)$ and touches the straight line $u=\alpha$ for a given value $v_1(c)\in(0,+\infty)$.

3. To determine how the unstable/stable manifolds change with the parameter
$c$, let us observe that
\begin{equation*}
	\begin{aligned}
	(\phi,\psi,0)^\top\land(\partial_c \phi,\partial_c\psi,0)^\top
		&=\det\begin{pmatrix} 
			\mathbf{i}		&	\mathbf{j}		&	\mathbf{k}	\\
			c\tau f(U)+V 	&	-f(U)-cV 		& 	0		\\
			\tau f(U)	 	&	-V			& 	0
				\end{pmatrix}\\
		&=\bigl(\tau f(U)^2-V^2\bigr)\mathbf{k}.
	\end{aligned}
\end{equation*}
Thus, the vector field $(\phi,\psi)$ defining \eqref{twode1} rotates clockwise in the region
$\{(u,v)\,:\,v\geq \sqrt{\tau}|f(u)|\}$ as the parameter $c$ increases.
As a consequence, the curves describing ${\mathbb{U}}_0^+(c)$ and
${\mathbb{S}}_1^-(c)$ rotate clockwise when $c$ increases and 
the functions $c\mapsto v_0(c)$ and $c\mapsto v_1(c)$ are, respectively,
strictly monotone decreasing and strictly monotone increasing.

4. To conclude the existence of the orbit, we analyze the behavior of \eqref{twode1}
in the limiting regimes $c\to \pm 1/\sqrt{\tau}$.
For $c=-1/\sqrt{\tau}$, the system reduces to 
\begin{equation*}
	U_\eta=V-\sqrt{\tau} f(U),\qquad
	V_\eta=\frac{1}{\sqrt{\tau}}\Big(V-\sqrt{\tau} f(U)\Big).
\end{equation*}
In particular, all the trajectories lie along straight lines of the form 
\begin{equation*}
	v=\frac{u}{\sqrt{\tau}}+C,\qquad\qquad C\in\R,
\end{equation*}
and converge, as $t\to -\infty$, to the unique intersection between the invariant straight line
and the graph of the function $\sqrt{\tau} f$
(see Fig. \ref{fig:pplaneminus}, depicting the $(U,V)$ plane for the particular case of
the cubic reaction function $f$).
\begin{figure}\centering
\resizebox{0.8\hsize}{!}{\includegraphics*{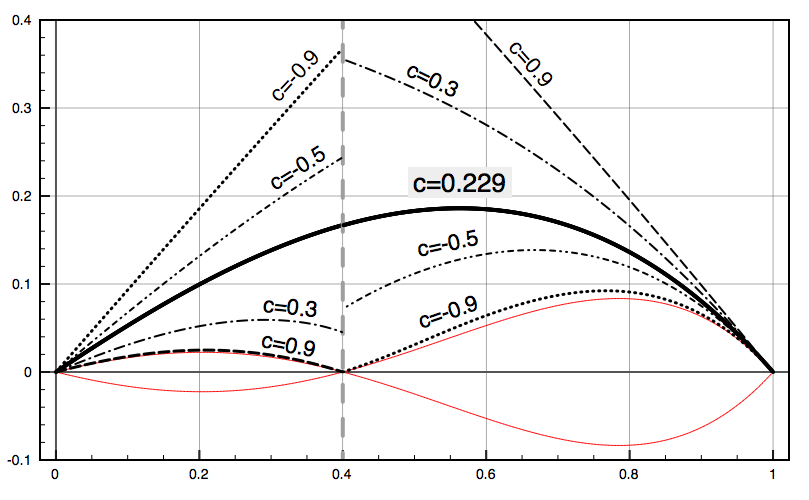}}
\caption{Plane $(U,V)$ in the case $f(u)=u(1-u)(u-\alpha)(0.5+u)$, 
$\alpha=0.4$, $\tau=1$.
The manifolds ${\mathbb{U}}_{0}^+$ and ${\mathbb{S}}_{1}^-$ are represented
for different choices of $c\in(-1,1)$.
Note the monotonicity with respect to the parameter $c$.
For the choice $c=0.229$ the two curves intersect at $u=\alpha=0.4$.
The thin (red; online version) lines on the bottom are the graphs of the functions
$\pm\sqrt{\tau}\,f$.
}\label{fig:pplaneminus}
\end{figure}
The unstable manifold ${\mathbb{U}}_{0}(c)$ of the singular point $(0,0)$ is the
straight line  $v=u/\sqrt{\tau}$, while the center-stable manifold 
${\mathbb{S}}_{1}(c)$  of $(1,0)$ is the graph of the function $\sqrt{\tau} f$.
In particular, there holds
\begin{equation*}
	{\mathbb{U}}_{0}(-1/\sqrt{\tau})\bigr|_{U=\alpha}=\alpha(1,1/\sqrt{\tau}\,),\qquad
	{\mathbb{S}}_{1}(-1/\sqrt{\tau})\bigr|_{U=\alpha}=\alpha(1, 0),
\end{equation*}
that gives
\begin{equation*}
	v_0(-1/\sqrt{\tau})=\alpha/\sqrt{\tau},\qquad v_1(-1/\sqrt{\tau})=0.
\end{equation*}
The situation for $c=1/\sqrt{\tau}$ is similar, yielding
\begin{equation*}
	v_0(1/\sqrt{\tau})=0,\qquad
	v_1(1/\sqrt{\tau})=(1-\alpha)/\sqrt{\tau}.
\end{equation*}
The conclusion is at hand, since
\begin{equation*}
	(v_1-v_0)(-1/\sqrt{\tau})=-\alpha/\sqrt{\tau}<0
		<(1-\alpha)/\sqrt{\tau}=(v_1-v_0)(1/\sqrt{\tau}),
\end{equation*}
implying, together with the monotonicity of $v_0$ and $v_1$,
that there exists a unique value $c_\ast$ such that $v_1(c_\ast)=v_0(c_\ast)$,
and then system \eqref{twode1} possesses a heteroclinic orbit.

5. To verify the monotonicity of the component $U$, we note that from system \eqref{twode1}
one has $U_\eta = c\tau f(U) + V$. Let us suppose that $c > 0$.
Assume, by contradiction, that $U_\eta(\eta_*) = 0$ for some $\eta_* \in \R$.
Therefore $V_* = - c\tau f(U_*)$, where $V_* = V(\eta_*)$, $U_* = U(\eta_*)$.
Since $V$ is positive for all $\eta \in \R$, this implies that $f(U_* ) < 0$
and, necessarily, that $U_* \in (0,\alpha)$. The subcharacteristic condition \eqref{subchar},
however, yields
\begin{equation*}
	0 < -c \tau f(U_*) < - \sqrt{\tau} f(U_*),
\end{equation*}
which is a contradiction with the fact that the trajectory $(U,V)$ lies above the graph of the
function $v = - \sqrt{\tau} f(u)$ for $u \in (0,\alpha)$.
The case $c < 0$ is analogous and a similar argument for $U_* \in (\alpha,1)$ applies.
Therefore, $U_\eta$ never changes sign along the trajectory.
Since $U$ connects $U(-\infty) = 0$ with $U(+\infty) = 1$, the function is strictly increasing
and $U_\eta > 0$ for all $\eta \in \R$.
The subcharacteristic condition guarantees that the same statement holds for the
original profile in the $\xi$ variable, that is, $U_\xi > 0$.

6. Exponential decay of the profile is a consequence of the hyperbolicity of the
non-degenerate end-points $(0,0)$ and $(1,0)$.
Indeed, rewriting system \eqref{twode0} (in the original moving variable $\xi$) as
\begin{equation*}
  	U_\xi = (1 - c^2 \tau)^{-1}(c\tau f(U) + V),\qquad
	V_\xi = -(1 - c^2 \tau)^{-1}(\tau f(U) + cV),
\end{equation*}
and linearizing the right hand side around $(\bar U,0)$, one readily notes that its eigenvalues
are the same as the eigenvalues of the linearization of \eqref{twode1} at the same point
multiplied by a factor $(1 - c^2 \tau)^{-1}$.
The positive (unstable) eigenvalue at $(0,0)$ is
\begin{equation*}
	\mu_u = \tfrac{1}{2}(1 - c^2 \tau)^{-1}\big( -c(1-\tau f'(0))
		+ \sqrt{c^2(1-\tau f'(0))^2 - 4 f'(0)}\big) \, > 0,
\end{equation*}
and the orbit decays to $(0,0)$ as $\xi \to - \infty$ with exponential rate
\begin{equation*}
	|(U,V)(\xi)| \leq C \exp (\mu_u \xi).
\end{equation*}
Likewise, the negative (stable) eigenvalue at $(1,0)$ is 
\begin{equation*}
	\mu_s = \tfrac{1}{2}(1 - c^2 \tau)^{-1}\big( -c(1-\tau f'(1))
		- \sqrt{c^2(1-\tau f'(1))^2 - 4 f'(1)}\big) \, < 0,
\end{equation*}
and the orbit decays to $(1,0)$ as $\xi \to + \infty$ with rate
\begin{equation*}
	|(U,V)(\xi)| \leq C \exp (-|\mu_s| \xi).
\end{equation*}
Setting $\nu = \nu(\tau) := \min \{\mu_u, |\mu_s|\} > 0$, we find that
\begin{equation*}
	\left| \frac{d^j}{d\xi^j}(U-U_\pm,V)(\xi)\right| \leq C \exp (-\nu |\xi|),
	\qquad \xi\in\R,
\end{equation*}
for some constant $C > 0$ and with $j = 0,1,2$.

7. Finally, we have to show continuity of the speed $c_\ast$ with respect to $\tau$, 
as stated in (c). To this aim, denoting explicitly the dependence on $\tau$, let us consider the function
\begin{equation*}
	\delta v(c,\tau):=v_0(c,\tau)-v_1(c,\tau)
\end{equation*}
with $v_0$ and $v_1$ defined at Step 2.
For any $\tau$, the value $c_\ast$ is determined implicitly by the equality $\delta v(c,\tau)=0$.
Also, smooth dependence with respect to parameters $c$ and $\tau$ of the system \eqref{twode1}
implies that the function $\delta v$ is continuous with respect to its variables.
Moreover, since $v_0$ is monotone decreasing and $v_1$ monotone increasing as functions of $c$
(see the end of Step 3.), the function $\delta v$ is monotone decreasing with respect to $c$.

Fix $\tau_0\in(0,\tau_m)$ and $\eta>0$.
Then, there holds
\begin{equation*}
	\delta v(c_\ast(\tau_0)+\eta,\tau_0)<\delta v(c_\ast(\tau_0),\tau_0)=0
		<\delta v(c_\ast(\tau_0)-\eta,\tau_0)
\end{equation*}
Since $\delta v$ is continuous, for any $\tau$ in a neighborhood of $\tau_0$, there holds
\begin{equation*}
	\delta v(c_\ast(\tau_0)+\eta,\tau)<0<\delta v(c_\ast(\tau_0)-\eta,\tau),
\end{equation*}
which gives, for $\delta v(\tau,c_\ast(\tau))=0$ and the monotonicity of $\delta v$,
\begin{equation*}
	c_\ast(\tau_0)-\eta<c_\ast(\tau)<c_\ast(\tau_0)+\eta.
\end{equation*}
The property relative to the limiting behavior as $\tau\to 0$ can be proved in the same way,
observing that the dependence of the differential system with respect to $\tau$ is smooth 
and that the heteroclinic orbit relative to the classical Allen--Cahn case can be obtained 
by the same procedure.
\end{proof}

\subsection{Numerics of the propagating speed}\label{secnumspeed}

In the case of the Allen--Cahn equation \eqref{parAC} with $f$ given by \eqref{cubicf},
it is possible to determine an explicit form for the speed of the propagation front, namely
\begin{equation}\label{parACspeed}
	c_\ast^{0}:=\sqrt{\frac{2}{\kappa}}\left(\alpha-\frac12\right),
\end{equation}
and for the corresponding profile, which is given by a hyperbolic tangent.
For the Allen--Cahn system with relaxation, however, finding analogous
explicit formulas is awkward if not impossible. 
Some attempts to derive approximated expressions applying a series expansion
method have been performed in \cite{Abdu04,Fahm07,VanGVajr10} with very
restricted success.

Here, fixed the reaction strength $\kappa=1$,
we address the problem of computing numerically the value of $c_\ast^\tau$,
the speed of propagation corresponding to the relaxation parameter $\tau>0$,
and we discuss its dependency with respect to the parameters $\alpha, \tau$
and its relation with the limiting value $c_\ast^0$.
\begin{figure}\centering
\resizebox{0.8\hsize}{!}{\includegraphics*{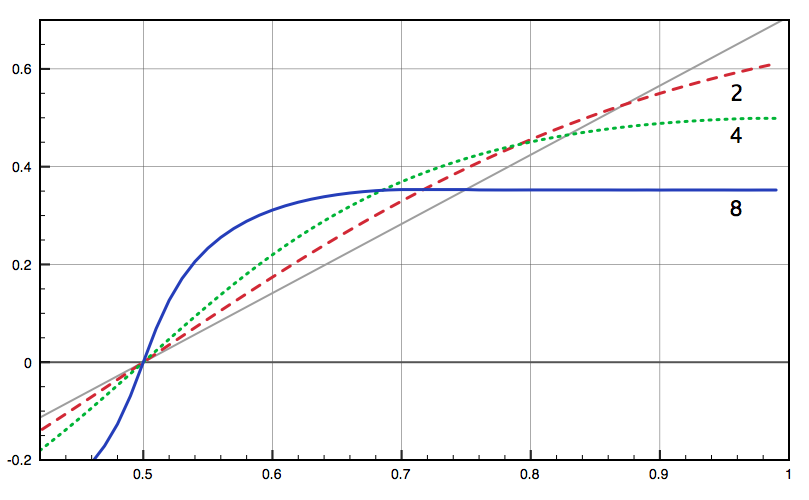}}
\caption{Graph of the map $\alpha\mapsto c_\ast^\tau$
for $\tau=2$ (dashed line), $\tau=4$ (dotted line), $\tau=8$ (continuous line).
The thin straight line corresponds to $c_\ast^{0}$ (parabolic Allen--Cahn equation).
}\label{fig:hetRelcVal}
\end{figure}

To determine reliable approximations of the value $c_\ast^\tau$, we first evaluate numerically
the functions $v_0=v_0(c)$ and $v_1=v_1(c)$, defined in the proof of Theorem \ref{thm:existence}.
In view of that, we compute the solutions $V_0=V_0(U)$ and $V_1=V_1(U)$ to \eqref{traj} with initial conditions
\begin{equation*}
	V_0(\delta)=\delta\bigl(\mu_+(0)-c\tau f'(0)\bigr),\qquad
	V_1(1-\delta)=-\delta\bigl(\mu_-(1)-c\tau f'(1)\bigr),
\end{equation*}	
for $\delta>0$ small (in the following computations, we actually choose $\delta=10^{-8}$) and we set
\begin{equation*}
	v_0(c):=V_0(\alpha),\qquad v_1(c):=V_1(\alpha).
\end{equation*}
As explained in the proof of Theorem \ref{thm:existence}, the function
$c\mapsto v_0(c)-v_1(c)$ is monotone decreasing and it has a single zero,
so that, by following a standard bisection procedure, we find an approximated
value for the unique zero $c_\ast^\tau$ of the difference $v_0-v_1$.
Some of the numerically computed speeds for different values of
$\alpha$ and $\tau$ can be found in Table \ref{tab:numerspeeds}.

\begin{table}\centering
\caption{Numerically computed speeds $c_\ast^\tau=c_\ast^\tau(\alpha)$ for different values of
$\alpha$ and $\tau$.
The column $\tau=0$ gives the values of the speed for the parabolic Allen--Cahn equation.
The presence of parenthesis indicates that condition \eqref{smalltau2} is not satisfied.
\label{tab:numerspeeds}}
{\begin{tabular}{@{}r|c|c|c|c|c|c|c|c|c|c|@{}} 
$\alpha$	& $\tau=$0 & 1 	   & 2      & 3 	   & 4 	& 5 	      & 6 	   & 7 	& 8 	     \\ \hline
	0.6	& 0.14	  & 0.16 & 0.17 & (0.20) & (0.22)	& (0.25) & (0.27) & (0.30)	& (0.31) \\
	0.7	& 0.28	  & 0.31 & 0.33 & 0.35	& (0.37)	& (0.38) & (0.38) & (0.37)	& (0.35) \\
	0.8	& 0.42	  & 0.44 & 0.46 & 0.46	   & 0.45 	& (0.43) & (0.41) & (0.38)	& (0.35) \\
	0.9	& 0.57	  & 0.56 & 0.55 & 0.52    & 0.49	& 0.45   & 0.41	   & 0.38	& 0.35
\end{tabular}}
\end{table}

Applying a variational approach, explicit estimates from below and from above for the value
of the speed $c_\ast^\tau$ have been determined by M\'endez \textit{et al.} \cite{MendFortFarj99}.
As already noted by these authors, the estimate from below
\begin{equation}\label{estbelow}
	c_\ast^\tau(\alpha)\geq
		c_{{\textrm{low}}}^\tau(\alpha)
		:=\frac{\sqrt{2}(\alpha-1/2)}
		{\sqrt{(1-\frac{1}{5}(1-2\alpha+2\alpha^2)\tau)^2+\frac12\tau(1-2\alpha)^2}}
\end{equation}
manifests an excellent agreement with the numerically computed values of the speed.

The value $c_\ast^\tau$ depends on both $\tau$ and $\alpha$.
For fixed $\tau$, as a function of $\alpha$ alone, $c_\ast^\tau$ is oddly symmetric with
respect to $\alpha=1/2$.
Moreover, numerical simulations show that $c_\ast^\tau$ is monotone increasing
and $S$-shaped (see Fig. \ref{fig:hetRelcVal}).

\begin{figure}\centering
\resizebox{0.8\hsize}{!}{\includegraphics*{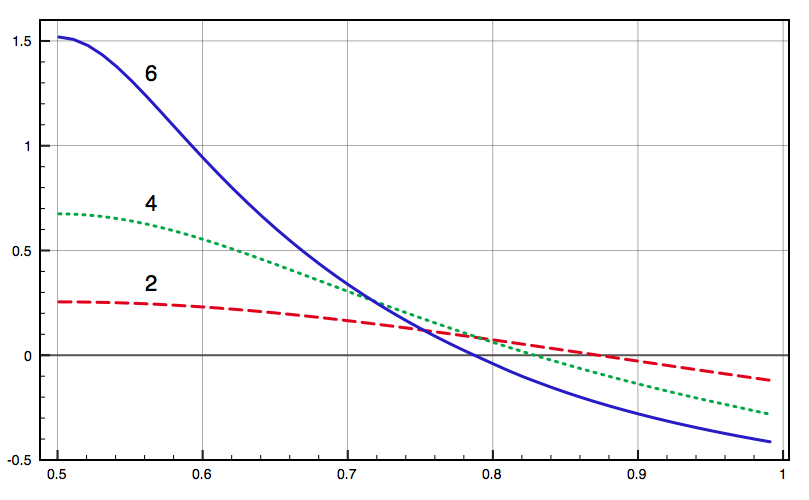}}
\caption{Graph of the map $\alpha\mapsto (c_\ast^\tau-c_\ast^0)/|c_\ast^0|$
for $\alpha\in(0.5,1)$ measuring the relative variation of the speed passing from the parabolic 
Allen--Cahn equation to its hyperbolic counterpart:
$\tau=2$ (dashed), $\tau=4$ (dotted), $\tau=6$ (continuous).
}\label{fig:hetRelRela}
\end{figure}

In some regimes of the parameter $\alpha$ (depending on the size of $\tau$),
the speed for the hyperbolic model is larger than the one for the parabolic Allen--Cahn equation.
Such behavior is different with respect to the {\it damped Allen--Cahn equation},
obtained by solely adding the inertial term $\tau u_{tt}$ to equation \eqref{parAC}
\begin{equation*}
	\tau u_{tt} +  u_t = u_{xx} +f(u),
\end{equation*}
(see \cite{GallJoly09} and the references therein).
Indeed, for such equation, the traveling wave equation can be directly
reduced to the one of the parabolic case by a simple rescaling of the independent variable.
Such procedure furnishes an explicit relation between the speed of propagation of the fronts
with and without the inertial term, that is
\begin{equation*}
	c^\tau_{\text{damped}} = \frac{c^0_\ast}{\sqrt{1+\tau (c^0_\ast)^2}}.
\end{equation*}
From this relation, it is evident that such hyperbolic front is always slower with respect
to the corresponding parabolic one.

Coming back to the case of the Allen--Cahn equation with relaxation,
the discrepancy between $c_\ast^\tau$ and $c_\ast^0$ is described by the function of the 
relative variation $\alpha\mapsto (c_\ast^\tau-c_\ast^0)/|c_\ast^0|$, whose behavior is depicted
in Fig. \ref{fig:hetRelRela} for different values of $\tau$.
For large values of $\tau$, the relative increase of the front can be particularly 
relevant (as an example, about 150\% for $\tau=6$ and values of $\alpha$ close to $0.5$).
The limiting value of the relative variation as $\alpha\to 1/2$ can be approximated by
using the estimate \eqref{estbelow}, that gives
\begin{equation*}
	\lim_{\alpha\to 1/2^+} \frac{c_{\ast}^\tau-c_\ast^0}{|c_\ast^0|}
	\approx
	\lim_{\alpha\to 1/2^+} \frac{c_{{\textrm{low}}}^\tau-c_\ast^0}{|c_\ast^0|}
		=\frac{\tau}{10-\tau}
\end{equation*}
in good agreement with the numerical values expressed in Fig. \ref{fig:hetRelRela}.
It is worth noting the presence of a singularity for $\tau\to 10^-$ that is outside
the range of smallness on the parameter $\tau$ that we are considering.
We do consider the behavior for large relaxation times $\tau$ a very interesting issue
to be analyzed in detail in the future.

Complementary information is provided by the analysis of the value of $c_\ast^\tau$
as a function of $\tau$ for fixed $\alpha$ (see Fig. \ref{fig:hetRel4}).
The numerical evidence shows that the speed function has different monotonicity properties
depending on the chosen value $\alpha$, passing through the monotone increasing case
($\alpha=0.6$), increasing-decreasing ($\alpha=0.7$ and $0.8$), monotone decreasing ($\alpha=0.9$).

It also worthwhile to observe the relation between the propagation speed $c_\ast^\tau$
and the characteristic speed $1/\sqrt{\tau}$.
As stated in Proposition \ref{prop:properties}, $|c_\ast^\tau|$ is always strictly smaller than
$1/\sqrt{\tau}$, as it is observed in Fig. \ref{fig:hetRel4}, where the graph of
the characteristic speed is depicted by a thick (gray, in the online version) band, lying above
the speed curves for all of the values $\alpha$.
For large values of $\tau$, the hyperbolic structure of the principal part of \eqref{relAC} becomes
crucial and the propagation speed of the front tends to one of the
characteristic values $\pm 1/\sqrt{\tau}$
(except for the case of zero speed).
\begin{figure}\centering
\resizebox{0.8\hsize}{!}{\includegraphics*{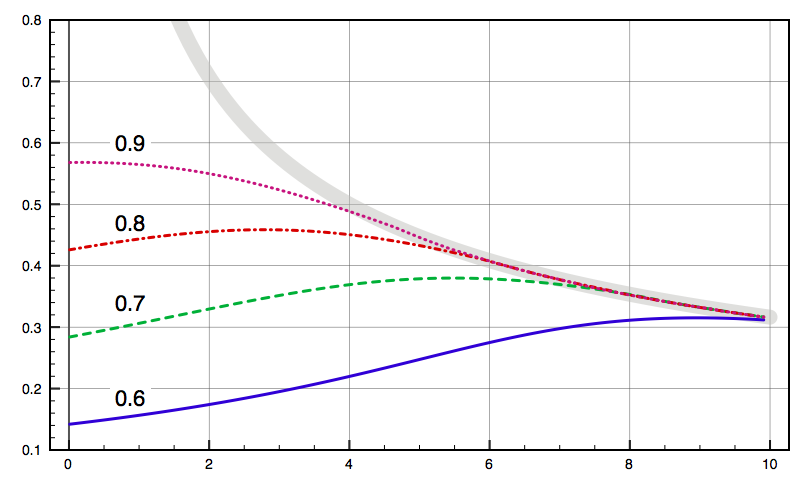}}
\caption{Graph of the map $\tau\mapsto c_\ast^\tau(\alpha)$ for different values of $\alpha$:
continuous line $\alpha=0.6$, dashed line $\alpha=0.7$, dash-dotted line $\alpha=0.8$,
dotted line $\alpha=0.9$.
The thick (gray; online version) band corresponds to the graph of the characteristic
speed $\tau\mapsto 1/\sqrt{\tau}$;
as $\tau$ increases, the speed $c_\ast^\tau(\alpha)$ tends to $1/\sqrt{\tau}$ from below
as a consequence of the emergent r\^ole of the principal part of the system \eqref{relAC}.}
\label{fig:hetRel4}
\end{figure}

\section{Linearization and spectral stability}\label{sect:linspecstab}

This section is devoted to establishing both the equations for the perturbation
of the traveling wave and the corresponding spectral stability problem.
For stability purposes, it is often convenient to recast the system of equations \eqref{relAC}
in a moving coordinate frame.
For fixed $\tau \in [0,\tau_m)$, let $c := c_*(\tau)$ be the unique wave speed of
Theorem \ref{thm:existence}.
Rescaling the spatial variables as $x\mapsto x - c t$, we obtain the nonlinear system
\begin{equation}\label{relACmove}
	\begin{aligned}
 		u_t &= cu_x  + v_x + f(u),\\
		\tau v_t &= u_x  + c \tau v_x -v.
	\end{aligned}
\end{equation}
From this point on and for the rest of the paper the variable $x$ will denote the
moving (galilean) variable $x - ct$.
With a slight abuse of notation, traveling wave solutions to the original system \eqref{relAC}
transform into stationary solutions $(U,V)(x)$ of the new system \eqref{relACmove} and
satisfy the ``profile'' equations
\begin{equation}\label{profileq}
	cU_x + V_x + f(U)= 0,\qquad
	U_x + c\tau V_x - V = 0,
\end{equation}
together with the asymptotic limits \eqref{asymptoticstates}.
Furthermore, the convergence is exponential,
\begin{equation}\label{expodecay}
	\left| \partial_x^j (U-U_\pm,V)(x)\right| \leq C \exp (-\nu |x|),
	\qquad x \in \R
\end{equation}
for some $C,\nu > 0$ and $j = 0,1,2$. It is to be noted that $U_x, V_x \in H^1(\R;\R)$. Moreover, by a 
bootstrapping argument and the second equation in \eqref{profileq}, it is easy to verify that $U_{x}, V_{x} 
\in H^2(\R;\R)$.

\subsection{Perturbation equations and the spectral problem}

Consider a solution to the nonlinear system \eqref{relACmove} of the form
$(u,v)(x,t) + (U,V)(x)$, being $u$ and $v$ perturbation variables.
Upon substitution (and using the profile equations \eqref{profileq}) we arrive
at the following nonlinear system for the perturbation:
\begin{equation}\label{nlsystw}
	\begin{aligned}
		u_t  &=  c u_x  + v_x  + f(u+U) - f(u),\\
		\tau v_t  &= u_x  + c\tau v_x  -v.
	\end{aligned}
\end{equation}
The standard strategy to prove stability of the traveling wave is based on linearizing
system \eqref{nlsystw} around the wave.
The subsequent analysis can be divided into three steps: the spectral analysis of the
resulting linearized operator, the establishment of linear stability estimates for the associated
semigroup, and the nonlinear stability under small perturbations to solutions to \eqref{nlsystw}.
Thus, we first linearize last system around the profile solutions $(U,V)$. The result is
\begin{equation*}
 \begin{aligned}
		u_t  &=  c u_x  + v_x  + f'(U)u,\\
		\tau v_t  &= u_x  + c\tau v_x  -v.
	\end{aligned}
\end{equation*}
Specializing these equations to perturbations of the form
$e^{\lambda t}(\hat u,\hat v)$, where $\lambda \in \C$ is the spectral parameter and $(\hat u,\hat v)(x)$
belongs to an appropriate Banach space $X$, we obtain a naturally associated spectral problem
\begin{equation*}
 \begin{aligned}
		\lambda \hat u  &=  c \hat u_x  + \hat v_x  + f'(U)\hat u,\\
		\lambda \tau \hat v  &= \hat u_x  + c\tau \hat v_x  - \hat v.
	\end{aligned}
\end{equation*}
With a slight abuse of notation we denote again $(\hat u, \hat v)^\top = (u,v)^\top \in X$.
Henceforth, for each $\tau \in (0,\tau_m)$ we are interested in studying the spectral problem
\begin{equation}\label{spectproblem}
	\mathcal{L}^\tau \begin{pmatrix} 
	                  u\\ v
	                 \end{pmatrix} =
\lambda \begin{pmatrix}
	                  u\\ v
	                 \end{pmatrix}, \qquad \begin{pmatrix} 
	                  u\\ v
	                 \end{pmatrix} \in D \subset X, 
\end{equation}
where $\mathcal{L}^\tau$ denotes the first order differential operator determined by
\begin{equation}
\label{defLtau}
	\mathcal{L}^\tau  = - \mathbf{B}^{-1}\left( \mathbf{A} \, \partial_x + \mathbf{C}(x) \right),
\end{equation}
with domain $D$ in $X$, and where
\begin{equation}\label{defABC}
	\mathbf{A}=\begin{pmatrix} -c & -1 \\ -1 & -c\tau \end{pmatrix},\qquad
	\mathbf{B}=\begin{pmatrix} 1 & 0 \\ 0 & \tau \end{pmatrix},\qquad
	\mathbf{C}(x)=\begin{pmatrix} -a(x) & 0 \\ 0 & 1 \end{pmatrix},
\end{equation}
and
\begin{equation*}
 a(x) := f'(U).
\end{equation*}
In this analysis we choose the perturbation space as
$X = L^2(\R;\C^2)$, with dense domain $D = H^1(\R;\C^2)$, which corresponds
to the study of stability under \textit{localized} perturbations.
In this fashion, we obtain a family of closed, densely defined first order operators
in $L^2(\R;\C^2)$, parametrized by $\tau \in (0,\tau_m)$. 

It is to be observed that $\mathcal{L}^\tau$ is not defined for $\tau = 0$, where $\mathbf{B}$ becomes singular.
Formally, in the limit $\tau \to 0^+$ system \eqref{nlsystw} can be written (after substitution) as the
scalar perturbation equation for the parabolic Allen-Cahn (or Nagumo) front:
\begin{equation*}
	u_t  = u_{xx} + c_0 u_x  + f(U^0+u) - f(U^0),
\end{equation*}
where $U^0$ denotes the unique (up to translations) traveling wave solution to the
parabolic Allen-Cahn equation \eqref{parAC}, traveling with speed $c_0 = c_{|\tau = 0}$.
Its linearization leads to the well-studied spectral problem for the operator
$\mathcal{L}_0 : H^2(\R;\C) \rightarrow L^2(\R;\C)$, defined by
\begin{equation*}
	\mathcal{L}_0 u := u_{xx} + c_0 u_x  + \overline{a}(x)u \, = \, \lambda u,
	\quad u \in H^2(\R;\C),
\end{equation*}
where $\overline{a}(x) = f'(U^0)$.
The stability of the parabolic traveling front is a well-known fact: it was first established
by Fife and McLeod \cite{FifeMcLe77} using maximum principles.
The spectral analysis of the operator $\mathcal{L}_0$ can be found in \cite[pp.128--131]{He81},
and in \cite[Chapter 2]{KaPro13}. The following proposition summarizes the spectral stability properties of the parabolic front.

\begin{proposition}[cf. \cite{He81,KaPro13}]\label{prop:tau0}
There exists $\omega_0>0$ such that the spectrum $\sigma(\mathcal{L}_0)$ of the
operator $\mathcal{L}_0$ can be decomposed as
\begin{equation*}
	\sigma(\mathcal{L}_0)=\{0\}\cup\sigma_-^{(0)},
\end{equation*}
where $\lambda = 0$ is an (isolated) eigenvalue with algebraic multiplicity equal to one
and eigenspace generated by $U^0_x \in L^2(\R;\C)$, and $\sigma_-^{(0)}$ is contained in the
half-space $\{\lambda\in\C\,:\,\Real \, \lambda\leq -\omega_0\}$.
\end{proposition}

Since the operators $\mathcal{L}^\tau$ have been defined on the appropriate spaces,
the standard definitions for the resolvent $\rho(\mathcal{L}^\tau)$ and spectrum
$\sigma(\mathcal{L}^\tau)$ follow (see \cite{EE87,Kat80} and Section \ref{secfirstorder} below).
Our goal is to establish the spectral stability for the family of operators \eqref{defLtau},
for parameter values $\tau \in (0,\tau_m)$,  by proving a result analogous to Proposition \ref{prop:tau0}:

\begin{theorem}[Spectral stability]\label{thm:spectrum}
For each $\tau \in (0,\tau_m)$, there exists $\omega_0(\tau)>0$ such that the spectrum
$\sigma(\mathcal{L}^\tau)$ of the operator $\mathcal{L}^\tau$ can be decomposed as
\begin{equation*}
	\sigma(\mathcal{L}^\tau)=\{0\}\cup\sigma_-^{(\tau)},
\end{equation*}
where $\lambda = 0$ is an (isolated) eigenvalue with algebraic multiplicity equal to one
and eigenspace generated by $(U_x,V_x) \in D(\mathcal{L}^\tau)$, and $\sigma_-^{(0)}$ is contained in the
half-space $\{\lambda\in\C\,:\,\Real\lambda\leq -\omega_0(\tau)\}$.
\end{theorem}

The approach to prove Theorem \ref{thm:spectrum} is based on rewriting the spectral
problem as a first order system with the eigenvalue as a parameter.
Then, applying a general theorem for convergence of approximate flows (cf. \cite{PZ1}),
we show that the spectral stability for $\tau = 0$ persists for $\tau \sim 0^+$.
By a continuation argument, we extend the result to the whole parameter domain $\tau \in (0,\tau_m)$.

\subsection{Reformulation of the spectral problem}\label{secfirstorder}

Set $\tau \in (0,\tau_m)$.
Component-wise the spectral problem \eqref{spectproblem} can be written as
\begin{equation}\label{evalue}
	\begin{aligned}
		c{u_x} + {v_x} +(a(x)-\lambda) u &= 0,\\
		{u_x} + c \tau {v_x} -(1+\tau\lambda)v &= 0.
	\end{aligned}
\end{equation}
Just like for the original nonlinear Allen-Cahn model with relaxation \eqref{relAC}, the $v$ variable can be
eliminated to obtain a second order equation for $u$ (a spectral version of Kac's trick):
multiply the first equation by $c\tau$, substract it from the second and differentiate with respect to $x$.
The result is the following second order spectral equation,
\begin{equation}\label{2ndorderu}
	(1-c^2\tau){u_{xx}} + c \big( 1+\tau(2\lambda - a) \big) {u_x}
		+ \big( (1+\tau \lambda)(a(x)-\lambda) - c\tau a'(x) \big) u = 0.
\end{equation}
Following Alexander, Gardner and Jones \cite{AGJ90}, we recast the scalar spectral
equation \eqref{2ndorderu} as a first order system of the form
\begin{equation}\label{firstorders}
	\mathbf{w}_x = \At (x,\lambda)\mathbf{w},
\end{equation}
where $\mathbf{w}=(u,u_x)^\top$ and, for $a=a(x)$ and $a'=a'(x)$,
\begin{equation}\label{coeffA}
	\At (x,\lambda) := \frac{1}{1 - c^2 \tau}
	\begin{pmatrix} 0 & & 1-c^2\tau \\
				c\tau a' + (1+\tau\lambda)(\lambda - a) &  & \;c(\tau a -1-2\tau\lambda) 
				\end{pmatrix}.
\end{equation}
Observe that the coefficient matrices \eqref{coeffA} can be written as
\begin{equation*}
	\At (x,\lambda) = \At_0(x) + \lambda \At_1(x) + \lambda^2 \At_2(x),
\end{equation*}
where
\begin{equation*}
\begin{aligned}
	\At_0(x) &:= \At(x,0) = \frac{1}{1 - c^2 \tau}\begin{pmatrix} 0 & & 1-c^2\tau \\ c\tau a'(x) - a(x) &  & \; c(\tau a(x) - 1) 
                                       \end{pmatrix},\\
 	\At_1(x) &:=  \frac{1}{1 - c^2 \tau}\begin{pmatrix} 0 & \;0 \\ 1 - \tau a(x) &  \;-2c\tau  
                                       \end{pmatrix}, \\
	\At_2(x) &:= \frac{\tau}{1 - c^2 \tau} \begin{pmatrix}
                                      0 & \;0 \\ 1 & \; 0
                                     \end{pmatrix}.
\end{aligned}
\end{equation*}
Since $c = c(\tau)$ is a continuous function of $\tau \in [0,\tau_m)$,
we conclude that $\At (\cdot,\lambda)$ is a function from $(\tau, \lambda) \in [0,\tau_m) \times \C$
to $L^\infty(\R;\R^{2\times 2})$, continuous in $\tau$ and analytic
(second order polynomial) in $\lambda$.

It is a well-known fact (see \cite{San02,AGJ90} and the references therein) that an alternate but
equivalent definition of the spectra and the resolvent sets associated to the spectral problem
\eqref{spectproblem} can be expressed in terms of the first order systems \eqref{firstorders}.
Consider the following family of linear, closed, densely defined operators
\begin{equation}
\label{defofTau}
 \begin{aligned}
  \mathcal{T}^\tau (\lambda) &: \overline{D} =  H^1(\R;\C^2) \to L^2(\R;\C^2),\\
\mathcal{T}^\tau \mathbf{w} &:= \mathbf{w}_x - \A^{\tau}(x,\lambda)\mathbf{w}, \qquad \mathbf{w} \in H^1(\R;\C^2),
 \end{aligned}
\end{equation}
indexed by $\tau \in (0,\tau_m)$, parametrized by $\lambda \in \C$ and with domain
$\overline{D} = H^1(\R;\C^2)$, which is independent of $\lambda$ and $\tau$.
With a slight abuse of notation we call $\mathbf{w} \in H^1(\R;\C^2)$ an \textit{eigenfunction} associated
to the eigenvalue $\lambda \in \C$ provided $\mathbf{w}$ is a bounded solution to the equation
\begin{equation*}
 \mathcal{T}^\tau (\lambda) \mathbf{w} = 0.
\end{equation*}
More precisely, for each $\tau \in [0,\tau_m)$ we define the resolvent $\rho$,
the point spectrum $\ptsp$ and the essential spectrum $\ess$ of problem \eqref{evalue} as
\begin{equation*}
	\begin{aligned}
		\rho &= \{\lambda \in \C \, : \, \mathcal{T}^\tau (\lambda) \,\text{ is injective and onto, 
and }
			\mathcal{T}^\tau (\lambda)^{-1} \, \text{is bounded} \, \},\\
		\ptsp &= \{ \lambda \in \C \,: \; \mathcal{T}^\tau (\lambda) \,\text{ is Fredholm with index zero
			and non-trivial kernel} \},\\
		\ess &= \{ \lambda \in \C \,: \; \mathcal{T}^\tau (\lambda) \,\text{ is either not Fredholm or
			with non-zero index} \},
\end{aligned}
\end{equation*}
respectively.
The whole spectrum is $\sigma = \ess \cup \ptsp$. Since each operator $\mathcal{T}^\tau (\lambda)$
is closed, then $\rho = \C \backslash \sigma$ (cf. \cite{Kat80}).
When $\lambda \in \ptsp$ we call it an eigenvalue, and any element in $\ker \mathcal{T}^\tau(\lambda)$ is an eigenvector.

We call the reader's attention to the fact that, unlike equation \eqref{spectproblem},
the spectral problem formulated in terms of the first order systems \eqref{firstorders}
is well defined for $\tau = 0$, with
\begin{equation*}
 	\A^0(x,\lambda) = \begin{pmatrix} 0 & & 1 \\ \lambda - a(x) & & -c_0 \end{pmatrix},
\end{equation*}
and where $c_0 = c(\tau)\bigr|_{\tau = 0}$ is the velocity of the parabolic front.

\begin{remark}
\label{remaux}
Suppose $\tau \in (0,\tau_m)$. If we substitute $\lambda = -1/\tau$ into \eqref{evalue} we arrive at the equation 
\begin{equation*}
 	(1-c^2 \tau)u_x - c(1+\tau a(x))u = 0.
\end{equation*}
Taking the real part of the $L^2$ product of last equation with $u$ we obtain 
\begin{equation*}
 	0 = \int_\R (\underbrace{1+\tau a(x))}_{> 0})|u|^2 \, dx \geq 0,
\end{equation*}
inasmuch as $\Real \langle u,u_x\rangle_{L^2} = 0$, and $\tau < \tau_m =1/\sup |f'|$.
This implies that $u = 0$, and hence $v = 0$, showing that $\lambda = -1/\tau$ does not belong to the point spectrum.
\end{remark}

\subsubsection{On algebraic and geometric multiplicities}

In the stability of traveling waves literature, it is customary to analyze
the spectrum of a differential operator $\mathcal{L}$ of second (or higher) order, for which there is a
natural invertible transformation from the kernel of $\mathcal{L}-\lambda$ to the kernel of a first order
operator of the form \eqref{defofTau}.
In such cases, the matrices \eqref{coeffA} are linear in $\lambda$ and there is a natural
correspondence between the Jordan block structures of $\mathcal{L}-\lambda$ and those of the
corresponding operators $\mathcal{T} (\lambda)$ (see \cite{San02}).
Since we arrived at systems \eqref{firstorders} through a different transformation (Kac's trick),
we must show that such property remains in our case.
For each $\tau \in (0,\tau_m)$ and $\lambda \in \ptsp$, we call the mapping,
\begin{equation}\label{specKac}
	\begin{aligned}
 	\mathcal{K} &: \ker (\mathcal{L}^\tau - \lambda) 
\, \to \ker \mathcal{T}^\tau (\lambda), \\ 
	\mathcal{K} \begin{pmatrix} 
	             u \\ v
	            \end{pmatrix}
&:= \begin{pmatrix} 
	             u \\ u_x
	            \end{pmatrix} = \mathbf{w}, \qquad \; \begin{pmatrix} 
	             u \\ v
	            \end{pmatrix} \in \ker (\mathcal{L}^\tau - \lambda),
	\end{aligned}
\end{equation}
as the \textit{spectral Kac's transformation}. It is a one-to-one and onto map.
Indeed, if $\mathbf{w}_1 = \mathbf{w}_2 \in \ker \mathcal{T}^\tau (\lambda)$ then $u_1 = u_2$ and
$\partial_x u_1 = \partial_x u_2$ a.e., and the first equation in \eqref{evalue} yields
$\partial_x v_1 = \partial_x v_2$, whereas the second equation implies $v_1 = v_2$ a.e.
Thus, $(u_1, v_1)^\top = (u_2, v_2)^\top \in \ker (\mathcal{L}^\tau - \lambda)$.
It is onto because for each $\mathbf{w} = (u, u_x)^\top \in \ker \mathcal{T}^\tau (\lambda)$ clearly
there exists 
\begin{equation*}
\begin{pmatrix}
 u \\ v
\end{pmatrix}
=	\begin{pmatrix}
 u \\ (1+\tau \lambda)^{-1} \big( (1-c^2\tau)u_x + c\tau (\lambda - a(x)) u\big) 
\end{pmatrix} \in \ker (\mathcal{L}^\tau - \lambda),
\end{equation*}
such that $\mathbf{w} = \mathcal{K}(u,v)^\top$.
(Provided, of course, that $\lambda \neq -1/\tau$. But $-1/\tau \notin \ptsp$ by Remark \ref{remaux}.)
Since $v$ satisfies the first equation of \eqref{evalue} we conclude that $v \in H^2(\R;\C)$ as well.

\begin{proposition}\label{prop:propequiv}
Spectral Kac's transformation \eqref{specKac} induces a one-to-one correspondence between Jordan chains.
\end{proposition}

\begin{proof}
Suppose $(\varphi, \psi)^\top \in \ker (\mathcal{L}^\tau - \lambda)$.
This is equivalent to the system
\begin{equation}\label{e0}
	c{\varphi_x} + {\psi_x} +(a(x)-\lambda) \varphi = 0,\qquad
	{\varphi_x} + c \tau {\psi_x} -(1+\tau\lambda)\psi = 0.
\end{equation}
If we take the next element in a Jordan chain, say $(u,v)^\top \in H^2(\R;\C^2)$, solution to
\begin{equation*}
 (\mathcal{L}^\tau - \lambda)\begin{pmatrix}
                              u \\ v
                             \end{pmatrix} = \begin{pmatrix}
                              \varphi \\ \psi
                             \end{pmatrix},
\end{equation*}
then we obtain the system
\begin{equation}\label{e1}
	cu_x + v_x +(a(x)-\lambda) u = \varphi,\qquad
	u_x + c \tau v_x -(1+\tau\lambda)v = \tau \psi.
\end{equation}
Multiply the first equation by $c\tau$, substract from the second one, differentiate it,
and substitute $v_x$ from the first equation and $\psi_x$ from the second in \eqref{e0}.
The result is the following scalar equation
\begin{equation}\label{eq2}
	\begin{aligned}
	&(1-c^2\tau){u_{xx}} + c\big( 1+\tau(2\lambda - a(x))\big){u_x}\\
	&\;\;\; +\bigl\{(1+\tau \lambda)(a(x)-\lambda) - c\tau a'(x)\bigr\} u
		= (1+ 2\tau \lambda - \tau a(x)) \varphi - 2 c\tau \varphi_x.
	\end{aligned}
\end{equation}
Written as a system for $\mathbf{w}_1 := (u, u_x)^\top$, equation \eqref{eq2} is equivalent to
\begin{equation*}
 \partial_ x \mathbf{w}_1 - \At(x,\lambda) \mathbf{w}_1 = \big(\At_1(x) + 2 \lambda \At_2(x)\big) \begin{pmatrix}
                                                                         \varphi \\ \varphi_x
                                                                        \end{pmatrix}.
\end{equation*}
Generalizing this procedure, we observe that solutions to 
\begin{equation*}
 (\mathcal{L}^\tau - \lambda)\begin{pmatrix}
                              u_j \\ v_j
                             \end{pmatrix}
= \begin{pmatrix}
                u_{j-1} \\ v_{j-1}              
                             \end{pmatrix},
\end{equation*}
are in one-to-one correspondence with solutions to the equation
$\mathcal{T}^\tau (\lambda) \mathbf{w}_j = (\partial_\lambda \At(x,\lambda))\mathbf{w}_{j-1}$,
where $w_j$ and $(u_j, v_j)^\top$ are related to each other through Kac's transformation.
Hence, a Jordan chain for $\mathcal{L}^\tau-\lambda$ induces a Jordan chain for
$\mathcal{T}^\tau (\lambda)$ with same block structure and length. 
\end{proof}

Consequently, we have the following definition. 

\begin{definition}\label{def:spectT}
Assume $\lambda \in \ptsp$. Its geometric multiplicity (\textit{g.m.}) is the maximal number of
linearly independent elements in $\ker \mathcal{T}^\tau (\lambda)$.
Suppose $\lambda \in \ptsp$ has $g.m. = 1$, so that $\ker \mathcal{T}^\tau (\lambda) =$ span $\{\mathbf{w}_0\}$.
We say $\lambda$ has algebraic multiplicity (\textit{a.m.}) equal to $m$ if we can solve
$\mathcal{T}^\tau (\lambda) \mathbf{w}_j = (\partial_\lambda \At(x,\lambda)) \mathbf{w}_{j-1}$,
for each $j = 1, \ldots, m-1$, with $\mathbf{w}_j \in H^1$, but there is no bounded $H^1$ solution
$\mathbf{w}$ to $\mathcal{T}^\tau (\lambda) \mathbf{w} = (\partial_\lambda \At(x,\lambda)) \mathbf{w}_{m-1}$.
For an arbitrary eigenvalue $\lambda \in \ptsp$ with $g.m.= l$, the algebraic multiplicity
is defined as the sum of the multiplicities $\sum_k^l m_k$ of a maximal set of linearly
independent elements in $\ker \mathcal{T}^\tau (\lambda) = $ span $\{\mathbf{w}_1, \ldots, \mathbf{w}_l\}$.
\end{definition}

Thanks to Proposition \ref{prop:propequiv} and Definition \ref{def:spectT} we readily obtain the following

\begin{corollary}
For each $\tau \in [0,\tau_m)$, the sets $\ptsp$ and $\ptsp(\mathcal{L}^\tau)$ (the latter defined
as the set of complex $\lambda$ such that $\mathcal{L}^\tau-\lambda$ is Fredholm with index
zero and has a non-trivial kernel) coincide, with same algebraic and geometric multiplicities. 
\end{corollary}

\subsection{The (translation invariance) eigenvalue $\lambda=0$}

Here, we prove that $\lambda = 0$ is an eigenvalue of $\mathcal{L}^\tau$ for each $\tau \in (0,\tau_m)$
(the eigenvalue associated to translation invariance of the wave, with eigenfunction $(U_x,V_x)^\top$),
and, moreover, that it is simple.

\begin{lemma}\label{lemgm1}
For each $\tau \in (0,\tau_m)$, $\lambda = 0$ is an eigenvalue of $\mathcal{L}^\tau$
with geometric multiplicity equal to one, and with eigenspace generated by $(U_x, V_x)^\top \in H^2(\R;\C^2)$.
\end{lemma}

\begin{proof}
Differentiate system \eqref{profileq} with respect to $x$ to verify that $(U_x, V_x)$ is a solution
to the spectral system  \eqref{evalue} with $\lambda = 0$, that is,
\begin{equation}\label{lam0syst}
	cU_{xx} + V_x + a(x) V_x = 0, \qquad U_x + c\tau V_{xx} - V_x = 0.
\end{equation}
Due to exponential decay \eqref{expodecay} of the wave, and solving for $U_{xx}$ and $V_{xx}$
in equations \eqref{lam0syst}, we observe that $(U_x, V_x)^\top \in H^2(\R;\C^2) = D(\mathcal{L}^\tau)$.
This shows that $(U_x, V_x)^\top$ belongs to $\ker \mathcal{L}^\tau$ for each $\tau \in (0,\tau_m)$.
Thus, $\lambda = 0 \in \ptsp(\mathcal{L}^\tau)$.
In view of the equivalence established by Kac's spectral transformation (Proposition \ref{prop:propequiv}), this implies that
\begin{equation*}
	\mathbf{w}_0 = \begin{pmatrix}
                 U_x \\ U_{xx}
                \end{pmatrix} \in \ker \mathcal{T}^\tau.
\end{equation*}
To analyze its multiplicity we observe that system \eqref{evalue} with $\lambda = 0$ is equivalent to
the following scalar equation (substitute $\lambda = 0$ in \eqref{2ndorderu}):
\begin{equation}\label{eq:mult01}
	\mathcal{A}u := a_0 u_{xx} + a_1(x) u_x + a_2(x) u = 0,
\end{equation}
where we have introduced the closed, densely defined auxiliary operator
$\mathcal{A} : D(\mathcal{A}) = H^2(\R;\C) \to L^2(\R;\C)$, where $a(x) = f'(U(x))$ as before, and with
\begin{equation*}
	a_0 = 1-c^2 \tau > 0, \quad a_1(x) = c(1-\tau a(x)), \quad a_2(x) = a(x) - c\tau a'(x).
\end{equation*}
Let us denote
\begin{equation*}
	\phi := U_x \in H^2(\R;\C).
\end{equation*}
Clearly, $\phi$ is a bounded solution to \eqref{eq:mult01}, and $\lambda = 0$ is an eigenvalue
of $\mathcal{A}$ associated to the eigenfunction $\phi$.

We shall rewrite \eqref{eq:mult01} in self-adjoint form by eliminating the first derivative.
To this aim, we introduce the new variable $w$ as follows:
\begin{equation}\label{eq:changeuw}
	u(x) = z(x) w(x), \qquad z(x) = \exp \Big(\int^x b(y) \, dy \Big).
\end{equation}
A direct calculation shows that
\begin{equation*}
	u_x = (w_x + wb)z, \qquad u_{xx} = (w_{xx} + 2bw_x + (b_x + b^2)w)z.
\end{equation*}
Upon substitution,
\begin{equation*}
	\mathcal{A}u = \big( a_0 w_{xx} + (a_1(x) + 2a_0b) w_x + (a_0(b_x + b^2) + a_1(x) b + a_2(x))w\big)z.
\end{equation*}
Choose $b = - a_1(x)/2a_0$. This yields
\begin{equation*}
	z_x = - \frac{a_1(x)}{2a_0} z, \qquad z(x) = \exp \Big(- \int^x a_1(y)/2a_0 \, dy \Big),
\end{equation*}
and
\begin{equation*}
	\mathcal{A} u = \big( \widetilde{\mathcal{A}} w\big) z = 0,
\end{equation*}
where the self-adjoint operator $\widetilde{\mathcal{A}} : H^2(\R;\C) \to L^2(\R;\C)$ is defined as
\begin{equation*}
	\widetilde{\mathcal{A}} w := a_0 w_{xx} + h(x) w, \quad
	h(x) = a_2(x) - \tfrac{1}{2}a_1'(x) - \tfrac{1}{4}a_0^{-1} a_1(x)^2.
\end{equation*}
Since $z(x) > 0$ for all $x$, this readily implies that $\mathcal{A} u = 0$
if and only if $w = uz^{-1} \in \ker \widetilde{\mathcal{A}}$. 

Upon inspection, we observe that any decaying solution (at $x = \pm \infty$)
to $\widetilde{\mathcal{A}} w = 0$ converges to zero exponentially with rate
\begin{equation}\label{eq:decayw}
	\mp \frac{1}{2a_0}\sqrt{(a_1^{\pm})^2 - 4a_0f'(U^{\pm})} \, .
\end{equation}
This is true, in particular, for $\varphi = \phi/z$, because of the behavior at $\pm \infty$
of the weight function $z$.
In view of these observations, we conclude that there is a one-to-one correspondence between
$L^2$ eigenfunctions of $\mathcal{A}$ and $\widetilde{\mathcal{A}}$,
determined by the change of variables \eqref{eq:changeuw}.

Now suppose that $u \in H^2$ is a solution to $\mathcal{A} u = 0$.
Therefore, it decays at $x = \pm \infty$ with rate
\begin{equation*}
	- \frac{-a_1^\pm}{2a_0} \mp \frac{1}{2a_0}\sqrt{(a_1^{\pm})^2 - 4a_0f'(U^{\pm})} \, .
\end{equation*}
This implies that $w = uz^{-1}$ is an $L^2$ solution to $\widetilde{\mathcal{A}} w = 0$ which
decays with rate \eqref{eq:decayw}.
Hence, the Wronskian determinant of these two solutions, viz. $w \varphi_x - \varphi w_x$
goes to zero when $x \to \pm \infty$ and, moreover,
\begin{equation*}
	a_0(w \varphi_x - \varphi w_x)_x = a_0(w \varphi_{xx} - \varphi w_{xx}) = h(x)w\varphi - h(x)\varphi w = 0,
\end{equation*}
implying that $w \varphi_x - \varphi w_x = 0$ for all $x \in \R$.
Therefore, $w$ and $\varphi$ (and hence, $u$ and $\phi$) are linearly dependent.

By the equivalence between solutions to the scalar equation \eqref{2ndorderu} and solutions to the
first order system \eqref{firstorders}, we have shown that any bounded solution $\mathbf{w}$ to
$\mathbf{w}_x = \A^\tau(x,0) \mathbf{w}$ is a multiple of $\mathbf{w}_0$, and consequently
$\mathcal{T}^\tau(0)$ is Fredholm with index zero, with non-trivial kernel spanned by $\mathbf{w_0}$.
Whence, $\lambda = 0 \in \ptsp$ with geometric multiplicity equal to one.
By Proposition \ref{prop:propequiv}, this implies, in turn, that $\lambda = 0 \in \ptsp(\mathcal{L}^\tau)$
with geometric multiplicity equal to one, and with associated eigenfunction $(U_x, V_x)^\top$.
\end{proof}

\begin{corollary}
	The adjoint equation 
\begin{equation}\label{adjeq}
\mathbf{y}_x = - \A^\tau(x,0)^* \mathbf{y},
	\end{equation}
has a unique (up to constant multiples) bounded solution 
$\mathbf{y}_0 = (\zeta, \eta )^\top \in H^1(\R;\C^2)$
where $\eta \in H^2(\R;\C)$ is the unique bounded solution to
\begin{equation}\label{eqeta}
	\mathcal{A}^* \eta = a_0 \eta_{xx} - a_1(x) \eta_x + (a_2(x) - a_1'(x)) \eta = 0, 
\end{equation}
and $\mathcal{A}^* : H^2(\R;\C) \to L^2(\R;\C)$ denotes the formal adjoint of the auxiliary operator $\mathcal{A}$.
\end{corollary}

\begin{proof}
Since $\mathcal{T}^\tau(0)$ is Fredholm with index zero and
$\ker \mathcal{T}^\tau(0) = \mathrm{span} \{\mathbf{w}_0\}$, by an exponential dichotomies argument
(see Remark 3.4 in \cite{San02}), the adjoint equation \eqref{adjeq} has a unique bounded solution
$\mathbf{y}_0 \in H^1(\R;\C^2)$. Observing that
\begin{equation*}
	\begin{aligned}
 	- \A^\tau(x,0)^* = -\A_0^\tau(x)^\top &= (1-c^2\tau)^{-1} \begin{pmatrix}
                                                        0 & -(1-c^2 \tau) \\ a(x) - c\tau a'(x) & c(1-\tau a(x))
                                                       \end{pmatrix}^\top \\&= a_0^{-1}\begin{pmatrix}
                                                        0 & a_2(x) \\ - a_0 & a_1(x)
                                                       \end{pmatrix},
	\end{aligned}
\end{equation*}
we arrive at the following system of equations for the components of $\mathbf{y}_0$:
\begin{equation*}
	\zeta_x = a_0^{-1} a_2(x) \eta, \qquad \eta_x = - \zeta + a_0^{-1} a_1(x) \eta.
\end{equation*}
Since the coefficients are bounded, by bootstrapping it is easy to verify that $\eta \in H^2(\R;\C)$.
Thus, differentiate the second equation and substitute the first to arrive at
\begin{equation*}
	a_0 \eta_{xx} - a_1(x) \eta_x + (a_2(x) - a_1'(x)) \eta = 0.
\end{equation*}
That the left hand side of last equation is $\mathcal{A}^* \eta$ follows from a direct calculation
of the formal adjoint.
\end{proof}

\begin{lemma}\label{lemam1}
For each $\tau \in (0,\tau_m)$ the algebraic multiplicity of $\lambda = 0 \in \ptsp(\mathcal{L}^\tau)$ is equal to one.
\end{lemma}

\begin{proof}
We define the quantity
\begin{equation}\label{defGamma}
\Gamma := \langle \mathbf{y}_0, \A^\tau_1(x) \mathbf{w}_0 \rangle_{L^2} = \int_{-\infty}^{+\infty} \begin{pmatrix}
                                                                                                    \zeta \\ \eta
                                                                                                   \end{pmatrix}^* \A_1^\tau(x) \begin{pmatrix}
                                                                                                    \phi \\ \phi_x
                                                                                                   \end{pmatrix} \, dx.
\end{equation}
Substituting the expression for $\A_1^\tau$ we obtain
\begin{equation*}
	\Gamma = a_0^{-1} \int_{-\infty}^{+\infty} \overline{\eta} \big( (1-\tau a(x))\phi - 2c\tau \phi_x \big) \, dx.
\end{equation*}
Let us suppose that $(u_1, v_1)^\top \in H^1(\R;\C)$, $(u_1, v_1) \neq 0$, is a non-trivial first element of a Jordan chain for $\mathcal{L}^\tau$ associated to $\lambda = 0$, that is, a solution to 
\begin{equation*}
 \mathcal{L}^{\tau} \begin{pmatrix}
                     u_1 \\ v_1
                    \end{pmatrix} = \begin{pmatrix} U_x \\ V_x
\end{pmatrix}
 \in \ker \mathcal{L}^\tau.
\end{equation*}
Hence, $(u_1, v_1)^\top$ is a solution to system \eqref{e1}  with $\lambda = 0$.
By substitution, this is equivalent to equation \eqref{eq2} for $u_1$ and with $\lambda = 0$, namely to
\begin{equation*}
 \mathcal{A}u_1 = (1-\tau a(x))U_x - 2c\tau U_{xx} = (1-\tau a(x))\phi - 2c\tau \phi_x.
\end{equation*}
Apply the change of variables \eqref{eq:changeuw} to last equation to obtain 
\begin{equation*}
 \mathcal{A}u_1 = z (\widetilde{\mathcal{A}}w_1) = (1-\tau a(x))\phi - 2c\tau \phi_x,
\end{equation*}
where $w_1 = u_1 z^{-1}$.
Now, since $z$ is real and $\widetilde{\mathcal{A}}$ is self-adjoint, we obtain 
\begin{equation*}
	\Gamma = a_0^{-1}  \int_{-\infty}^{+\infty} \overline{z \eta}  \widetilde{\mathcal{A}} w_1 \, dx \\
			= a_0^{-1} \langle z\eta, \widetilde{\mathcal{A}} w_1 \rangle_{L^2}
			= a_0^{-1} \langle \widetilde{\mathcal{A}}(z\eta),  w_1 \rangle_{L^2} \, .
\end{equation*} 
Use equation \eqref{eqeta} to compute 
\begin{equation*}
 (z\eta)_{xx} = - a_0^{-1} h(x) z \eta,
\end{equation*}
yielding $\widetilde{\mathcal{A}}(z\eta) = a_0 (z\eta)_{xx} + h(x)z \eta = 0$. We conclude that $\Gamma = 0$.

Therefore, the contrapositive statement holds true: if $\Gamma \neq 0$ then there exists no non-trivial first 
element
of the Jordan chain. In other words, if we show that $\Gamma \neq 0$ then the algebraic multiplicity
of $\lambda = 0$ is equal to one.
To compute $\Gamma$ we make the observation that the unique bounded solution to
$\mathcal{A}^* \eta = 0$ is precisely $\eta = z^{-2} \phi$.
Indeed, by a direct calculation we get
\begin{equation*}
	\begin{aligned}
	 \eta_x & = \frac{1}{z^2}(\phi_x + a_0^{-1}a_1(x) \phi), \qquad \text{and,}\\
	\eta_{xx} & = \frac{a_0^{-1}a_1(x)}{z^2}(\phi_x + a_0^{-1}a_1(x) \phi)
		+ \frac{1}{z^2}(\phi_{xx} + a_0^{-1} a_1'(x) \phi + a_0^{-1}a_1(x) \phi_x).
	\end{aligned}
\end{equation*}
This yields
\begin{equation*}
 \mathcal{A}^* \eta = \frac{1}{z^2} (a_0 \phi_{xx} + a_1(x) \phi_x + a_2(x) \phi) = \frac{1}{z^2} \mathcal{A}\phi = 0.
\end{equation*}
Substituting in the expression for $\Gamma$, and since $\phi$ is real, after integration by parts one gets
\begin{equation*}
 \begin{aligned}
  \Gamma &= a_0^{-1} \int_{-\infty}^{+\infty} z^{-2} \phi ((1-\tau a(x))\phi - 2c\tau \phi_x) \, dx\\
&= a_0^{-1} \int_{-\infty}^{+\infty} z^{-2} |\phi|^2 (1-\tau a(x)) \, dx - 2c\tau a_0^{-1} \int_{-\infty}^{+\infty} z^{-3}z_x |\phi|^2 \, dx\\
&= a_0^{-1} (1 + c^2 \tau a_0^{-1}) \int_{-\infty}^{+\infty} |1-\tau a(x)| z^{-2}|\phi|^2 \, dx,
 \end{aligned}
\end{equation*}
because $1 - \tau a(x) > 0$ for all $x$ if $\tau \in (0,\tau_m)$.
Since $\phi \not\equiv 0$ a.e. we conclude that $\Gamma > 0$ and the proof is complete.
\end{proof}

\begin{remark}
It is well known that the integral 
\begin{equation*}
\Gamma = \langle \mathbf{y}_0, ((d/d\lambda) \A^\tau(x,\lambda))_{|\lambda = 0} \mathbf{w}_0\rangle_{L^2}, 
\end{equation*}
known as a \textit{Melnikov integral}, decides whether $\lambda = 0$ has higher algebraic multiplicity:
$\Gamma$ is proportional, modulo a non-vanishing orientation factor, to the derivative of the Evans
function $D'(0)$ at $\lambda = 0$ (see Definition \eqref{defD} below);
thus, if $\Gamma \neq 0$ then the algebraic multiplicity is equal to one. See \cite{San02},
Section 4.2.1 for further information.
\end{remark}

\subsection{Further properties of the point spectrum}

Next, we use energy estimates to show that there are no 
purely imaginary eigenvalues different from zero.

\begin{lemma}\label{lemma:imagevalue}
If $\lambda$ is an eigenvalue for \eqref{evalue} and $\lambda\in i\R$, then $\lambda=0$.
\end{lemma}

\begin{proof}
Let $\lambda\in i\R$ be such that equation \eqref{2ndorderu} is satisfied for some function $u$.
Applying the transformation \eqref{eq:changeuw} of the proof of Lemma \ref{lemgm1}, i.e. $u(x)=w(x)\,z(x)$ with $z(x)=\exp(-\int^x b)$ and $b=-a_1/2a_0$, equation is transformed into
\begin{equation*}
 {w_{xx}}+\alpha\lambda\,{w_x}-\beta(x,\lambda)w=0,
\end{equation*}
with
\begin{equation*}
	\begin{aligned}
	\alpha&=\frac{2c\tau}{1-c^2\tau},\\
	\beta(x,\lambda)&=\frac{1}{4a_0^2}\left(a_1(x)^2 -4a_0a_2(x) -2a_0 a_1'(x) \right)
		+\frac{(1-\tau\,a(x))}{a_0^2}\lambda+\frac{\tau}{a_0}\,\lambda^2.
	\end{aligned}
\end{equation*}
Multiplying by $\bar w$, we infer the relation
\begin{equation*}
	\left|w_x\right|^2-\left(\bar w w_x\right)_x
		+ \alpha\lambda \bar w w_x +\beta(x,\lambda)|w|^2=0.
\end{equation*}
Thus, integrating in $\R$ and taking the imaginary part one obtains
\begin{equation*}
	\int_{\R} \Imag\lambda\bigl\{2c\tau(1-c^2\tau)\Real\left({w_x}\bar w\right)\,
		+\big(1-\tau\,a(x)\big)|w|^2\bigr\}dx=0.
\end{equation*}
since, by assumption, $\lambda\in i\R$.
For $\lambda\neq 0$, thanks to the relation
\begin{equation*}
	\Real (\bar w w_x) = \tfrac{1}{2} \big( |w|^2\big)_x,
\end{equation*}
the previous equality implies $w=0$ a.e. since $1-\tau a(x)>0$.
\end{proof}

\subsection{Consistent splitting and essential spectrum}

Let us look at the asymptotic (constant coefficient) systems derived from \eqref{coeffA}
when we take the limit as $x \to \pm \infty$.
If we define the positive parameters
\begin{equation*}
	\begin{aligned}
	0 &< \delta_\pm := - \lim_{x \to \pm \infty} a(x) = - \lim_{x \to \pm \infty} f'(U) = - f'(U_\pm)\\
	0 &< b_\pm := \lim_{x \to \pm \infty} \bigl(1 - \tau f'(U)\bigr) = 1 + \tau \delta_\pm,
	\end{aligned}
\end{equation*}
for each $\tau \in (0,\tau_m)$, then systems \eqref{firstorders} tend to the constant coefficient asymptotic systems
\begin{equation}\label{constcoeff}
	\mathbf{w}_x = \A^\tau_\pm (\lambda) \mathbf{w},
\end{equation}
where
\begin{equation}\label{asymptcoeffA}
	\begin{aligned}
	\At_\pm(\lambda) &:= \lim_{x \to \pm \infty} \At(x,\lambda) \\
			&= (1-c^2 \tau)^{-1}\begin{pmatrix}
				0 & & 1-c^2 \tau \\ \tau \lambda^2 + \lambda b_\pm + \delta_\pm &  &-c(b_\pm + 2 \tau \lambda)
				\end{pmatrix}.
	\end{aligned}
\end{equation}
The location of the essential spectrum of problem \eqref{evalue} is determined by
systems \eqref{constcoeff}.
Let us denote the characteristic polynomial of $\At_\pm(\lambda)$ as
\begin{equation}
\label{defcharpol}
	\pi_\pm^{(\tau, \lambda)}(\kappa) := \det ( \At_\pm(\lambda) - \kappa I).
\end{equation}
Thus, we compute
\begin{equation*}
 	\begin{aligned}
	&\det (\kappa I - (1-c^2 \tau)  \At_\pm(\lambda))
	= \det \begin{pmatrix} \kappa & & -(1-c^2 \tau) \\
		-(\tau \lambda^2 + \lambda b_\pm + \delta_\pm) &  &\kappa -c(b_\pm + 2 \tau \lambda)
		\end{pmatrix}\\
	&\hskip3.25cm 
		= \kappa^2 + \kappa c (b_\pm + 2 \tau \lambda) - (1-c^2 \tau)(\tau \lambda^2 + \lambda b_\pm + \delta_\pm).
	\end{aligned}
\end{equation*}
Hence, if we assume $\kappa = i \xi$, $\xi \in \R$, is a purely imaginary root of \eqref{defcharpol}, then 
\begin{equation}\label{disprel}
	\xi^2 - ic\xi (b_\pm + 2 \tau \lambda) + (1-c^2\tau)(\tau \lambda^2 + b_\pm \lambda + \delta_\pm) = 0.
\end{equation}
Equation \eqref{disprel} is the dispersion relation of wave solutions to the constant coefficient
equations \eqref{constcoeff}. Its $\lambda$-roots, functions of $\xi \in \R$, define algebraic curves
in the complex plane.
They bound the essential spectrum on the right as we shall verify. We denote these curves as
\begin{equation}\label{algcurves}
	\lambda = \lambda_{1,2}^\pm(\xi), \qquad \xi \in \R. 
\end{equation}
It is to be noticed that $\lambda = 0$ does not belong to any of the algebraic curves \eqref{algcurves} inasmuch as
\begin{equation*}
 	\Real (\xi^2 - ic\xi b_\pm + (1-c^2 \tau) \delta_\pm) = \xi^2 + (1-c^2 \tau) \delta_\pm > 0,
\end{equation*}
for all $\xi \in \R$.

\subsubsection{Analysis of the dispersion relation}

Fix $0 < \tau < \tau_m$. Suppose that $\lambda(\xi)$ belongs to one of the algebraic
curves \eqref{algcurves} and denote $\eta := \Real \lambda(\xi)$ and $\beta := \Imag \lambda(\xi)$.
Taking the real and imaginary parts of \eqref{disprel} yields
\begin{equation}\label{realp}
	\xi^2 + 2c\tau \xi \beta + (1-c^2\tau)\big( \tau(\eta^2 - \beta^2) + \eta b_\pm + \delta_\pm\big) = 0.
\end{equation}
\begin{equation}\label{imagp}
	\big((1-c^2 \tau) \beta - c\xi \big) \big( b_\pm + 2 \tau \eta\big) = 0.
\end{equation}
We readily notice that if $\eta = 0$ for some $\xi \in \R$ then from equation \eqref{imagp}
we get $\beta = c\xi/(1-c^2 \tau)$, as $b_\pm > 0$.
Upon substitution in \eqref{realp} we obtain
\begin{equation*}
	\xi^2 + \frac{c^2 \tau \xi^2}{1-c^2 \tau} + (1-c^2 \tau) \delta_\pm = 0,
\end{equation*}
which yields a contradiction with $\delta_\pm > 0$, $\tau > 0$ and the subcharacteristic
condition \eqref{subchar}.
We conclude that the algebraic curves never cross the imaginary axis: they remain in either
the stable or in the unstable complex half plane.
Now, from equation \eqref{imagp} we distinguish two cases:
\begin{align}
	\textrm{either}\quad \eta &= - \frac{b_\pm}{2\tau}, \label{casei}\\
	\textrm{or}\quad \beta &= \frac{c\xi}{1-c^2\tau}. \label{caseii}
\end{align}
Let us first assume \eqref{casei}. Substituting into \eqref{realp} we obtain the equation
\begin{equation}\label{eqforbeta}
	\tau \beta^2 - \frac{2c\tau \xi}{1-c^2\tau}\,\beta - \delta_\pm
	+ \frac{b_\pm^2}{4\tau} - \frac{\xi^2}{1-c^2\tau} = 0.
\end{equation}
This equation has real solutions $\beta$ provided that
\begin{align}
	\Delta_1(\xi) &:= \frac{4c^2 \tau^2 \xi^2}{(1-c^2 \tau)^2}
		- 4\tau \Big( -\delta_\pm + \frac{b_\pm^2}{4\tau} - \frac{\xi^2}{1-c^2\tau}\Big) \geq 0, \nonumber \\
	&\hskip2cm \iff \;\; \xi^2 (1-c^2 \tau)^{-2} + \delta_\pm \geq \frac{b_\pm^2}{4 \tau}. \label{star}
\end{align}
Secondly, substitute \eqref{caseii} into \eqref{realp}.
The result is 
\begin{equation} \label{eqforeta}
	\tau \eta^2 + b_\pm \eta + \delta_\pm + \frac{\xi^2}{(1-c^2\tau)^2} = 0.
\end{equation}
Last equation has real solutions $\eta$ if and only if
\begin{align}
	\Delta_2(\xi) &:= b_\pm^2 - 4\tau \Big( \delta_\pm + \frac{\xi^2}{(1-c^2\tau)^2} \Big) \geq 0, \nonumber \\
 	&\hskip2cm \iff \;\; \xi^2(1-c^2\tau)^{-2} + \delta_\pm \leq \frac{b_\pm^2}{4\tau}. \label{dstar}
 \end{align}
Then, clearly, from \eqref{star} and \eqref{dstar} we have $\sgn \Delta_2 = - \sgn \Delta_1$.
Observe, however, that by definition,
\begin{equation*}
	\frac{b_\pm^2}{4\tau} = \frac{(1+\tau \delta_\pm)^2}{4\tau} > \delta_\pm,
\end{equation*}
because $(1-\delta_\pm \tau)^2 > 0$ for all $0 < \tau < \tau_m =1/\sup |f'|$, $\delta_\pm = |f'(1)|, |f'(0)|$.
Therefore, for small values of $|\xi|$, \eqref{dstar} holds, $\sgn \Delta_2 = +1$, and the only algebraic
curve solutions $\lambda = \lambda(\xi)$ are
\begin{equation}\label{ddstar}
	\Imag \lambda(\xi) = \beta(\xi) = \frac{c\xi}{1-c^2 \tau}, \qquad
	\Real \lambda(\xi) = \eta(\xi) = \frac{1}{2\tau} \Big( - b_\pm \pm \sqrt{\Delta_2(\xi)} \Big).
\end{equation}
Let $\xi^{(0)}_\pm$ be the positive solution to 
\begin{equation*}
	(\xi^{(0)}_\pm)^2 = (1-c^2 \tau)^2 \Big( \frac{b_\pm^2}{4\tau} - \delta_\pm\Big) > 0.
\end{equation*}
Hence, for each $\xi \in (-\xi^{(0)}_\pm,\xi^{(0)}_\pm)$, condition \eqref{dstar} holds and
the algebraic curves are determined by \eqref{ddstar}.
Observe that:
\begin{itemize}
	\item[\textbullet] $\Delta_1(\xi), \Delta_2(\xi) \to 0$, 
	\item[\textbullet] $\eta(\xi) \to - b_\pm / 2\tau$, $\, \beta(\xi) \to \pm c \xi^{(0)}_\pm/(1-c^2 \tau)$,
\end{itemize}
as $|\xi| \uparrow \xi^{(0)}_\pm$. 
This behavior guarantees the continuity of the algebraic curves at $|\xi| = \xi^{(0)}_\pm$,
as $\beta$ tends to the roots of \eqref{eqforbeta} and $\eta$ tends to \eqref{casei}.
Therefore, for $|\xi| \geq \xi^{(0)}_\pm$, $\Delta_1(\xi) \geq 0$ and the solutions
$\lambda(\xi)$ are determined uniquely by
\begin{equation}\label{dddstar}
	\Imag \lambda(\xi) = \beta(\xi) = \frac{c\xi}{1-c^2 \tau} \pm \frac{\sqrt{\Delta_1(\xi)}}{2\tau}, \qquad
	\Real \lambda(\xi) = \eta(\xi) = - \frac{b_\pm}{2\tau}.
\end{equation}
The algebraic curves \eqref{algcurves} in the case of a cubic reaction \eqref{cubicf}, with $\kappa = 1$
and for the parameter value $\alpha = 3/4$ can be found in Figure \ref{figalgcurves}.
To compute them, we approximated the value of the speed $c$ by its lower bound \eqref{estbelow}.

\begin{figure}\centering
\includegraphics[scale=.7, clip=true]{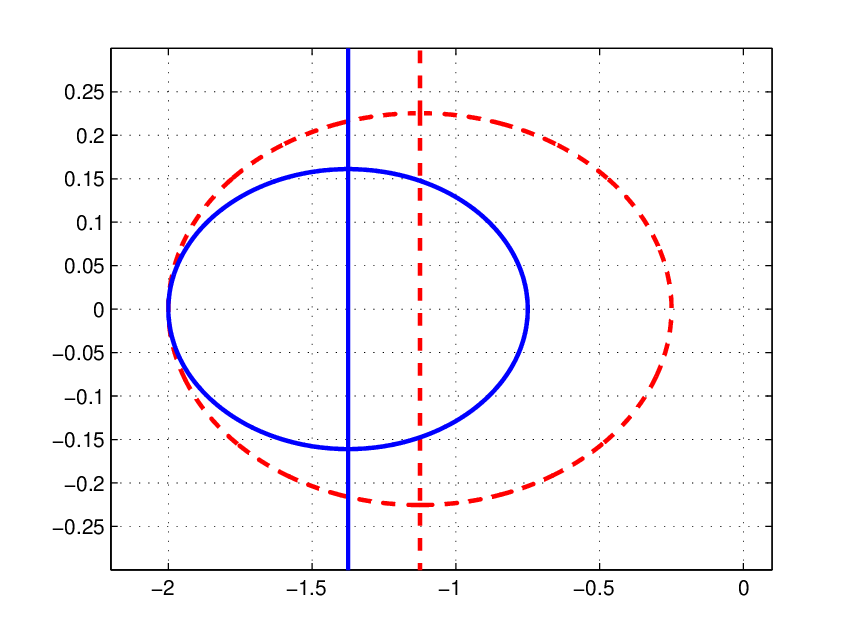}
\caption{Algebraic Fredholm curves \eqref{algcurves} for systems \eqref{constcoeff} in the case of a cubic
nonlinearity \eqref{cubicf} with $\kappa = 1$, $\tau = 1/2$, and unstable state $u = \alpha = 3/4$. 
The value of the speed $c = c(\tau)$ is approximated by its lower bound \eqref{estbelow}.
The curves at $+\infty$, $\lambda_{1,2}^+(\xi)$ are depicted by the solid continuous
(blue; online version) curves, whereas the curves at $-\infty$, $\lambda_{1,2}^-(\xi)$, are represented
by the dashed (red; online version) curves.}
\label{figalgcurves}
\end{figure}

Finally, notice that, for $|\xi| \leq \xi^{(0)}_\pm$, from equation \eqref{ddstar} we obtain the
following bound for the real part of $\lambda$:
\begin{equation*}
	\begin{aligned}
 	\Real \lambda = \eta &= \frac{1}{2\tau} \Big( - b_\pm \pm \sqrt{\Delta_2(\xi)} \Big)
		\leq \frac{1}{2\tau} \Big( - b_\pm + \sqrt{b_\pm^2 - 4 \tau \delta_\pm} \Big)\\
	&= \frac{1}{2\tau} (-(1+\tau \delta_\pm) + (1- \delta_\pm \tau))
		= - \delta_\pm < - \frac{\delta_\pm}{2}.
	\end{aligned}
\end{equation*}
Likewise, when $|\xi| \geq \xi^{(0)}_\pm$, we have the uniform bound 
\begin{equation*}
 \Real \lambda = \eta = - \frac{b_\pm}{2\tau} = - \frac{1+\tau \delta_\pm}{2\tau} < - \frac{\delta_\pm}{2},
\end{equation*}
for all $0 < \tau < \tau_m$. We have proved the following 

\begin{lemma}
For all $\tau \in (0, \tau_m)$, there exists a uniform
\begin{equation}
 \label{defchi0}
\chi_0 := \tfrac{1}{2} \min \{\delta_+, \delta_-\} > 0,
\end{equation} 
such that the algebraic curves $\lambda = \lambda_{1,2}^\pm(\xi)$, $\xi \in \R$, solutions to the dispersion relations \eqref{disprel}, satisfy
\begin{equation}
\label{spectralgapeq}
\mathrm{Re}\, \lambda_{1,2}^\pm(\xi) < - \chi_0 < 0, 
\end{equation}
for all $\xi \in \R$.
\end{lemma}

\subsubsection{Stability of the essential spectrum}

We define the following open, connected region of the complex plane,
\begin{equation}\label{defOmega}
	\Omega := \{\lambda \in \C \, : \, \Real \lambda > - \chi_0\}.
\end{equation}
It properly contains the unstable complex half plane $\C_+ = \{ \Real \lambda > 0\}$.
Denote ${\mathbb{S}^\tau_\pm}(\lambda)$ and ${\mathbb{U}^\tau_\pm}(\lambda)$ as the stable and unstable
eigenspaces of $\A^\tau_\pm(\lambda)$, respectively.

\begin{lemma}\label{lemconsistsplit}
For all $\tau \in (0,\tau_m)$, and all $\lambda \in \Omega$, the coefficient matrices
$\A^\tau_\pm(\lambda)$ have no center eigenspace and, moreover,
\begin{equation*}
	\dim {\mathbb{S}^\tau_\pm}(\lambda) = \dim {\mathbb{U}^\tau_\pm}(\lambda) = 1.
\end{equation*}
\end{lemma}

\begin{proof}
Take $\lambda \in \Omega$ and suppose $\kappa = i \xi$, with $\xi \in \R$, is an
eigenvalue of $\A_\pm^\tau(\lambda)$.
Then $\lambda$ belongs to one of the algebraic curves \eqref{algcurves}.
But \eqref{spectralgapeq} yields a contradiction with $\lambda \in \Omega$.
Therefore, the matrices $\A_\pm^\tau(\lambda)$ have no center eigenspace.

Since $\Omega$ is a connected region of the complex plane, it suffices to compute
the dimensions of ${\mathbb{S}^\tau_\pm}(\lambda)$ and ${\mathbb{U}^\tau_\pm}(\lambda)$ when
$\lambda = \eta \in \R_+$, sufficiently large.
The characteristic polynomial \eqref{defcharpol} of $\A_\pm^\tau(\lambda)$ is
\begin{equation*}
	\kappa^2 + \kappa c(b_\pm + 2 \tau \lambda)
		- (1 - c^2 \tau)(\tau \lambda^2 + \lambda b_\pm + \delta_\pm) = 0.
\end{equation*}
Assuming $\lambda = \eta \in \R_+$, the roots are
\begin{equation*}
	\kappa = - \frac{c}{2}(b_\pm + 2 \tau \eta) \pm \frac{1}{2} \sqrt{c^2 (b_\pm + 2\tau \eta)^2
		+ 4(1-c^2\tau)(\tau \eta^2 + \eta b_\pm + \delta_\pm)}.                                                     
\end{equation*}
Clearly, for each $\eta > 0$, one of the roots is positive and the other is negative.
This proves the lemma.
\end{proof}

In view of last result, the region $\Omega$ is often called the \textit{region of consistent splitting} \cite{San02}. 

\begin{corollary}[Stability of the essential spectrum]\label{lemspectralgap}
For each $\tau \in (0,\tau_m)$, the essential spectrum is contained in the stable half-plane.
More precisely,
\begin{equation*}
 	\sigma_{\mathrm{ess}} \subset \{\lambda \in \C \, : \, \Real \, \lambda \leq - \chi_0 < 0\}.
\end{equation*}
\end{corollary}

\begin{proof}
Fix $\lambda \in \Omega$. By exponential dichotomies theory
(see \cite{Cop2,San02}), since $\A_\pm^\tau(\lambda)$ are hyperbolic,
the asymptotic systems \eqref{firstorders} have exponential dichotomies in $x \in \R_+ = (0,+\infty)$
and in $x \in \R_- = (-\infty,0)$ with respective Morse indices
\begin{equation*}
	i_+(\lambda) = \dim {\mathbb{U}^\tau_+}(\lambda) = 1, \qquad i_-(\lambda) = \dim {\mathbb{U}^\tau_-}(\lambda) = 1.
\end{equation*}
This implies (see \cite{Pal1,Pal2} and also \cite{San02}) that the variable
coefficient operators $\mathcal{T} ^\tau(\lambda)$ are Fredholm as well, with index
\begin{equation*}
	\ind \mathcal{T} ^\tau(\lambda) = i_+(\lambda) - i_-(\lambda) = 0,
\end{equation*}
showing that $\Omega \subset \C\backslash \ess$, or equivalently,
$\ess \subset \C\backslash \Omega = \{\Real \lambda \leq - \chi_0  < 0\}$. 
\end{proof}

The significance of Corollary \ref{lemspectralgap} is that there is no accumulation of
essential spectrum at the eigenvalue $\lambda = 0$, which is an isolated eigenvalue
with finite multiplicity. In other words, there is a \textit{spectral gap}.

\subsection{Evans function analysis}

The Evans function (cf. \cite{AGJ90,KaPro13,San02}) is a powerful tool to locate the point spectrum.
Thanks to Lemma \ref{lemconsistsplit}, $\Omega$ is the open, connected component of
$\C \backslash \ess$ containing the (unstable) right half-plane in which the asymptotic matrices
$\A^\tau_{\pm}(\cdot)$ are hyperbolic and the dimensions of their stable ${\mathbb{S}^\tau_\pm}$
(respectively, unstable ${\mathbb{U}^\tau_\pm}$) spaces agree.
By spectral separation of ${\mathbb{U}^\tau_\pm}, {\mathbb{S}^\tau_\pm}$, the associated eigenprojections are
analytic in $\lambda$ and there exists analytic representations for the bases of subspaces
${\mathbb{S}^\tau_\pm}$ and ${\mathbb{U}^\tau_\pm}$ (by a Kato construction, cf. \cite[pp.99--102]{Kat80}).
In our special (low dimensional) case, 
\begin{equation*}
	{\mathbb{S}^\tau_+} = \mathrm{span} \{{\mathbf{w}}_+(\lambda)\}, \qquad
	{\mathbb{U}^\tau_-} = \mathrm{span} \{{\mathbf{w}}_-(\lambda)\},
\end{equation*}
where ${\mathbf{w}}_\pm(\lambda)$ can be chosen analytic in $\lambda \in \Omega$.
The associated Evans function
\begin{equation}\label{defD}
	D^\tau(\lambda) := \det ({\mathbf{w}}_-(\lambda), \, {\mathbf{w}}_+(\lambda)),
\end{equation}
is defined to locate non-trivial intersections of the initial conditions ${\mathbf{w}}_+$ which produce solutions
to the variable coefficient systems \eqref{firstorders} that decay when $x \to +\infty$, with the initial
conditions ${\mathbf{w}}_-$ which produce solutions to \eqref{firstorders} that decay at $x \to - \infty$.
The Evans function is not unique, but they all differ by a non-vanishing factor.
It is endowed with the following properties: 
\begin{itemize}
	\item[\textbullet] $D^\tau$ is analytic in $\lambda \in \Omega$;
	\item[\textbullet] $D^\tau(\lambda) = 0$ if and only if $\lambda \in \ptsp \cap \Omega$; and,
	\item[\textbullet] the order of $\lambda$ as a zero of $D^\tau$ is equal to its algebraic multiplicity 
\end{itemize}
In our case, we end up with a family of Evans functions $D^\tau(\cdot)$ indexed by $\tau \in [0,\tau_m)$
and defined on $\Omega$. It is to be observed that the region $\Omega$ is independent of $\tau$,
and that the case when $\tau = 0$ is included in the family.
In view of Proposition \ref{prop:tau0}, which guarantees the spectral stability of the parabolic Allen-Cahn front,
we have the following

\begin{corollary}\label{corD0}
$D^0(\lambda) \neq 0$ for all $\Real \lambda \geq 0$, $\lambda \neq 0$.
Moreover, $\lambda = 0$ is a simple zero of $D^0(\cdot)$.
\end{corollary}

In order to establish spectral stability in the regime $\tau \in (0, \tau_m)$, we shall apply a result from
Evans function theory (see \cite{PZ1}) which assures that, under suitable structural but rather general conditions,
the Evans functions for $\tau > 0$ converge uniformly to the Evans function with $\tau = 0$ in bounded
regions of $\lambda \in \Omega$. For that purpose, it will be necessary to show that large  $|\lambda|$ 
values belong to the resolvent set. By analiticity and uniform convergence, the non-vanishing property of 
$D^0$ persists for $D^\tau$
for each $0 < \tau \ll 1$ sufficiently small.
Next, by continuity in $\tau$ of eigenvalues and by Lemma \ref{lemam1} and  Lemma \ref{lemma:imagevalue}, we rule out possible
crossing of eigenvalues across the imaginary axis as $\tau$ varies within the full set $(0,\tau_m)$,
establishing point spectral stability for all values of $\tau$ under consideration.

Therefore, let us consider the family of first order systems \eqref{firstorders} for $\lambda \in \Omega$
and with $\tau$ varying in a compact set $\mathcal{V} : = [0,\tau_1]$, with $\tau_1 < \tau_m$.
First, we observe that the coefficients $\A^\tau (\cdot,\lambda)$ are functions of
$(\lambda, \tau) \in \Omega \times \mathcal{V}$ into $L^\infty(\R;\R^{2 \times 2})$ (the coefficients are bounded),
they are analytic in $\lambda \in \Omega$ (second order polynomial in $\lambda$), and continuous
in $\tau \in \mathcal{V}$ (this follows from the continuity of the coefficients and of the velocity $c$ in $\tau$).
Moreover, in view of Theorem \ref{thm:existence}(c),
\begin{equation*}
c(\tau) = c_0 + \zeta(\tau), \qquad \zeta(\tau) = o(1) \;  \text{as} \, \tau \to 0^+. 
\end{equation*}
Here $c_0$ is, of course, the speed of the traveling wave for the parabolic Allen-Cahn equation (or Nagumo front).
Also, notice that from the expressions of the coefficients \eqref{asymptcoeffA} we may write
\begin{equation*}
\A^\tau_\pm(\lambda) = \A^0_\pm(\lambda) + (1-c^2 \tau)^{-1} \Q^\tau_\pm(\lambda),
\end{equation*}
where the residual is
\begin{equation*}
	\begin{aligned}
	\Q^\tau_\pm(\lambda) &= \begin{pmatrix} 0 & \; & 0 \\
	\tau \lambda^2 + \tau(c^2 -a(x))\lambda +c\tau a'(x) & \; & c\tau(cc_0 - a(x) -2\lambda) - (c-c_0) \end{pmatrix}\\
	&= (\lambda^2 + \lambda + 1) O(\tau) + O(|\zeta(\tau)|),
	\end{aligned}
\end{equation*}
so that
\begin{equation*}
	|\Q^\tau_\pm(\lambda)| \leq O(\tau + |\zeta(\tau)|)(1+|\lambda| + |\lambda|^2).
\end{equation*}
Thanks to exponential decay of the wave \eqref{expodecay} we conclude that the coefficients
$\At(\cdot,\lambda)$ approach exponentially to its limit coefficients $\At_\pm(\lambda)$ as $x \to \pm \infty$:
\begin{equation*}
	|\At(x,\lambda) - \At_\pm(\lambda)| \leq C e^{-\nu |x|}, \qquad \text{for} \;\; |x| \to +\infty,
\end{equation*}
uniformly on compact subsets of $(\lambda,\tau) \in \Omega \times \mathcal{V}$.
In addition, by Lemma \ref{lemconsistsplit} the limiting coefficient matrices are hyperbolic
with agreeing dimensions of their unstable eigenspaces.

Finally, the geometric separation assumption of Gardner and Zumbrun \cite{GZ98}
(namely, that the limits of the spaces ${\mathbb{S}^\tau_\pm}$ and ${\mathbb{U}^\tau_\pm}$
along $\lambda$-rays, $\lambda = r\lambda_0$ as $r \to 0^+$, $\lambda_0 \in \Omega$, are continuous)
holds trivially in our case as the matrices $\At_\pm(0)$ are hyperbolic and the eigenspaces are
one-dimensional with uniform spectral separation.

To sum up, we have verified that assumptions (A0)-(A1)-(A2) in \cite[p.894]{PZ1} are satisfied,
and the systems \eqref{firstorders} belong to the generic class of equations for which there is
convergence of approximate flows  \cite[Section 2]{PZ1}.
We need to verify one final hypothesis to apply \cite[Proposition 2.4]{PZ1}.

\begin{lemma}
\label{lemapprox}
Let $(\lambda,\tau) \in \Omega \times \mathcal{V}$.
Then the stable eigenvector ${\mathbf{w}}_+(\lambda)$ of $\At_+(\lambda)$ and the unstable
eigenvector ${\mathbf{w}}_-(\lambda)$ of $\At_-(\lambda)$ converge as $\tau \to 0^+$ with
rate $\eta(\tau) := O(\tau + |\zeta(\tau)|)$ to the stable and unstable eigenvectors of $\A^0_+(\lambda)$
and $\A^0_-(\lambda)$, respectively.
Moreover, for all $\tau \in \mathcal{V}$,
\begin{equation*}
|(\At - \At_\pm) - (\A^0 - \A^0_\pm)| \leq C_1 \eta(\tau) e^{-\tilde \nu |x|},
\end{equation*}
as $x \to \pm \infty$ for some constants $C_1, \tilde \nu > 0$, uniformly in compact sets of $\Omega$.
\end{lemma}
\begin{proof}
Let $\gamma$ be a closed rectifiable contour enclosing the stable eigenvalue of $\At_\pm$. By continuity on $\tau$, we may as well select $\gamma$ such that it encloses the stable eigenvalue of $\At_\pm$ for each $\tau > 0$ sufficiently small. By compactness of $\bar \gamma$ we have a uniform resolvent bound of the form
\begin{equation*}
 |(\A^0_+ - z)^{-1}| \leq C, \qquad z \in \gamma.
\end{equation*}
Thus, expanding,
\begin{equation*}
 \begin{aligned}
  (\At_+ - z)^{-1} &= (\A^0_+ - z + (1-c^2 \tau)^{-1}\Q^\tau_+)^{-1}\\
&= ((\A^0_+ - z)(I + (\A^0_+ - z)^{-1}(1-c^2\tau)^{-1}\Q^\tau_+))^{-1}\\
&= (I + (\A^0_+ - z)^{-1}(1-c^2\tau)^{-1}\Q^\tau_+)^{-1}(\A^0_+ - z)^{-1}\\
&= ( I + O(|\eta(\tau)|))(\A^0_+ - z)^{-1}.
 \end{aligned}
\end{equation*}
Here $\eta(\tau)$ depends on $|\A^0_+|$. Therefore, the (one-dimensional) projection $\mathbf{P}^\tau_+$ onto ${\mathbb{S}^\tau_+}$ satisfies the bound 
\begin{equation*}
	\begin{aligned}
	|\mathbf{P}^\tau_+(\lambda) - \mathbf{P}^0_+(\lambda)|
		&= \left| \frac{1}{2\pi i} \oint_\gamma (z-\A^\tau_+(\lambda))^{-1} \, dz \; - \; \frac{1}{2\pi i}
		\oint_\gamma (z-\A^0_+(\lambda))^{-1} \, dz\right|\\
	&\leq \frac{1}{2\pi i} \oint_\gamma O(|\eta(\tau)|)|(\A^0_+ -z)^{-1}| \, dz \leq C |\eta(\tau)|,
	 \end{aligned}
\end{equation*}
that is, $\mathbf{P}^\tau_+(\lambda) = \mathbf{P}^0_+(\lambda) + O(|\eta(\tau)|)$,
showing that ${\mathbf{w}}_+(\lambda) \to {\mathbf{w}}_+^0(\lambda)$ as $\tau \to 0^+$ with
rate $|\eta(\tau)|$ in $\Omega$-neighborhoods of $\lambda$.
The same applies to the unstable  eigenvectors.
The second assertion is an immediate consequence of the exponential decay \eqref{expodecay}
of the coefficients.
\end{proof}

\subsection{Resolvent estimates and proof of Theorem \ref{thm:spectrum}}

In order to give a complete proof of Theorem \ref{thm:spectrum}, we establish a general resolvent
estimates, based on an approximate diagonalization technique introduced by Mascia and Zumbrun \cite{MasZum02}.

\subsubsection{Resolvent estimates}

Given two real symmetrix matrices $\mathbf{A}, {\mathbf{B}}\in\R^{n\times n}$,
a real matrix valued function $x\mapsto \mathbf{C}(x)$ defined for any $x\in\R$, 
and $\lambda\in\C$, let us consider the {\it resolvent equation}
\begin{equation}\label{resolvent}
	\mathbf{A}{\mathbf{w}_x}+(\lambda {\mathbf{B}}+\mathbf{C}(x)) \mathbf{w} = \mathbf{f}, \qquad \mathbf{f} \in H^m(\R;\C^n),
\end{equation}
for the unknown $\mathbf{w}\in H^m(\R;\C^n)$.
Here, the concern is to show that, for appropriate choices of sets $\tilde \Omega\subset\C$,
there exists some constant $M$ such that for any solution $\mathbf{w}=\mathbf{w}(\cdot, \,\lambda)$
to \eqref{resolvent} with $\lambda\in \tilde \Omega$, there holds
\begin{equation}\label{genres}
	|\mathbf{w}(\cdot,\lambda)|_{H^m}\leq M|\mathbf{f}|_{H^m}.
\end{equation}
The first classical result, consequence of the underlying hyperbolic problem from which
\eqref{resolvent} arises, provides an estimate for sets $\tilde \Omega$ composed by numbers
with large positive real part.

\begin{proposition}\label{prop:resesti0}
Given $m\in\mathbb{N}$, let ${\mathbf{B}}$ be positive definite and $\mathbf{C}=\mathbf{C}(x)$ uniformly
bounded in $\R$ together with its derivatives of order $j=1,\dots,m$.
Then there exist $L,M>0$ such that for any $\lambda$
with $\Real \lambda\geq L$ there holds
\begin{equation}\label{res0}
	|\mathbf{w}(\cdot, \lambda)|_{H^m}\leq \frac{M}{\Real \lambda}\,|\mathbf{f}|_{H^m}.
\end{equation}
\end{proposition}

\begin{proof}
As a first step, let us consider the case of $L^2$, i.e. $m=0$.
We recall that
if $\mathbf{A}$ is a constant real symmetrix matrix then
\begin{equation*}
 \mathbf{w}^* \mathbf{A}{\mathbf{w}_x} = \tfrac{1}{2}\left(\mathbf{w}^* \mathbf{A}\mathbf{w}\right)_x.
\end{equation*}
Taking the scalar product of \eqref{resolvent} against $\mathbf{w}$ and integrating in $\R$, we get
\begin{equation*}
	\lambda \langle {\mathbf{w}, \mathbf{B}}\mathbf{w}\rangle_{L^2}+\langle \mathbf{w},
	\mathbf{C}(x)\mathbf{w} \rangle_{L^2} = \langle \mathbf{w}, \mathbf{f} \rangle_{L^2}.
\end{equation*}
Since ${\mathbf{B}}$ is symmetric, the term $\langle \mathbf{w},{\mathbf{B}}\mathbf{w}\rangle_{L^2}$ is real; thus,
taking the real part, we infer that there exists some constant $C>0$ such that
\begin{equation*}
	(\Real \lambda) \langle \mathbf{w}, \mathbf{B}\mathbf{w}\rangle_{L^2}
	\leq  C\left(|\mathbf{w}|_{L^2}^2+|\mathbf{f}|_{L^2}^2\right)
\end{equation*}
having used the standard Young inequality.
As a final step, under the hypothesis that ${\mathbf{B}}$ is positive definite, it is possible
to absorb the term with $\mathbf{w}$ at the right-hand side into the corresponding term
in the left-hand side and deduce \eqref{res0} with $m = 0$.

Next, we may proceed inductively, assuming estimate \eqref{res0} for any $j=0,1,\dots,m-1$.
By differentiation of \eqref{resolvent}, we deduce that the function 
$\mathbf{z}:=d^m \mathbf{w}/dx^m$ solves
\begin{equation}\label{resolvent_kth}
	\mathbf{A} \mathbf{z}_x + (\lambda {\mathbf{B}}+\mathbf{C}(x)){\mathbf{z}}=\frac{d^m\mathbf{f}}{dx^m}
		-\sum_{j=0}^{m-1} \binom{m}{j} \frac{d^{m-j} \mathbf{C}}{dx^{m-j}}\,\frac{d^j \mathbf{w}}{dx^j}.
\end{equation}
Since the coefficients of the derivatives of $\mathbf{C}$ are assumed to be bounded,
as a consequence of \eqref{res0} with $m=0$, we infer the estimate
\begin{equation*}
	|z(\cdot, \,\lambda)|_{{L^2}}\leq \frac{C}{\Real\,\lambda}
	\left(\left|\frac{d^m\mathbf{f}}{dx^m}\right|_{{L^2}}
		+\sum_{j=0}^{m-1} \left|\frac{d^jw}{dx^j}\right|_{{L^2}}\right)
\end{equation*}
for some $C>0$.
Then, the inductive assumption provides the conclusion.
\end{proof}

Next, we turn to the problem of proving \eqref{genres} for sets $\tilde \Omega$ contained
in the half plane $\{\Real \lambda\geq 0\}$ and with sufficiently large modulus.
Of course, additional restrictions on the matrices $\mathbf{A}, {\mathbf{B}}, \mathbf{C}$ are needed.
Our approach is based on the approximate diagonalization procedure presented
in \cite[p.817 and following]{MasZum02}.
Specifically, let us consider a change of variable $\mathbf{w}=\mathbf{T}(x,\lambda)\mathbf{v}$ with
$\mathbf{T}(x,\lambda)$ in the form
\begin{equation*}
	\mathbf{T}(x,\lambda)=\mathbf{T}_0(I+\lambda^{-1}\mathbf{T}_1(x)),
\end{equation*}
with ${\mathbf{T}_0}$ and ${\mathbf{T}_1}={\mathbf{T}_1}(x)$ to be determined.
Plugging into \eqref{resolvent} and assuming $\mathbf{A}$ invertible,
we deduce the equation solved by the new unknown $\mathbf{v}$
\begin{equation*}
	\mathbf{v}_x+\widetilde{\mathbf{A}}(x,\lambda)\mathbf{v}=\tilde{\mathbf{f}},
\end{equation*}
where
\begin{equation*}
	\begin{aligned}
	\widetilde{\mathbf{A}}(x,\lambda)
		& = (I+\lambda^{-1}{\mathbf{T}_1})^{-1} \Bigl\{\lambda {\mathbf{T}_0}^{-1}\mathbf{A}^{-1}{\mathbf{B}}{\mathbf{T}_0}
		+{\mathbf{T}_0}^{-1}\mathbf{A}^{-1}\bigl({\mathbf{B}}{\mathbf{T}_0}{\mathbf{T}_1}+\mathbf{C}{\mathbf{T}_0}\bigr)\\
	&\hskip5.0cm		
+\lambda^{-1}\bigl({\mathbf{T}_0}^{-1}\mathbf{A}^{-1}\mathbf{C}{\mathbf{T}_0}{\mathbf{T}_1}+{\mathbf{T}_1}
'\bigr)\Bigr\},\\	
\tilde{\mathbf{f}}&=(I+\lambda^{-1}{\mathbf{T}_1}(x))^{-1}{\mathbf{T}_0}^{-1}\mathbf{A}^{-1}\mathbf{f}.
	\end{aligned}
\end{equation*}
The matrix $\widetilde{\mathbf{A}}$ can be represented as
\begin{equation*}
	\widetilde{\mathbf{A}}(x,\lambda)=\lambda \mathbf{D}_0+{\mathbf{D}_1}+o(1)\qquad\qquad\lambda\to\infty
\end{equation*}
where the matrices $\mathbf{D}_0$ and ${\mathbf{D}_1}$ are given by
\begin{equation}\label{D01form}
	\mathbf{D}_0:={\mathbf{T}_0}^{-1}\mathbf{A}^{-1}{\mathbf{B}}{\mathbf{T}_0},\qquad
	{\mathbf{D}_1}:=[\mathbf{D}_0,{\mathbf{T}_1}(x)]+{\mathbf{T}_0}^{-1}\mathbf{A}^{-1}\mathbf{C}(x){\mathbf{T}_0},
\end{equation}
and $[\mathbf{A},{\mathbf{B}}]=\mathbf{A}{\mathbf{B}}-{\mathbf{B}}\mathbf{A}$.
If the matrix $\mathbf{A}^{-1}{\mathbf{B}}$ is diagonalizable, ${\mathbf{T}_0}$ can be chosen so that $\mathbf{D}_0$
is diagonal and ${\mathbf{T}_1}={\mathbf{T}_1}(x)$ is such that the matrix ${\mathbf{D}_1}$ is diagonal
(see \cite[Lemma 4.6]{MasZum02}).
Moreover, since $[\mathbf{D}_0,\mathbf{T}]$ is equal to zero for any diagonal matrix, the term ${\mathbf{T}_1}$
can be chosen with zero entries in the principal diagonal.
With these choices, the special form of the system satisfied by $\mathbf{v}$ can be used
to obtain a different form of estimate \eqref{genres}.

\begin{proposition}\label{prop:resesti1}
Given $m\in\mathbb{N}$, let the matrices $\mathbf{A}$ ${\mathbf{B}}$ and $\mathbf{C}$ be such that $\mathbf{A}$ is invertible,
$\mathbf{A}^{-1}{\mathbf{B}}$ is diagonalizable in $\R$, and $\mathbf{C}=\mathbf{C}(x)$ is uniformly bounded in $\R$ together
with all its derivatives of order $j=1,\dots,m$.
Let the elements of $\mathbf{D}_0$ and ${\mathbf{D}_1}$, defined in \eqref{D01form}, be denoted by
$\mu_0^k$ and $\mu_1^k$, $k=1,\dots,n$.
If $\mu_0^k$ and $\Real \mu_1^k$ have the same sign for any $k$ and 
\begin{equation}\label{mind1k}
	\min_{k=1,\dots,n}\inf_{x\in\R} |\Real \mu_1^k|>0,	
\end{equation}
then there exist $R>0$ and $M$ such that the estimate \eqref{genres} holds
for any $\lambda\in\tilde\Omega_R:=\{|\lambda|\geq R,\,\Real \lambda\geq 0\}$.
\end{proposition}

\begin{proof}
As in the proof of Proposition \ref{prop:resesti0}, we initially consider the $L^2$-case, i.e. $m=0$.
With the change of variable $\mathbf{v}=\mathbf{T}(x,\lambda)\mathbf{w}$ where $\mathbf{T}$ has been chosen following
the above procedure, we end up with a system for the unknown $\mathbf{v}$, whose $k-$th
component solves
\begin{equation*}
	\frac{dv_k}{dx}+\left(\lambda \mu_{0}^{k}+\mu_{1}^{k}(x)\right)v_k
		=o(1)v_k + \tilde{f}_k.
\end{equation*}
Taking the scalar product against $\bar v_k$, integrating in $\R$, and taking its real
part, we deduce
\begin{equation*}
	\int_{\R}\left(\Real \lambda\,\mu_{0}^{k}+\Real\,\mu_{1}^{k}(x)\right)|v_k|^2\,dx
		=\Real\int_{\R}(o(1)v_k+\tilde{f}_k)\bar v_k\,dx.
\end{equation*}
For any $\lambda$ with positive real part, since $\mu_{0}^{k}$ and $\Real\,\mu_{1}^{k}(x)$ 
have the same sign, we obtain
\begin{equation*}
	\int_{\R}\left(\Real\lambda\,|\mu_{0}^{k}|+|\Real\,\mu_{1}^{k}(x)|\right)|v_k|^2\,dx
		\leq o(1)|v|_{{L^2}}^2+C|\tilde{f}|_{{L^2}}^2
\end{equation*}
for some $C>0$ and then, as a consequence of \eqref{mind1k}, 
\begin{equation*}
	C_0|\mathbf{v}|_{{L^2}}^2\leq o(1)|\mathbf{v}|_{{}_{L^2}}^2+C|\tilde{\mathbf{f}}|_{{L^2}}^2
\end{equation*}
for some $C_0>0$.
For $|\lambda|$ sufficiently large, the term $o(1)$ can be controlled by $C_0$,
so that we end up with
\begin{equation*}
	|\mathbf{v}|_{{L^2}}^2\leq C|\tilde{\mathbf{f}}|_{{L^2}}^2.
\end{equation*}
Finally, since $C$ is uniformly bounded, also $\mathbf{T}$ and ${\mathbf{T}_1}$ are,
and the estimate can be brought back to the original variable $\mathbf{w}$.

The case of $m\geq 1$ follows by differentiating \eqref{resolvent} and proceeding
as in the proof of Proposition \ref{prop:resesti0}.
\end{proof}

In the case at hand, system \eqref{resolvent} is two-dimensional and defined by
the matrices $\mathbf{A}$, ${\mathbf{B}}$ and $\mathbf{C}$ of \eqref{defABC}.
Then, the matrix 
\begin{equation*}
	\mathbf{A}^{-1}{\mathbf{B}}=\frac{1}{1-c^2\tau}\begin{pmatrix} -c\tau & \tau \\ 1 & -c\tau \end{pmatrix}
\end{equation*}
has real eigenvalues $\mu_0^\pm =\pm\sqrt{\tau}(1\mp c\sqrt{\tau})$ and 
\begin{equation*}
	{\mathbf{T}_0}^{-1}\mathbf{A}^{-1}{\mathbf{B}}{\mathbf{T}_0}=\mathbf{D}_0=\diag(\mu_0^-,\mu_0^+),
	\qquad\textrm{where}\quad
	{\mathbf{T}_0} =\begin{pmatrix} -\sqrt{\tau} & \sqrt{\tau} \\ 1 & 1 \end{pmatrix}.
\end{equation*}
Then, straightforward computations give
\begin{equation*}
	{\mathbf{T}_0}^{-1}\mathbf{A}^{-1}\mathbf{C}(x){\mathbf{T}_0}=\frac{1}{2\sqrt{\tau}}
	\begin{pmatrix} -(1-\tau a)/(1-c\sqrt{\tau}) & -(1+\tau a)/(1-c\sqrt{\tau}) \\ 
			(1+\tau a)/(1+c\sqrt{\tau}) & (1-\tau a)/(1+c\sqrt{\tau}) \end{pmatrix},
\end{equation*}
(where $a=a(x)$) so that
\begin{equation*}
 {\mathbf{D}_1}=\diag(\mu_1^-,\mu_1^+)
 	=\frac{1}{2\sqrt{\tau}}\diag\left(-\frac{1-\tau a}{1-c\sqrt{\tau}},\frac{1-\tau a}{1+c\sqrt{\tau}}\right).
\end{equation*}
Thus, if the function $a=a(x)$ is such that $\sup a<{1}/{\tau}$,
then assumption of Proposition \ref{prop:resesti1} holds and the resolvent estimate
\eqref{genres} holds in $\tilde\Omega_R:=\{|\lambda|\geq R,\,\Real\,\lambda\geq 0\}$
for $R$ sufficiently large.

\subsubsection{Proof of Theorem \ref{thm:spectrum}}

In view of Proposition \ref{prop:resesti1}, take $R > 0$ sufficiently large
so that $\ptsp \cap \Omega \subset \{|\lambda| \leq R\}$, and consider the following compact subset of $\Omega$:
\begin{equation*}
 \Omega_R := \{\lambda \in \C \, : \, |\lambda| \leq R, \, \Real \, \lambda \geq - \tfrac{1}{2} \chi_0\}.
\end{equation*}
By Lemma \ref{lemapprox}, systems \eqref{firstorders} satisfy the hypotheses of  \cite[Proposition 2.4]{PZ1}.
Hence, for each $\tau \in \mathcal{V}$ and in a $\Omega_R$-neighborhood of $\lambda$, the local Evans
functions $D^\tau(\lambda)$ converge uniformly to $D^0(\lambda)$ in a (possible smaller) neighborhood
of $\lambda$ as $\tau \to 0^+$ with rate $|D^\tau(\cdot) - D^0(\cdot)| = O(\eta(\tau)) = O(\tau + |\zeta(\tau)|) \to 0$.

By point spectral stability for $\tau = 0$ (Corollary \ref{corD0}), and by analiticity and uniform convergence, we conclude that $D^\tau(\lambda) \neq 0$ for $\lambda \in \Omega_R$, $\Real \, \lambda \geq 0$, except only at $\lambda = 0$, and for each $0 \leq  \tau \ll 1$ sufficiently small. Hence there exists $\tau_0 \in (0,\tau_m)$ such that point spectral stability holds for each $\tau \in (0,\tau_0)$. Finally, noticing that $\Omega_R$ contains only isolated eigenvalues with finite multiplicity, and by continuous dependence of eigenvalues of Fredholm operators in Banach spaces with respect to its coefficients (cf. \cite{MoZe96}), the eigenvalues $\lambda = \lambda(\tau) \in \ptsp \cap \Omega_R$ are continuous functions of $\tau \in (0,\tau_m)$. In view of Lemma \ref{lemma:imagevalue}, such eigenvalues may cross the imaginary axis towards the unstable half plane only through the origin. But $\lambda = 0 \in \ptsp$ is a simple eigenvalue for each $\tau \in (0,\tau_m)$ as proved in Lemma \ref{lemam1}. Hence all point spectrum remains in the stable half plane $\Real \, \lambda < 0$ for all $\tau \in (0,\tau_m)$. This proves the Theorem.

\section{Decaying semigroup and nonlinear stability}\label{sect:linnonlinstab}

In this Section, we establish the conditions for the generation of a $C_0$-semi\-group of solutions operators for
the linearization around the wave, as well as for its asymptotic decaying properties.
We also present the proof of nonlinear (orbital) stability, which uses such information in a key way. 

\subsection{Generation of the semigroup}

We are now ready to establish that each operator $\mathcal{L}^\tau$ generates a $C_0$ semigroup in $L^2(\R;\C^2)$.

\begin{lemma}
 \label{lemgenC0}
For each $\tau \in (0,\tau_m)$, the operator $\mathcal{L}^\tau : D = H^2(\R;\C^2) \to L^2(\R;\C^2)$ is the
infinitesimal generator of a $C_0$-semigroup, $\{{\mathcal{S}(t)}\}_{t\geq 0}$, satisfying
\begin{equation}
 \label{quasicontr}
\|{\mathcal{S}(t)}\| \leq e^{\omega t},
\end{equation}
for some $\omega = \omega(\tau) \in \R$, all $t \geq 0$.
\end{lemma}

Here $\| \cdot \|$ denotes the operator norm.

\begin{proof}
First, we note that the domain $D = H^2(\R;\C^2)$ is dense in $L^2(\R;\C^2)$. Now, for each $\mathbf{u} = (u,v)^\top \in D$:
\begin{equation*}
 \begin{aligned}
  \Real \, \langle {\mathbf{u}}, \mathcal{L}^\tau {\mathbf{u}}\rangle_{L^2} &= - \Real \, \langle {\mathbf{u}}, {\mathbf{B}}^{-1}( \mathbf{A} {\mathbf{u}_x} + \mathbf{C}(x){\mathbf{u}}) \rangle_{L^2}\\
&= - \Real \, \langle {\mathbf{u}}, {\mathbf{B}}^{-1} \mathbf{A} {\mathbf{u}_x}\rangle_{L^2} - \Real \, \langle {\mathbf{u}}, {\mathbf{B}}^{-1}\mathbf{C}(x) {\mathbf{u}} \rangle_{L^2}\\
&= - \int_\R a(x) |u|^2 \, dx + \tau^{-1} |v|^2_{L^2}\\
&\leq \sup_\R |a(x)| |u|^2_{L^2} + \tau^{-1} |v|^2_{L^2} \leq \omega |{\mathbf{u}}|_{L^2}^2,
 \end{aligned}
\end{equation*}
with  $\omega = \max \{\sup |a(x)|, \tau^{-1}\} > 0$.

Now, thanks to the resolvent estimate \eqref{genres} for any $|\lambda| \geq R$, $\Real \, \lambda \geq 0$ with $R$ sufficiently large, there are no $L^2$ solutions to $\mathcal{L}^\tau {\mathbf{u}} = \lambda_0 {\mathbf{u}}$ for $\lambda_0 \in \R$, $\lambda_0 > \omega$ and sufficiently large. Thus, for each $\lambda_0 > \omega$ sufficiently large $\mathcal{L} -  \lambda_0$ is onto. A direct application of the classical Hille-Yosida theorem \cite{EN00,Pa83} yields the result together with the estimate \eqref{quasicontr}.
\end{proof}

As a consequence of the semigroup properties we have that 
\begin{equation*}
 \frac{d}{dt} ({\mathcal{S}(t)}{\mathbf{u}}) = {\mathcal{S}(t)} \mathcal{L}^\tau {\mathbf{u}} = \mathcal{L}^\tau {\mathcal{S}(t)}{\mathbf{u}},
\end{equation*}
for all ${\mathbf{u}} = (u,v)^\top \in D = H^2(\R;\C^2)$.

Naturally, the growth rate $\omega$ of estimate \eqref{quasicontr} is not optimal. Actually, $\omega \to +\infty$ as $\tau \to 0^+$, due to the fact that, in the limit, the operator $\mathcal{L}^\tau$ is not defined and becomes singular. The optimal growth rate in the appropriate subspace will be provided by spectral stability. The significance of Lemma \ref{lemgenC0} is simply that, for each fixed $\tau \in (0,\tau_m)$, the operator $\mathcal{L}^\tau$ is the generator of a $C_0$-semigroup. Let us recall the growth bound for a semigroup ${\mathcal{S}(t)}$:
\begin{equation*}
 \omega_0 = \inf \{\omega \in \R \, : \, \lim_{t \to + \infty} e^{-\omega t}\|{\mathcal{S}(t)}\| \; \text{exists}\}.
\end{equation*}
We say a semigroup is uniformly (exponentially) stable whenever $\omega_0 < 0$. Let $\cL$ be the infinitesimal generator of the semigroup ${\mathcal{S}(t)}$. Its spectral bound is defined as 
\begin{equation*}
 s(\cL) = \sup \{ \Real\, \lambda \, : \, \lambda \in \sigma(\cL)\}.
\end{equation*}
Since the spectral mapping theorem
--- namely, that, $\sigma({\mathcal{S}(t)}) \backslash \{0\} = e^{t\sigma(\cL)}$ ---
is \textit{not} true in general for $C_0$-semi\-groups (cf. \cite{EN00}), for stability purposes
we rely on the Gearhart-Pr\"uss theorem \cite{Gea78,Pr84}, which restricts our attention
to semigroups on Hilbert spaces (see also \cite{CrL,EN00}).
It states that any $C_0$-semigroup $\{{\mathcal{S}(t)}\}_{t\geq 0}$ on a Hilbert space $H$ is uniformly
exponentially stable if and only if its generator satisfies $s(\cL) <0$, and the following resolvent estimate holds:
\begin{equation*}
\sup_{\Real\, \lambda > 0} \|(\cL - \lambda)^{-1}\| < + \infty. 
\end{equation*}
This task is already substantially completed thanks to the general resolvent estimates of the previous 
section. It remains to be shown that the estimate
holds inside a half circle, with large radius, and on the projected
space, which is the content of the proof of Proposition \ref{propues} below.

It is known (see Kato \cite[Remark 6.23, p.184]{Kat80}) that if $\lambda \in \C$ is an eigenvalue of a closed 
operator $\cL : D \subset H \to H$ then $\overline{\lambda}$ is an eigenvalue of $\cL^*$ (formal adjoint) with 
the same geometric and algebraic multiplicities. Also, since $H^2$ and $L^2$ are reflexive Hilbert spaces, 
$\cL : D = H^2 \to L^2$ has a formal adjoint which is also densely defined and closed. Moreover, $\cL^{**} = 
\cL$ (cf. \cite[Theorem 5.29, p.168]{Kat80}).
Upon these observations we immediately have the following
\begin{lemma}
 \label{lempre1}
$\lambda = 0$ is an isolated, simple eigenvalue of the formal adjoint
\begin{equation*}
 (\mathcal{L}^\tau)^* : D(\mathcal{L}^\tau) = H^2(\R;\C^2) \to L^2(\R;\C^2),
\end{equation*}
and there exists an eigenfunction $(\Psi,\Phi)^\top \in D(\mathcal{L}^\tau)$ such that $(\mathcal{L}^\tau)^* (\Psi,\Phi)^\top = 0$.
\end{lemma}

Let us denote the inner product:
\begin{equation*}
 \Theta := \langle (U_x, V_x), (\Psi,\Phi) \rangle_{L^2} = \int_{-\infty}^{+\infty} \begin{pmatrix}
                                                                                     U_x \\ V_x
                                                                                    \end{pmatrix}^* \begin{pmatrix}
                                                                                     \Psi \\ \Phi
                                                                                    \end{pmatrix} \, dx.
\end{equation*}
It is not hard to see that $\Theta \neq 0$. Indeed, suppose by contradiction that $\Theta = 0$. Then $(\Psi,\Phi)^\top \in (\ker {\mathcal{L}^\tau})^\perp = {\mathrm{range}}({\mathcal{L}^\tau}^*)$. Hence, there exists $0 \neq (u,v)^\top \in D(\mathcal{L}^\tau) = H^2$ such that $(\mathcal{L}^\tau)^*(u,v)^\top = (\Psi,\Phi)^\top$. Thus, $((\mathcal{L}^\tau)^*)^2(u,v)^\top = 0$, which is a contradiction with $\lambda = 0$ being a simple eigenvalue of $(\mathcal{L}^\tau)^*$. Thus, we may define the Hilbert space $\tilde X \subset L^2(\R;\C^2)$ as the range of the spectral projection,
\begin{equation*}
 {\mathcal{P}}\begin{pmatrix}
               u \\ v
              \end{pmatrix}:= \begin{pmatrix}
               u \\ v
              \end{pmatrix} - \Theta^{-1}\langle (u,v),(\Psi,\Phi)\rangle_{L^2}\,\begin{pmatrix}
               U_x \\ V_x
              \end{pmatrix}.
\end{equation*}
In this fashion we project out the eigenspace spanned by the single eigenfunction $(U_x, V_x)^\top$. Outside 
this eigenspace, the semigroup decays exponentially, as we shall see next.

\subsection{Linear decay rates}

We now observe that on a reflexive Banach space, weak and weak$^*$ topologies coincide, and therefore the family of dual operators $\{ {\mathcal{S}(t)}^*\}_{t\geq 0}$, consisting of all the formal adjoints in $L^2$ is a $C_0$-semigroup as well (cf. \cite[p.44]{EN00}). Moreover, the infinitesimal generator of this semigroup is simply $(\mathcal{L}^\tau)^*$ (see  \cite[Corollary 10.6]{Pa83}). By semigroup properties we readily have
\begin{equation*}
 {\mathcal{S}(t)} \begin{pmatrix}
               U_x \\ V_x
              \end{pmatrix} = \begin{pmatrix}
               U_x \\ V_x
              \end{pmatrix}, \qquad {\mathcal{S}(t)}^* \begin{pmatrix}
                     \Psi \\ \Phi
                    \end{pmatrix}
 = \begin{pmatrix}
                     \Psi \\ \Phi
                    \end{pmatrix}.
\end{equation*}
As a result of these properties and the definition of the projector it is easy to verify that
\begin{equation*}
 {\mathcal{S}(t)} {\mathcal{P}} = {\mathcal{P}} {\mathcal{S}(t)}.
\end{equation*}
Hence $\tilde X$ is an ${\mathcal{S}(t)}$-invariant closed (Hilbert) subspace of $H^2(\R;\C^2)$. So we define the domain
\begin{equation*}
 \tilde{D} := \{ \mathbf{u} \in D \cap \tilde X \, : \, \mathcal{L}^\tau \mathbf{u} \in \tilde X \}
\end{equation*}
and the operator
\begin{equation*}
 \widetilde{\mathcal{L}^\tau} : \tilde D \subset \tilde X \to \tilde X,
\end{equation*}
\begin{equation*}
 \widetilde{\mathcal{L}^\tau} \mathbf{u} := \mathcal{L}^\tau \mathbf{u}, \qquad \mathbf{u} \in \tilde D,
\end{equation*}
as the restriction of $\mathcal{L}^\tau$ on $\tilde X$. Therefore, $\widetilde{\mathcal{L}^\tau}$ is a closed, densely defined operator on the Hilbert space $\tilde X$. Moreover, we observe that $0 \neq (U_x,V_x)^\top \in \ker {\mathcal{P}}$. Hence, $\lambda = 0 \notin \ptsp(\widetilde{\mathcal{L}^\tau})$. As a consequence of point spectral stability of $\mathcal{L}^\tau$ we readily obtain 
\begin{equation*}
 \sigma( \widetilde{\mathcal{L}^\tau} ) \subset \{ \lambda \in \C \, : \, \Real \, \lambda < 0 \},
\end{equation*}
and hence the spectral bound of $\widetilde{\mathcal{L}^\tau}$ is strictly negative,
$s(\widetilde{\mathcal{L}^\tau}) < 0$. By the above observations, we obtain the following
\begin{lemma}
The family of operators $\{{\widetilde{\mathcal{S}}(t)}\}_{t\geq 0}$, ${\widetilde{\mathcal{S}}(t)} : \tilde X \to \tilde X$, defined as 
\begin{equation*}
 {\widetilde{\mathcal{S}}(t)} \mathbf{u} := {\mathcal{S}(t)} {\mathcal{P}} \mathbf{u}, \qquad \mathbf{u} \in \tilde X, \;\; t \geq 0,
\end{equation*}
is a $C_0$-semigroup in the Hilbert space $\tilde X$ with infinitesimal generator $\widetilde{\mathcal{L}^\tau}$.
\end{lemma}

\begin{proof}
 The semigroup properties are inherited from those of ${\mathcal{S}(t)}$ in $L^2(\R;\C^2)$.
 That $\widetilde{\mathcal{L}^\tau}$ is the infinitesimal generator follows from Corollary in Section 2.2 of \cite{EN00}, p.61.
 \end{proof}

Finally, we apply Gearhart-Pr\"uss theorem.

\begin{proposition}[Uniform exponential stability]
\label{propues}
For each $\tau \in (0,\tau_m)$ there exist constants $C \geq 1$ and $\theta > 0$ such that
\begin{equation}\label{decayest}
	|{\widetilde{\mathcal{S}}(t)} {\mathbf{u}}|_{L^2}
	\leq C e^{-\theta t}|{\mathbf{u}}|_{L^2}, \qquad {\mathbf{u}} \in \tilde X, \;\; t \geq 0. 
\end{equation}
\end{proposition}

\begin{proof}
In view of the resolvent estimates \eqref{genres}, we can find a radius sufficiently large such that,
if $|\lambda| \geq R$ and $\Real \, \lambda \geq 0$, then
\begin{equation*}
	 \| (\mathcal{L}^\tau - \lambda)^{-1} \|_{L^2 \to L^2} \leq C
\end{equation*}
for some uniform $C > 0$.
Since $\widetilde{\mathcal{L}^\tau} = \mathcal{L}^\tau$ on the subspace $\tilde X \subset L^2(\R;\C^2)$,
the same estimate applies to $\widetilde{\mathcal{L}^\tau}$ outside that half circle. Inside, however,
thanks to (strict) point spectral  stability of the operator restricted to $\tilde X$,
the resolvent of $\widetilde{\mathcal{L}^\tau}$ is uniformly bounded inside the intersection
of any ball of finite radius and $\Real \, \lambda \geq 0$. 
We conclude that
\begin{equation*}
	\sup_{\Real \, \lambda > 0} \| (\widetilde{\mathcal{L}^\tau}- \lambda)^{-1}\|_{\tilde X \to \tilde X} \leq C,
\end{equation*}
for some $C > 0$ independent of $\lambda$. In addition, $s(\widetilde{\mathcal{L}^\tau}) < 0$.
Thus, a direct application of Gearhart-Pr\"uss theorem to the operator $\widetilde{\mathcal{L}^\tau}$
on the Hilbert space $\tilde X$ implies that the semigroup ${\widetilde{\mathcal{S}}(t)}$ is uniformly
exponentially stable, and that estimate \eqref{decayest} holds for some $C \geq 1$ and some $\theta > 0$.
\end{proof}

This result establishes the decaying properties of the linearized operator around the wave, that is, linear stability.
The latter can be summarized in the following

\begin{theorem}[Linear stability]\label{thm:linstab}
There exists a projection operator ${\mathcal{Q}} = I- {\mathcal{P}}$ with one-dimensional range
$\mathrm{span}\{(U_x, V_x)^\top\} \subset L^2(\R;\C^2)$ such that for any $t>0$
\begin{equation*}
	{\mathcal{S}(t)}{\mathcal{Q}}={\mathcal{Q}} {\mathcal{S}(t)}={\mathcal{Q}}
	\qquad\textrm{and}\qquad
	\|{\mathcal{S}(t)}(I-{\mathcal{Q}})\|\leq C\,e^{-\theta t}
\end{equation*}
for some $C, \theta>0$.
\end{theorem}

\subsection{Nonlinear stability}
\label{secnonlinear}
The proof of Theorem \ref{thm:nonlinstab} on nonlinear orbital stability of the traveling
fronts for \eqref{relAC} is a consequence of the linear decay estimates combined in a
smart way with the standard Duhamel representation formula. 
The main difficulty stems in the fact that a single traveling front is not isolated as a
stationary solution and it belongs to a one-dimensional manifold generated by applying
an arbitrary translation in space. 
This is a common feature of many autonomous evolutive PDEs when considered in the
whole space as a consequence of the underlying translation invariance of the corresponding
initial value problem.

At the linear level, such feature of the problem is expressed by the membership of $\lambda=0$
to the spectrum of the linearized operator and by the presence of a time-independent projection
term into the representation of the solution semigroup.
Converting such structure at the nonlinear level amounts in identifying a nonlinear projection
operator describing the convergence of a given perturbed initial datum to a translate of the
original front. 
A possible approach is based on the application of the Implicit Function Theorem in Banach spaces
and it has been used by Sattinger in the classical paper \cite{Satt76}.
For the sake of clarity, we first present here a restyled version of this approach in the framework of Hilbert spaces, as needed in our case, and
then apply it to prove Theorem \ref{thm:nonlinstab}.

Let $\mathcal{W}$ be a Hilbert space with norm $|\cdot|_{{}_{\mathcal{W}}}$
and let $B_r(\overline{W})$ be the open ball with center $\overline{W}$ and radius $r$.
Let $F$ be a smooth function from $\mathcal{D}\subset\mathcal{W}$ into $\mathcal{W}$
such that $F(\overline{W})=0$ for some $\overline{W}\in \mathcal{D}$.
Additionally, let us assume that, for some $r>0$, there holds
\begin{equation*}
	\{W\in \mathcal{W}\,:\,F(W)=0\}\cap \{|W-\overline{W}|_{{}_{\mathcal{W}}}<r\}=\phi(I)
\end{equation*}
for some smooth function $\phi\,:\,I\to\mathcal{W}$, $I \subset \R$ an open interval.
Without loss of generality, we may assume $0\in I$ and $\phi(0)=\overline{W}$.

Let $W=W(t;W_0)$ be the solution to the abstract Cauchy problem
\begin{equation}\label{cauchy2}
	\frac{dW}{dt}=F(W), \qquad\qquad W(0)=W_0\in\mathcal{D}.
\end{equation}
By assumption, there holds  $W(t;\phi(\delta))=\phi(\delta)$ for any $t$.

The linearized problem at $\phi(\delta)$ is
\begin{equation}\label{lincauchy2}
	\frac{dZ}{dt}=dF(\phi(\delta))Z, \qquad\qquad Z(0)=Z_0\in\mathcal{D}.
\end{equation} 
Differentiating with respect to $\delta$ the relation $F(\phi(\delta))=0$ for $\delta\in I$, we infer
\begin{equation*}
	dF(\phi(\delta))\,\phi'(\delta)=0,
\end{equation*}
showing that $0\in \sigma\bigl(dF(\phi(\delta))\bigr)$ and that
$r(\delta):=\phi'(\delta)$ is a right eigenvector of $dF(\phi(\delta))$.
Let us denote by $\ell(\delta)$ the unique left eigenvector of $dF(\phi(\delta))$
such that $\ell(\delta)\cdot r(\delta)=1$.
Equivalently, $\ell(\delta)$ can be defined as the unique element in the kernel of
the adjoint operator $dF(\phi(\delta))^\ast$ satisfying the normalization condition
$\ell(\delta)\cdot r(\delta)=1$.
We also set for $\delta\in I$
\begin{equation*}
	P(\delta):=r(\delta)\otimes \ell(\delta),\qquad Q(\delta):=I-P(\delta).
\end{equation*}
In particular, there hold
\begin{equation*}
	dF(\phi)P=P\,dF(\phi)=0, \qquad\textrm{and}\qquad
	dF(\phi)Q=Q\,dF(\phi)=dF(\phi),
\end{equation*}
where the dependency on $\delta$ has been omitted for shortness.
\vskip.25cm

We assume the following hypotheses.
\vskip.15cm

{\bf H1.} There exist $C, \theta>0$ such that the solution $Z=Z(t;Z_0,\delta)$
to \eqref{lincauchy2} is such that
\begin{equation}\label{lindecay2}
	|Q(\delta)Z(t;Z_0,\delta)|\leq Ce^{-\theta t}|Q(\delta)Z_0|
\end{equation}
for any $Z_0\in\mathcal{D}$.
\vskip.15cm

{\bf H2.} The function $\phi$ is differentiable at $\delta=0$ and there exist
$C, \delta_0, \gamma>0$ such that 
\begin{equation}\label{regophi}
	|\phi(\delta)-\phi(0)-\phi'(0)\delta|_{{}_{\mathcal{W}}}\leq C\delta^{1+\gamma},
\end{equation}
for $|\delta|<\delta_0$.
\vskip.15cm

{\bf H3.} There exist $C, M, \delta_0,\gamma>0$ such that the function $F$ is differentiable at $\phi(\delta)$
for any $\delta\in(-\delta_0,\delta_0)$ and
\begin{equation}\label{regoF}
	|F(\phi(\delta)+W)-F(\phi(\delta))-dF(\phi(\delta))W|_{{}_{\mathcal{W}}}
		\leq C|W|_{{}_{\mathcal{W}}}^{1+\gamma},
\end{equation}
for $|\delta|<\delta_0$ and $|W|_{{}_{\mathcal{W}}}\leq M$.

\begin{theorem}\label{thm:abstract}
Assume that hypotheses {\bf H1}, {\bf H2} and {\bf H3} hold.
Then there exists $\varepsilon>0$ such that for any $W_0\in B_\varepsilon(\bar W)$
there exists $\delta\in I$ for which the solution $W(t;W_0)$ to \eqref{cauchy2} satisfies
\begin{equation}\label{nonlindecay2}
	|W(t;W_0)-\phi(\delta)|_{{}_{\mathcal{W}}}\leq C|W_0-\overline{W}|_{{}_{\mathcal{W}}}\,e^{-\theta\,t}
\end{equation}
for some $C,\theta>0$
\end{theorem}

\begin{proof}
Given $W_0\in \mathcal{W}$, let $w_0\in \mathcal{W}$ be such that
$W_0=\overline{W}+\varepsilon w_0$ where  $\varepsilon:=|W_0-\overline{W}|_{{}_{\mathcal{W}}}$
and let the solution $W$ to \eqref{cauchy2} be decomposed as
\begin{equation*}
	W=\phi(\varepsilon\eta)+\varepsilon\,w
\end{equation*}
with $\eta=\eta(\varepsilon)$ to be determined later, where the function $w$ solves
\begin{equation}\label{cauchypert2}
	\frac{dw}{dt}=dF(\phi(\varepsilon\eta))w+\varepsilon^\gamma R (\eta,w;\varepsilon),
	\quad
	w(0)=w_0-\phi'(0)\eta-\varepsilon^\gamma \psi(\eta;\varepsilon)
\end{equation}
where
\begin{equation*}
	\begin{aligned}
		R (\eta,w;\varepsilon)&:=\varepsilon^{-1-\gamma}
			\left\{F(\phi(\varepsilon\eta)+rw)-F(\phi(\varepsilon\eta))-dF(\phi(\varepsilon\eta))rw\right\},\\
		\psi (\eta;\varepsilon)&:=\varepsilon^{-1-\gamma}
			\left\{\phi(\varepsilon\eta)-\phi(0)-\phi'(0)\varepsilon\eta\right\}.\\
	\end{aligned}
\end{equation*}
Decomposing $w$ as
\begin{equation*}
	w=\alpha\,\phi'(\varepsilon\eta)+\omega, \qquad\qquad\textrm{where}\qquad
		\alpha:=\ell(\varepsilon\eta)\cdot w,\quad \omega:=Q(\varepsilon\eta) w,
\end{equation*}
and setting
\begin{equation*}
	S(\eta,\alpha,\omega;\varepsilon):=R(\eta,\alpha\,\phi'(\varepsilon\eta)+\omega;\varepsilon),
\end{equation*}
the unknowns $\alpha$ and $\omega$ solve
\begin{equation}
	\left\{\begin{aligned}
		\frac{d\alpha}{dt}&=\varepsilon^\gamma \,\ell(\varepsilon\eta)\cdot S(\eta,\alpha,\omega;\varepsilon),\\
		\frac{d\omega}{dt}&=dF(\phi(\varepsilon\eta))\omega
			+\varepsilon^\gamma Q(\varepsilon\eta) S(\eta,\alpha,\omega;\varepsilon),
	\end{aligned}\right.
\end{equation}
with initial conditions
\begin{equation}
	\left\{\begin{aligned}
		\alpha(0)&=\ell(\varepsilon\eta)\cdot\bigl(w_0-\phi'(0)\eta-\varepsilon^\gamma \psi(\eta;\varepsilon)\bigr),\\
		\omega(0)&=Q(\varepsilon\eta)\bigl(w_0-\phi'(0)\eta-\varepsilon^\gamma \psi(\eta;\varepsilon)\bigr).
	\end{aligned}\right.
\end{equation}
Therefore, the following relations hold
\begin{equation}\label{alphaomega}
	\begin{aligned}
	\alpha(t)&=\ell\cdot\bigl(w_0-\phi'(0)\eta\bigr)
		-\varepsilon^\gamma\left\{\ell\cdot\psi-\int_0^t \ell\cdot S\,d\tau\right\},\\
	\omega(t)&=e^{dF(\phi)t}Q\bigl(w_0-\phi'(0)\eta\bigr)\\
		&\hskip2.5cm 
			-\varepsilon^\gamma\left\{e^{dF(\phi)t}Q\psi-\int_0^t e^{dF(\phi)(t-\tau)}QS\,d\tau\right\}.
	\end{aligned}
\end{equation}
where $\ell=\ell(\varepsilon\eta), Q=Q(\varepsilon\eta), \phi=\phi(\varepsilon\eta), \psi=\psi(\eta;\varepsilon)$ and $S=S(\eta,\alpha,\omega;\varepsilon)$.

Next, we require the value $\eta=\eta(\varepsilon)$ to be such that $\alpha(+\infty)=0$ that is
\begin{equation*}
	\ell\cdot\bigl(-w_0+\phi'(0)\eta\bigr)
	+\varepsilon^\gamma\left\{\ell\cdot\psi -\int_0^{+\infty}  \hskip-.25cm  \ell\cdot S\,d\tau\right\}=0.
\end{equation*}
Thus, the triple $(\eta,\alpha,\omega)$ has to be such that
\begin{equation}\label{implrel}
	\mathcal{F}(\eta,\alpha,\omega;\varepsilon)+\varepsilon^\gamma \mathcal{G}(\eta,\alpha,\omega;\varepsilon) =0
\end{equation}
where 
\begin{equation*}
	\begin{aligned}
	\mathcal{F}&=\Bigl(\ell\cdot\bigl(-w_0+\phi'(0)\eta\bigr),
		\alpha, \omega-e^{dF(\phi)t} Q\bigl(w_0-\phi'(0)\eta\bigr) \Bigr)\\
	\mathcal{G}&=\Bigl(\ell\cdot\psi-\int_0^{+\infty} \hskip-.25cm \ell\cdot S\,d\tau ,
		\int_t^{+\infty} \hskip-.25cm \ell\cdot S\,d\tau , e^{dF(\phi)t} Q\psi-\int_0^t e^{dF(\phi)(t-\tau)} QS\,d\tau  \Bigr).
	\end{aligned}
\end{equation*}
We want to show that, for small $\varepsilon$, the implicit relation \eqref{implrel} defines a function
$\varepsilon\mapsto (\eta,\alpha,\omega)$.
To prepare for the application of the Implicit function theorem in Banach spaces
(among others, see \cite{AmbrProd95}), let us introduce an appropriate
functional setting.
Given $\theta>0$ and a Banach space $\mathcal{Y}$ with norm $|\cdot|_{{}_{\mathcal{Y}}}$, set
\begin{equation*}
	C^0_\theta(\R_+;\mathcal{Y}):=\left\{f\in C^0(\R_+;\mathcal{Y})\,:\,
		\sup_{t>0} e^{\theta\,t}|f(t)|_{{}_{\mathcal{Y}}}<+\infty\right\}.
\end{equation*}
Then, let us consider the Banach space
$\mathcal{X}=\R\times C^0_\theta(\R_+;\R)\times C^0_\theta(\R_+;\mathcal{W})$ 
with norm
\begin{equation*}
	\left\|(\eta,\alpha,\omega)\right\|_{{}_{\mathcal{X}}}
		:=|\eta|+\sup_{t>0} e^{\theta\,t}\bigl(|\alpha(t)|+|\omega(t)|_{{}_{\mathcal{W}}}\bigr).
\end{equation*}
Choosing $\theta$ as in \eqref{lindecay2}, for any $M>0$, the function $\mathcal{F}$
maps the set $\mathcal{X}\times (-\varepsilon,\varepsilon)$ into $\mathcal{X}\cap\{|\eta|\leq M\}$
for $\varepsilon$ sufficiently small, since
\begin{equation*}
	e^{\theta\,t} |e^{dF(\phi)t} Q\bigl(w_0-\phi'(0)\eta\bigr)|_{{}_{\mathcal{W}}}
	\leq  C\bigl|Q\bigl(w_0-\phi'(0)\eta\bigr)\bigr|_{{}_{\mathcal{W}}}<+\infty.
\end{equation*}
Moreover, as a consequence of the estimate
\begin{equation*}
	\begin{aligned}
	e^{\theta\,t}\int_t^{+\infty}
		&|S(\eta,\alpha(\tau),\omega(\tau);r)|_{{}_{\mathcal{W}}}\,d\tau \leq C\,e^{\theta\,t}\int_t^{+\infty}
			|\alpha(\tau)\phi'(\varepsilon\eta)+\omega(\tau)|_{{}_{\mathcal{W}}}^{1+\gamma}\,d\tau\\
		&\leq C\,e^{\theta\,t} \int_t^{+\infty} e^{-(1+\gamma)\theta \tau}\left\{e^{\theta \tau}\left(|\alpha(\tau)|
			+|\omega(\tau)|_{{}_{n}}\right)\right\}^{1+\gamma}\,d\tau\\
		&\leq C\,e^{-\gamma\theta\,t}\left\|(0,\alpha,\omega)\right\|_{{}_{\mathcal{X}}}^2,
	\end{aligned}
\end{equation*}
also the function $\varepsilon^\gamma\mathcal{G}$ maps $\mathcal{X}\times (-\varepsilon,\varepsilon)$
into $\mathcal{X}\cap\{|\eta|\leq M\}$ for any $M>0$ and for $\varepsilon$ sufficiently small.
Moreover, the smoothness of the functions $\ell, Q, \psi, S$ with respect to their arguments
guarantees that the map $\mathcal{F}+\varepsilon^\gamma\mathcal{G}$ is differentiable.

For $r=0$, there holds
$\mathcal{F}(\eta,\alpha,\omega;0)=\bigl(\ell(0)\cdot w_0-\eta,\alpha,\omega-Q(0) w_0\bigr)$,
thus \eqref{implrel} is satisfied if and only if
\begin{equation*}
	\eta=\ell(0)\cdot w_0,\qquad \alpha=0,\qquad \omega=Q(0)w_0.
\end{equation*}
In order to apply Implicit Function Theorem, it is sufficient to observe that
\begin{equation*}
	\frac{\partial\left(\mathcal{F}+\varepsilon^\gamma\mathcal{G}\right)}
		{\partial(\eta,\alpha,\omega)}\Bigr|_{\varepsilon=0}
		=\frac{\partial \mathcal{F}}{\partial(\eta,\alpha,\omega)}\Bigr|_{\varepsilon=0}
			=\begin{pmatrix} 1 &0 &0\\ 0 &I	&0\\ 0 &0 &I \end{pmatrix}
\end{equation*}
since
\begin{equation*}
	\frac{\partial}{\partial \eta}\ell(\varepsilon\eta)\Bigr|_{\varepsilon=0}
		=\frac{\partial}{\partial \eta}Q(\varepsilon\eta)\Bigr|_{\varepsilon=0}=0.
\end{equation*}
Thus, in a neighborhood of $\varepsilon=0$, there exist a smooth function $\Xi$ with values in
a neighborhood of $(\ell(0)\cdot w_0,0,Q(0)w_0)\in \mathcal{X}$
such that
\begin{equation}\label{implrel2}
	\mathcal{F}+\varepsilon\mathcal{G}=0
	\qquad\textrm{if and only if}\qquad
	(\eta,\alpha,\omega)=\Xi(r).
\end{equation}
The function $\Xi$ is locally bounded, $\|\Xi(r)\|_{\mathcal{X}}\leq C$ for $\varepsilon$ small, and thus
\begin{equation*}
	|w(t)|=|\alpha(t)\phi'(\varepsilon\eta)+\omega(t)|\leq C\,e^{-\theta\,t}.
\end{equation*}
Recalling that $W=\phi(\varepsilon\eta)+\varepsilon w$,  the decay estimate \eqref{nonlindecay2} follows.
\end{proof}

With Theorem \ref{thm:abstract} at hand, we are able to provide the proof of Theorem \ref{thm:nonlinstab}.
For the reader's convenience, let us briefly retrace the path toward nonlinear stability.
The traveling wave $(U,V)$ propagates with a specific speed $c$.
Thus, considering a reference frame moving with such speed,
we obtain the nonlinear system \eqref{nlsystw} for the perturbation variables, for which the Cauchy problem can be written as
\begin{equation}
\label{relACmove2}
\begin{aligned}
 \partial_{t}\begin{pmatrix} u\\ v\end{pmatrix}
	 &= - {\mathbf{B}}^{-1}\left( \mathbf{A} \,\partial_{x}\begin{pmatrix} u\\ v\end{pmatrix}
	 	+\begin{pmatrix} f(U)- f(u+U)\\ v\end{pmatrix}\right),\\
\begin{pmatrix} u\\ v\end{pmatrix}(0) &= \begin{pmatrix} u_0 - U\\ v_0 - V\end{pmatrix}
\end{aligned}
\end{equation}
where $\mathbf{A}$ and $\mathbf{B}$ are defined in \eqref{defABC}, and $(u_0,v_0)$ is the (unperturbed) initial data of Theorem \ref{thm:nonlinstab}.
Problem \eqref{relACmove2} corresponds to \eqref{cauchy2} in the general framework
previously considered in Theorem \ref{thm:abstract}. Therefore, we are able to make the following identifications:

{ 1.} the Hilbert space $\mathcal{W}$ is $H^1(\R;\R^2)$; the steady state $\overline{W}$ is $\overline{W}=0$, and
the function $\phi$ is defined by $\phi(\delta):=(U,V)(\cdot+\delta) -(U,V)(\cdot) $ with $\delta\in \R$; observe that for each $\delta \in \R$ fixed,
\[
\begin{aligned}
 |\phi(\delta)|_{{H^1}}^2 &= \int_\R |(U,V)(\zeta + \delta) - (U,V)(\zeta)|^2 \, d\zeta + \int_\R |(U_x,V_x)(\zeta + \delta) - (U_x,V_x)(\zeta)|^2 \, d\zeta\\
&= \int_\R |(U_x, V_x)(\hat \theta \zeta)|^2 \delta^2 \, d\zeta + \int_\R |(U_{xx}, V_{xx})(\hat \theta \zeta)|^2 \delta^2 \, d\zeta\\
&\leq C_\delta |(U_x,V_x)|_{{H^1}}^2,
\end{aligned}
\]
for some $\hat \theta \in (0,\delta)$, showing that $\phi(\delta) \in \mathcal{W}$.\par
{ 2.} the linearized equation (corresponding to the one in \eqref{lincauchy2}) is
\begin{equation*}
	\partial_{t}\begin{pmatrix} u\\ v\end{pmatrix}
	 = - {\mathbf{B}}^{-1}\Big( \mathbf{A}\,\partial_{x}
	 	+ \mathbf{C}(x) \Big) \begin{pmatrix} u\\ v\end{pmatrix};
\end{equation*}
\indent
{ 3.} the remainder, for which the estimate \eqref{regoF} has to be proved, is
\begin{equation*}
	\begin{pmatrix} u\\ v\end{pmatrix}	\quad \mapsto\quad
		\begin{pmatrix} R(U;u) \\ 0 \end{pmatrix}
		:=\begin{pmatrix} f(U+u) - f(U) - f'(U)u \\ 0 \end{pmatrix}.
\end{equation*}
Assuming $f\in C^3$, we next show that hypotheses {\bf H2} and {\bf H3} are verified
for the function space $\mathcal{W}=H^1(\R;\R^2)$.

First we verify \eqref{regophi}. Denoting by $\Phi$ the couple $(U,V)$, there holds
\begin{equation*}
	\begin{aligned}
	|\phi(\delta) - \phi(0) - \phi'(0)\delta|_{{L^2}}^2 &= \int_{\R}\left|\Phi(x+\delta)-\Phi(x)-\Phi_x(x)\delta\right|^2\,dx\\
		&=\delta^2\int_{\R}\left|\int_0^1 \left(\Phi_x(x+\theta\delta)-\Phi_x(x)\right)d\theta\right|^2\,dx\\
	&\leq \delta^2\int_{\R}\int_0^1 \left|\Phi_x(x+\theta\delta)-\Phi_x(x)\right|^2 d\theta\,dx\\
		&\leq \delta^4\int_{\R}\int_0^1\int_0^1 |\Phi_{xx}|^2 d\eta\,d\theta\,dx,
	\end{aligned}
\end{equation*}
and thus
\begin{equation*}
	|\phi(\delta) - \phi(0) - \phi'(0)\delta|_{{L^2}} \leq \delta^2 |\Phi_{xx}|_{{L^2}}.
\end{equation*}
A similar estimate can be obtained by differentiating with respect to $x$, so that
\begin{equation*}
	|\phi(\delta) - \phi(0) - \phi'(0)\delta|_{{H^1}} \leq \delta^2 |\Phi_{xx}|_{{H^1}}.
\end{equation*}
To prove estimate \eqref{regoF}, we first observe that
\begin{equation*}
	\begin{aligned}
	R(U;u)&=\left\{\int_0^1 f''(U+\theta u)(1-\theta)\,d\theta\right\} u^2;\\
	\partial_{x}R(U;u)
		&=\left\{\int_0^1 f'''(U+\theta u)\theta(1-\theta)\,d\theta\right\} u^2\,\partial_{x} u\\
		&\quad +2\left\{\int_0^1 f''(U+\theta u)(1-\theta)\,d\theta\right\} u\,\partial_{x} u .
	\end{aligned}
\end{equation*}
Thus, for $u\in H^1(\R)$, taking into account the embedding $H^1(\R)\subset L^\infty(\R)$,
there holds
\begin{equation*}
	|R(U;u)|_{{}_{L^2}}\leq C\,|u|_{{}_{L^2}}^2,\qquad
	|\partial_{x} R(U;u)|_{{}_{L^2}}\leq C\,|u|_{{}_{H^1}}^2,
\end{equation*}
where the constant $C$ depends on $f, U$ and the $L^\infty-$norm of  $u$,
so that, in particular, estimate \eqref{regoF} holds with $\gamma=1$.

Finally, thanks to Theorem \ref{thm:linstab}, also hypothesis {\bf H1} is verified,
so that Theorem \ref{thm:abstract} applies and  Theorem \ref{thm:nonlinstab} follows.


\section{Numerical experiments}\label{sect:numerics}

In this Section, we present some numerical experiment on system \eqref{relAC},
based on the observation that it can be rewritten as the weakly coupled semilinear
hyperbolic system (a reactive version of the Goldstein--Kac model for correlated random walk):
\begin{equation}\label{cineAC}
	\left\{\begin{aligned}
	\partial_{t} u_--\varrho\partial_{x} u_-
		&= \tfrac1{2}\,\tau^{-1}(-u_-+u_+)+\tfrac12 f(u_++u_-),\\
	\partial_{t} u_++\varrho\partial_{x} u_+
		&= \tfrac1{2}\tau^{-1}(u_--u_+)+\tfrac12 f(u_++u_-),
	\end{aligned}\right.
\end{equation}
where the coefficient $\varrho$ and the unknowns $u_\pm$ are given by
\begin{equation*}
	\varrho:=1/\sqrt{\tau},\qquad
	u_-:=\tfrac{1}{2}\left(u+\varrho^{-1} v\right),\qquad
	u_+:=\tfrac{1}{2}\left(u-\varrho^{-1} v\right).
\end{equation*}
Inverting the equality, we infer the relations $u=u_++u_-$ and $v=\varrho(u_--u_+)$.

Fixed the mesh size $\dx>0$, we discretize the space by approximating the first order space
derivatives in an upwind fashion. 
Thus, setting $r_j\approx u_-(j\,\dx,t)$ and $s_j\approx u_+(j\,\dx,t)$, we obtain
\begin{equation}\label{semidiscrete}
	\left\{\begin{aligned}
	\frac{dr_j}{dt}&=\frac{\varrho}{\dx}\left(r_{j+1}-r_{j}\right)
		+ \frac1{2\tau}\left(-r_j+s_j\right)+\frac12 f(r_j+s_j),\\
	\frac{ds_j}{dt}&=-\frac{\varrho}{\dx}\left(s_{j}-s_{j-1}\right)
		+ \frac1{2\tau}\left(r_j-s_j\right)+\frac12 f(r_j+s_j).
	\end{aligned}\right.
\end{equation}
Let us stress that, setting $u_j:=r_j+s_j$ and $v_j:=\varrho(r_j-s_j)$, we infer
a semi-discrete version of \eqref{relAC}
\begin{equation*}
	\left\{\begin{aligned}
	\frac{du_j}{dt}&=\frac{1}{2}\varrho\,\dx\,\frac{u_{j+1}-2u_j+u_{j-1}}{\dx^2}
		+\frac{v_{j+1}-v_{j-1}}{2\dx}+f(u_j),\\
	\frac{dv_j}{dt}&=\frac{1}{2}\varrho\,\dx\,\frac{v_{j+1}-2v_j+v_{j-1}}{\dx^2}
		+\frac{1}{\tau}\left(\frac{u_{j+1}-u_{j-1}}{2\dx}-v_j\right),
	\end{aligned}\right.
\end{equation*}
which formally corresponds to
\begin{equation*}
	\left\{\begin{aligned}
	u_t -v_x &=\nu\, u_{xx} u+f(u),\\
	\tau v_t - u_x &=\tau\nu\,v_{xx} - v,
	\end{aligned}\right.
	\qquad\qquad\textrm{where}\quad
	\nu:=\tfrac{1}{2}\varrho\,\dx,
\end{equation*}
so that we expect the appearance of a numerical viscosity with strength measured
by the parameter $\varrho$.

Next, fixed the time step $\dt>0$, we discretize the time derivative in \eqref{semidiscrete}
by means of an implicit-explicit approach, leaded by a simplicity criterion suggesting to 
discretize implicitly only the linear terms
\begin{equation*}
	\left\{\begin{aligned}
	\frac{r_j^{n+1}-r_j^{n}}{\dt}&=\frac{\varrho}{\dx}\bigl(r_{j+1}^{n+1}-r_{j}^{n+1}\bigr)
		+ \frac1{2\tau}\bigl(-r_j^{n+1}+s_j^{n+1}\bigr)+\frac12 f(r_j^{n}+s_j^{n}),\\
	\frac{s_j^{n+1}-s_j^{n}}{\dt}&=-\frac{\varrho}{\dx}\bigl(s_{j}^{n+1}-s_{j-1}^{n+1}\bigr)
		+ \frac1{2\tau}\bigl(r_j^{n+1}-s_j^{n+1}\bigr)+\frac12 f(r_j^{n}+s_j^{n}).
	\end{aligned}\right.
\end{equation*}
Fully implicit schemes have been tested with no significant advantage in the approximation,
but with a significant increase of the computational time. 

Setting
\begin{equation*}
	\alpha=\varrho\frac{\dt}{\dx},\qquad \beta=\frac{\dt}{2\tau},\qquad f_j^n=f(r_j^n+s_j^n),
\end{equation*}
and with an upwind discretization of the space derivatives, we end up with 
\begin{equation}\label{vectSch}
	\begin{pmatrix}
	\left(1+\beta\right)\mathbb{I}-\alpha\,\mathbb{D}_+ & -\beta\,\mathbb{I} \\
	-\beta\,\mathbb{I} &\left(1+\beta\right)\mathbb{I}+\alpha\,\mathbb{D}_- \\
	\end{pmatrix}
	\begin{pmatrix}	r^{n+1} \\ s^{n+1} \end{pmatrix}=
	\begin{pmatrix} r^{n}+f^{n}\dt /2\\ s^{n}+f^{n}\dt /2 \end{pmatrix}
\end{equation}
where the matrices $\mathbb{I}, \mathbb{D}_\pm$ are given by
\begin{equation*}
	\mathbb{I}=(\delta_{i,j}),\quad
	\mathbb{D}_+=(\delta_{i+1,j}-\delta_{i,j}),\quad
	\mathbb{D}_-=(\delta_{i,j}-\delta_{i,j+1})
\end{equation*}
(here $\delta_{i,j}$ is the standard Kronecker symbol).
The block-matrix in \eqref{vectSch} is invertible, since its spectrum is contained in
the complex half plane $\{\lambda\in\C\,:\,\Real\lambda\geq1\}$ as a consequence
of the Ger\v sgorin criterion.

A direct manipulation of \eqref{vectSch} gives
\begin{equation}\label{finalSch}
	\begin{aligned}
	r^{n+1}&=(\mathbb{S}-\alpha^2\mathbb{D}_-\mathbb{D}_+)^{-1}
		\Bigl\{[(1+\beta)\mathbb{I}+\alpha\mathbb{D}_-]r^n+\beta s^n\\
		&\hskip4.5cm+\tfrac12[(1+2\beta)\mathbb{I}+\alpha\mathbb{D}_-]f^n\dt\Bigr\}\\
	s^{n+1}&=(\mathbb{S}-\alpha^2\mathbb{D}_+\mathbb{D}_-)^{-1}
		\Bigl\{\beta r^n+[(1+\beta)\mathbb{I}-\alpha\mathbb{D}_+]s^n\\
		&\hskip4.5cm+\tfrac12[(1+2\beta)\mathbb{I}-\alpha\mathbb{D}_+]f^n\dt\Bigr\}
	\end{aligned}
\end{equation}
where $\mathbb{S}$ is the symmetric matrix
\begin{equation*}
	\mathbb{S}:=(1+2\beta)\mathbb{I}+\alpha(1+\beta)(\mathbb{D}_--\mathbb{D}_+)
\end{equation*}
To start with, we test the algorithm by analyzing its capability to recover the correct
wave speeds $c_\ast$ of the front connecting the stable states $0$ and $1$.
Following \cite{LeVYee90}, we introduce an {\it average speed} of the numerical solution
at time $t^n$ defined by
\begin{equation}\label{numerAve}
  c^n=\frac{1}{\dt}\mathbf{1}\cdot(u^{n}-u^{n+1})
  	=\frac{1}{\dt}\sum_{j} (u^{n}_{j}-u^{n+1}_{j})
\end{equation}
where $\mathbf{1}=(1,\dots,1)$. 
Indeed, for any differentiable function $\phi$ with asymptotic states $\phi_\pm$
and derivative integrable in $\R$, and for $h\in\R$, there holds
\begin{equation*}
	\begin{aligned}
	\int_{\R} \left(\phi(x+h)-\phi(x)\right)dx
		&=h\int_{\R} \int_0^1 \frac{d\phi}{dx}(x+\theta h)dh\,dx\\
		&=h\int_0^1\int_{\R}  \frac{d\phi}{dx}(x+\theta h)dx\,dh=h[\phi]
	\end{aligned}
\end{equation*}
where $[\phi]:=\phi(+\infty)-\phi(-\infty)$, so that for $h=-c\,\dt$, we infer
\begin{equation*}
	c=\frac{1}{[\phi]\,\dt}\int_{\R} \left(\phi(x)-\phi(x-c\,\dt)\right)dx.
\end{equation*}
As a test case, we consider the usual cubic function $f(u)=u(u-\alpha)(1-u)$ with $\alpha\in(0,1)$.
Our aim is to compare the values for the propagation speed $c_\ast$ as obtained by means
of the shooting argument (see Section \ref{sect:existence}) and the one given by 
calculating \eqref{numerAve} for the solution to the initial-value problem with an increasing
datum connecting $0$ and $1$ and computing $c^n$ at a time $t$ so large that stabilization
of the speed of propagation of the numerical solution is reached.

To start with, we test three different choices for the couple $(\tau, \alpha)$ for different values
of $\dx$ and $\dt$, where the range of variation of $\tau$ has been chosen so that the condition
$\tau\,f'(u)<1$ is satisfied for all the values of the unstable zero $\alpha$.

\begin{table}\centering
\caption{Riemann problem with jump at $\ell/2$, $\ell=25$.
Relative error for three different cases ($T$ final time):
A. $\tau=1$, $\alpha=0.9$, $c_\ast=0.5646$, $T=40$;
B. $\tau=2$, $\alpha=0.6$, $c_\ast=0.1737$, $T=30$;
C. $\tau=4$, $\alpha=0.7$, $c_\ast=0.3682$, $T=35$.
\label{tab:numerspeeds3}}
{\begin{tabular}{@{}r|c|r|r|r|r|r@{}} 
 			&$\dx$ 	&$2^0$	&$2^{-1}$	&$2^{-2}$	&$2^{-3}$	&$2^{-4}$	\\ \hline
 			&A 		&0.1664	&0.0787	&0.0325	&0.0091	&0.0018 	\\
$\dt=10^{-1}$ 	&B 		&0.0383	&0.0306	&0.0241	&0.0198	&0.0175 	\\
 			&C 		&0.1527 	&0.1144	&0.0818	&0.0581	&0.0442 	\\ 
\hline
			&A 		&0.1751	&0.0876	&0.0417	&0.0186	&0.0079	\\
$\dt=10^{-2}$	&B 		&0.0275	&0.0196	&0.0128	&0.0084	&0.0061	\\
			&C 		&0.1420	&0.1018	&0.0684	&0.0457	&0.0339	\\
\hline
			&A 		&0.1760 	&0.0885	&0.0427	&0.0196	&0.0089	\\
$\dt=10^{-3}$	&B 		&0.0265 	&0.0184 	&0.0117	&0.0072	&0.0049	\\
			&C 		&0.1411	&0.1006	&0.0670	&0.0441	&0.0321
\end{tabular}}
\end{table}

From Table \ref{tab:numerspeeds3}, we note that the smallness of the space mesh $\dx$ is more
relevant than the corresponding time-step value $\dt$.

Requiring to detect the correct speed value with an error that is always less than
5\% of the effective value, we heuristically determine the choice $\dx=2^{-3}$ and
$\dt=10^{-2}$, that will be used for subsequent numerical experiments. 
For such a choice, we record the results in Table \ref{tab:numerspeeds0}
for different choices of $\alpha$ and $\tau=1$ and $\tau=4$ together with
the corresponding relative error.

\begin{table}\centering
\caption{Final average speed \eqref{numerAve} and relative error
with respect to $c_\ast$ given in Sect.\ref{sect:existence}
($N=400$, $\dx=0.125$, $\dt=0.01$, $\ell=25$, $T=40$).
\label{tab:numerspeeds0}}
{\begin{tabular}{@{}r|c|c|c|c@{}} 
 			&$\alpha=0.6$	&$\alpha=0.7$	&$\alpha=0.8$	&$\alpha=0.9$	\\ \hline
$\tau=1$ 		&0.1580		&0.3096		&0.4497		&0.5751 	\\
			&{\scriptsize 0.0101}	
						&{\scriptsize 0.0118} 
									&{\scriptsize 0.0145}
												&{\scriptsize 0.0186} \\
\hline
$\tau=4$		&0.2102		&0.3533		&0.4337		&0.4825	\\
			&{\scriptsize 0.0396}
						&{\scriptsize 0.0404}
									&{\scriptsize 0.0365}
												&{\scriptsize 0.0118}
\end{tabular}}
\end{table}

Considering different form for matrices $\mathbb{D}_\pm$ giving a second order approximation of 
the derivatives, such as
\begin{equation*}
	\mathbb{D}_+=\left(-\tfrac12\delta_{i+2,j}+2\delta_{i+1,j}-\tfrac32\delta_{i,j}\right),\qquad
	\mathbb{D}_-=\left(\tfrac32\delta_{i,j}-2\delta_{i,j+1}+\tfrac12\delta_{i,j+2}\right),
\end{equation*}
the speed approximation gain in accuracy, as reported in Table \ref{tab:secondorder}, that shows
an increase of one order. 

\begin{table}\centering
\caption{Second order in space.
Final average speed \eqref{numerAve} and relative error
with respect to $c_\ast$ given in Sect.\ref{sect:existence}
($N=400$, $\dx=0.125$, $\dt=0.01$, $\ell=25$, $T=40$).
\label{tab:secondorder}}
{\begin{tabular}{@{}r|c|c|c|c@{}} 
 			&$\alpha=0.6$	&$\alpha=0.7$	&$\alpha=0.8$	&$\alpha=0.9$	\\ \hline
$\tau=1$ 		&0.1560		&0.3052		&0.4421	 	&0.5630	\\
			&{\scriptsize 0.0025}	
						&{\scriptsize 0.0025} 
									&{\scriptsize 0.0026}
												&{\scriptsize 0.0029} \\
\hline
$\tau=4$		&0.2184		&0.3672		&0.4485		&0.4885	\\
			&{\scriptsize 0.0022}
						&{\scriptsize 0.0025}
									&{\scriptsize 0.0034}
												&{\scriptsize 0.0004}
\end{tabular}}
\end{table}

In what follows, we keep considering the previously discussed first order discretization,
since we are interested in considering initial data with sharp transitions (as in the case
of Riemann problems). In such a case, higher order approximations of the derivatives 
introduce spurious oscillations that, even being transient, may leads to catastrophic
consequences because of the bistable nature of the reaction term.

As a consequence of its capability of correct computations of propagation speeds, we consider
the scheme \eqref{finalSch} to be a reliable tools for determining numerically the behavior
of the solutions to \eqref{relAC} and we use it to show that the actual domain of attraction of the
front is much larger than guaranteed by the nonlinear stability in Theorem \ref{thm:nonlinstab}.

\subsection{Riemann problem}
The rigorous result proved in the previous sections guarantees that small perturbations
to the propagating front are dissipated, with exponential rate, by the equation. 
Inspired by the many available results for the parabolic Allen--Cahn equation (starting from the 
landmark article \cite{FifeMcLe77}), we expect that the front possesses a very large domain of
attraction and, specifically, that any bounded initial data $u_0$ such that 
\begin{equation}\label{decay}
	\limsup_{x\to-\infty} u_0(x)<\alpha<\liminf_{x\to+\infty} u_0(x)
\end{equation}
gives raise to a solution that is asymptotically convergent to a member of the traveling
fronts connecting $u=0$ with $u=1$.

\begin{figure}[hbt]\centering
\resizebox{0.45\hsize}{!}{\includegraphics*{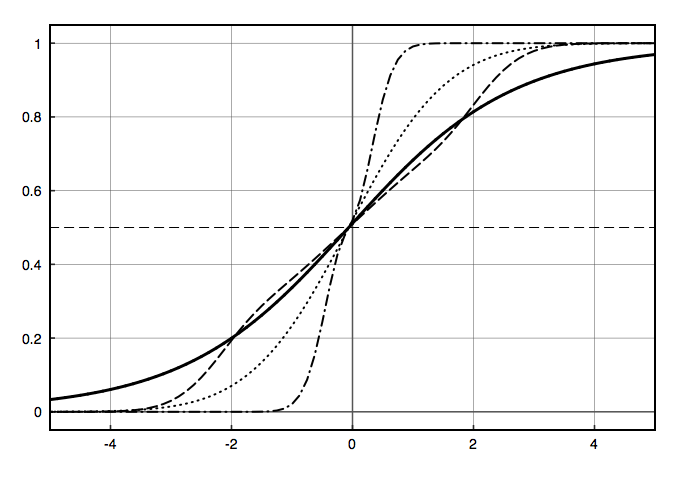}}
\resizebox{0.45\hsize}{!}{\includegraphics*{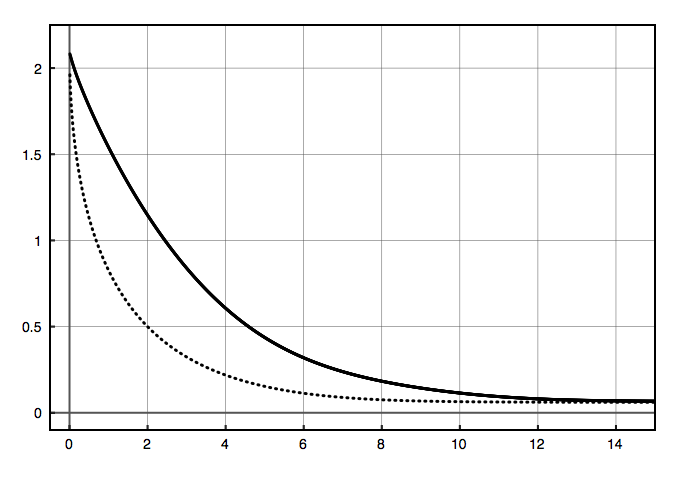}}
\caption{Riemann problem with initial datum $\chi_{{}_{(0,\ell)}}$ in $(-\ell,\ell)$, $\ell=25$.
Left: solution profiles zoomed in the interval $(-5,5)$ at time $t=1$ (dash-dot), $t=5$ (dash), $t=15$ (continuous),
for comparison, solution to the parabolic Allen--Cahn equation at time $t=1$ (dot).
Right: Decay of the $L^2$ distance to the exact equilibrium solution for the hyperbolic
(continuous) and parabolic (dot) Allen--Cahn equations.}
\label{fig:Riemann2}
\end{figure}

To support such conjecture, we perform some numerical experiments choosing
the parameters values
\begin{equation*}
	\tau=4,\qquad \ell=25,\qquad \dx=0.125,\qquad \dt=0.01.
\end{equation*}
Moreover, we consider the case $\alpha=1/2$ motivated by the fact that, in such a special case,
the profile of the traveling front for the hyperbolic Allen--Cahn equation is stationary and it coincides
with the one of the corrresponding original parabolic equation, explicitly given by
\begin{equation*}
	U(x)=\frac{1}{1+e^{-x/\sqrt{2}}},\qquad
	V(x)=\frac{dU}{dx}=\frac{1}{\sqrt{2}}\,\frac{1}{e^{x/\sqrt{2}}+2+e^{-x/\sqrt{2}}}
\end{equation*}
when normalized by the condition $U(0)=1/2$.

Numerical simulations confirm the decay of the solution to the equilibrium profile
(see Figure  \ref{fig:Riemann2}, left).
When compared with the standard Allen--Cahn equation, it appears evident that the
dissipation mechanism of the hyperbolic equation is weaker with respect to the 
parabolic case (see Figure  \ref{fig:Riemann2}, right).

\subsection{Randomly perturbed initial data}

Next, keeping all of the previous parameters fixed,
we consider initial data that resemble very roughly the transition from 0 to 1.
Namely, we divide the interval $(-\ell,\ell)$ into three parts and we choose a 
random value in each of these sub-intervals coherently with the requirement \eqref{decay}.
Precisely, we choose $u_0(x)$ to be a different random value in $(0,0.5)$ for each $x\in(-\ell,-\ell/3)$,
in $(0,1)$ for each $x\in(-\ell/3,\ell/3)$ and in $(0.5,1)$ for each $x\in(\ell/3,\ell)$.
\begin{figure}[hbt]\centering
\resizebox{0.8\hsize}{!}{\includegraphics*{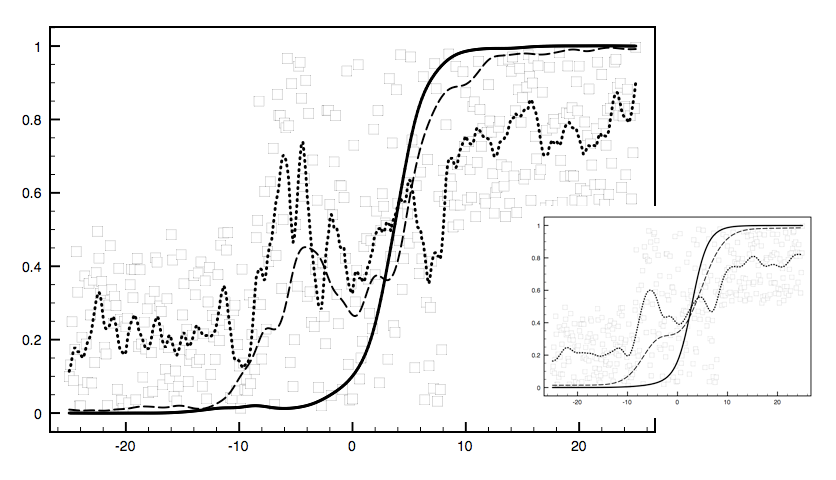}}
\caption{Random initial datum in $(-\ell,\ell)$, $\ell=25$ (squares).
Solution profiles for the hyperbolic Allen--Cahn equation with relaxation
at time $t=0.5$ (dot), $t=7.5$ (dash), $t=15$ (continuous).
For comparison, in the small window, the solution to the parabolic Allen--Cahn equation.
}\label{fig:Random1tot}
\end{figure}

The results for both hyperbolic and parabolic Allen--Cahn equation with the same initial
datum are illustrated in Fig.\ref{fig:Random1tot}.
Convergence to the equilibrium configuration is manifest.
It is also worthwhile to note that different level of smoothing produced by the 
presence/absence of the relaxation parameter $\tau$, measuring the ``hyperbolicity''
of the model.

The transition is even much more robust than what the previous computation shows.
With an initial datum $u_0(x)$ given by a random value in $(0,0.9)$ for each $x\in(-\ell,-\ell/3)$,
in $(0,1)$ for each $x\in(-\ell/3,\ell/3)$ and in $(0.1,1)$ for each $x\in(\ell/3,\ell)$, we still
observe the appearance and formation of the front, as shown in
Figure \ref{fig:Random2tot}.


\end{document}